\documentclass[reqno]{amsart}

\usepackage{graphicx,pdfsync,amssymb, latexsym,amsfonts,amsbsy,color,esint,mathtools,nicefrac,cancel,bbm,tikz, caption, subcaption}
\usepackage[alphabetic,nobysame]{amsrefs}
\usepackage[shortlabels]{enumitem}

\usepackage{parskip}

\usepackage{hyperref}
\hypersetup{
colorlinks=true,
linkcolor=blue,
filecolor=blue,      
urlcolor=blue,
citecolor=blue,
}
\newtheorem{theorem}{Theorem}[section]
\newtheorem{lemma}[theorem]{Lemma}
\newtheorem{proposition}[theorem]{Proposition}
\newtheorem{definition}[theorem]{Definition}

\newtheorem{remark}[theorem]{Remark}

\numberwithin{equation}{section}

\usepackage{times}

\let\div\relax
\DeclareMathOperator\div{div}

\DeclareMathOperator\Id{Id}
\DeclareMathOperator\supp{supp}

\newcommand\newD{\tilde{\mathcal{D}}}

\title[Instantaneous Type I blow-up for the Navier--Stokes]{Instantaneous Type I blow-up and non-uniqueness of smooth solutions of the Navier--Stokes equations}

\author [Alexey Cheskidov]{Alexey Cheskidov}

\address{Institute for Theoretical Sciences, Westlake University, China}
\email{cheskidov@westlake.edu.cn} 

\author [Mimi Dai]{Mimi Dai}

\address{Department of Mathematics, Statistics and Computer Science, University of Illinois at Chicago, Chicago, IL 60607, USA}
\email{mdai@uic.edu} 

\author [Stan Palasek]{Stan Palasek}

\address{Department of Mathematics, Princeton University \& Institute for Advanced Study, 1 Einstein Dr., Princeton, NJ 08540}
\email{palasek@ias.edu} 

\begin{document}

\begin{abstract}
For any smooth, divergence-free initial data, we construct a solution of the Navier--Stokes equations that exhibits Type I blow-up of the $L^\infty$ norm at time $T_*>0$, while remaining smooth in space and time on $\mathbb T^d\times([0,T]\setminus\{T_*\})$. An instantaneous injection of energy from infinite wavenumber initiates a bifurcation from the classical solution, producing an infinite family of spatially smooth solutions with the same data and thereby violating uniqueness of the Cauchy problem. A key ingredient is the first known construction of a complete inverse energy cascade realized by a classical Navier--Stokes flow, which transfers energy from infinitely high to low frequencies.  The result holds in all dimensions $d\geq2$.
\end{abstract}

\maketitle

\begingroup
\setlength{\parskip}{0pt}
\tableofcontents
\endgroup

\section{Introduction}

We consider the incompressible Navier--Stokes equations 
\begin{equation}\label{eq:NSE}
\begin{cases}
\partial_t u -  \nu \Delta u + \div( u \otimes u) + \nabla p = 0 &\\
\div u = 0
\end{cases}
\end{equation}
on $\mathbb T^d\times [0,T]$ for some $T>0$ and any spatial dimension $d\geq 2$, equipped with divergence-free initial data at $t=0$. The real parameter $\nu>0$ represents viscosity, which we typically normalize to $1$ without loss of generality. We say that a solution $(u,p)$ of \eqref{eq:NSE} is \emph{classical} if it is smooth on $\mathbb T^d\times[0,T]$ in space and time. The pioneering work of Leray~\cite{Leray1934} established local-in-time existence of classical solutions from any smooth initial data, as well as global existence of \emph{weak} solutions to~\eqref{eq:NSE} in any dimension $d\ge2$, for any finite-energy initial data.

\subsection{An inquiry into blow-up phenomena}
\label{sec:def-blowup}
A long-standing question is whether \eqref{eq:NSE} allows \emph{finite-time blow-up}, which refers to a scenario in which the unique classical solution from some initial data becomes singular in the sense that
\begin{equation}\label{eq:finite_time_blow_up_definition}
\lim_{t\to T_*-} \|u(t)\|_{L^\infty} = \infty.
\end{equation}
When $d=2$, finite-time blow-up was ruled out in the whole space by Leray~\cite{Leray1934} and in bounded domains by Ladyzhenskaya~\cite{Ladyzhenskaya1967}. For $d\geq3$, the question is open. Demonstration of a finite-time blow-up for $d=3$ would, by uniqueness, contradict global existence, resolving the Millennium Prize Problem on global regularity for the 3D Navier--Stokes equations in the negative direction.

The purpose of this paper is to introduce and prove existence of a related but, to our knowledge, previously unexplored blow-up phenomenon for the Navier--Stokes equations from smooth initial data. We say that a solution $u(x,t)$ exhibits \emph{smooth instantaneous blow-up} at time $T_*$ if the following conditions are met:
\begin{itemize}[left=1em]
\item $u$ is a classical solution of \eqref{eq:NSE} on $\mathbb T^d\times([0,T]\setminus\{T_*\})$,
\item $u$ is spatially smooth at the blow-up time (i.e., $u(T_*)\in C^\infty(\mathbb T^d)$) and remains a genuine weak solution (which holds automatically assuming suitable time continuity at $t=T_*$; see Lemma~\ref{le-weak}),
\item The solution blows up from the right in the sense that
\begin{align}\label{eq:instantaneous_blow_up_definition}
    \lim_{t\to T_*+} \|u(t)\|_{L^\infty} = \infty.
\end{align}
\end{itemize}
In the sequel, we show that a smooth instantaneous blow-up can appear for any smooth divergence-free initial data, and in any spatial dimension $d\geq2$. This behavior appears to be new for the Navier--Stokes equations and for parabolic equations in general.

An alternative perspective on instantaneous blow-up is as a finite-time blow-up that proceeds backwards in time, in the sense that norms of the solution become large as $t\to T_*+$. One might expect such a blow-up to occur more easily than the conventional forward-in-time blow-up because it would be facilitated by the Laplacian. We emphasize that this is not the case. Indeed, many of the critical regularity criteria (for instance, Ladyzhenskaya--Prodi--Serrin and Onsager-type ones) that are well-known to prevent finite-time blow-up also prevent instantaneous blow-up. In fact, the time asymmetry at the heart of the construction is not from the Laplacian, but the nonlinearity. Specifically, the key observation is that nonlocal interactions in frequency space are most robust from high to low frequencies forwards in time, or equivalently, from low to high frequencies backwards in time.

If a smooth solution of Navier--Stokes undergoes an instantaneous blow-up at $t=T_*$ as defined above, failure of uniqueness and even weak--strong uniqueness of the initial value problem directly follow. Indeed, because the solution is smooth at the blow-up time, one may apply standard local well-posedness theory with initial data $u(T_*)$ to obtain a classical solution $U$ that agrees with $u$ on $[0,T_*]$, but remains uniformly bounded in a neighborhood of the blow-up time. By \eqref{eq:instantaneous_blow_up_definition}, $u$ and $U$ must be distinct after $T_*$. In fact, in the examples constructed here, something much stronger is true: the classical solution $U$ is the limit of an infinite family of solutions that exhibit instantaneous blow-up at $t=T_*$ (see Figure~\ref{fig:Nonuniqueness}). The family consists of quasi-periodic solutions on the whole space that exhibit translation symmetry-breaking in the sense that the solutions spontaneously lose periodicity. A discrete subset of the family consists of genuinely periodic solutions with the same period as the background classical solution. The solutions we exhibit blow up in the Type~I sense (see Subsection \ref{subsec-type-I} for definitions and discussion) and lie arbitrarily close to known criteria for uniqueness and weak--strong uniqueness of Navier--Stokes weak solutions.

\subsection{Main results}\label{sec-results}

Our first main result is on the existence of instantaneous Type~I blow-up for the Navier--Stokes equations in $\mathbb T^d$ for general initial data.

\begin{theorem} \label{main-thm}
For any $d\geq2$ and $u_0 \in C^\infty(\mathbb{T}^d)$ divergence-free, there exists $T>0$ such that the following holds. For any $T_*\in[0,T)$, there exists a weak solution $u(t)$ of \eqref{eq:NSE} on $\mathbb T^d\times[0,T]$ such that:
\begin{enumerate}[1.,ref=\arabic*,left=1em]
    \item For $t\in[0,T_*]$, $u$ is a classical solution of \eqref{eq:NSE} with initial data $u(0)=u_0$,\label{blow_up_theorem_classical_1}
    \item For $t\in(T_*,T]$, $u$ is a classical solution of \eqref{eq:NSE} such that, for some $c>0$,
    \[
    \|u(t_n)\|_{L^\infty} \geq \frac{c}{\sqrt{t_n-T_*}}, \qquad  \|\nabla\wedge u(t_n)\|_{L^\infty} \geq \frac{c}{t_n-T_*},
    \]
    along a sequence of times $t_n \to T_*+$,\label{blow_up_theorem_classical_2_with_lower_bound}
    \item The solution $u(t)$ is weak-* continuous in time with values in $BMO^{-1}$,\label{blow_up_theorem_continuity}
    \item The blow-up is of Type~I:
    \[
    \|u(t)\|_{L^\infty} \leq \frac{C}{\sqrt{t-T_*}}, \qquad  \|\nabla u(t)\|_{L^\infty} \leq \frac{C}{t-T_*},
    \]
    for all $t\in(T_*,T]$ and some constant $C>0$. When $d\geq3$, we additionally have $u\in L_t^\infty \dot W^{-1,\infty}_x$,\label{blow_up_theorem_type_i}
    \item The solution obeys the logarithmically-weakened Ladyzhenskaya--Prodi--Serrin bound
\begin{equation}\label{logarithmically_weakened_prodi_serrin}
\int_0^T\frac{\|u(t)\|^2_\infty}{1+(\log\log(e+\|u(t)\|_\infty))^c}dt<\infty,
\end{equation}
and Beale--Kato--Majda condition
\begin{equation}\label{logarithmically_weakened_BKM}
\int_0^T\frac{\|\nabla u(t)\|_\infty}{1+(\log\log(e+\|\nabla u(t)\|_\infty))^c}dt<\infty,
\end{equation}
for some $c>0$. Moreover,
\begin{equation}
    u\in L^2([0,T];L^p),
\end{equation}
for all $p<\infty$.\label{blow_up_theorem_prodi_serrin}
\end{enumerate}
\end{theorem}

\begin{remark}
    We emphasize that, by definition, classical solutions are smooth in space and time. In particular, combining Properties~\ref{blow_up_theorem_classical_1}--\ref{blow_up_theorem_classical_2_with_lower_bound}, the solution $u(t)\in C^\infty(\mathbb T^d)$ for every $t\in[0,T]$, including at the blow-up time. In fact, such solutions are analytic in space and time away from $t=T_*$; see \cites{MR245989,MR870865}. See Figure~\ref{fig:blowup-from-the-right} for a schematic depiction of the regularity of the solution.
\end{remark}

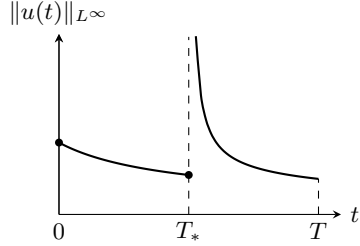
\begin{figure}[ht]
\centering
\begin{tikzpicture}[scale=0.75,>=stealth]
  \draw[->] (0,0) -- (5,0) node[right] {\small $t$};
  \draw[->] (0,0) -- (0,3.2) node[above] {\small $\|u(t)\|_{L^\infty}$};
  \draw[->] (0,0) -- (0,3.2);
  \draw[dashed] (2.3,0) -- (2.3,3.15);
  \draw[dashed] (4.6,0) -- (4.6,0.6305);
  \draw[thick,domain=2.426:4.6,smooth,variable=\x] plot ({\x},{1/(3.2*(\x-2.3))+1.65/(\x+1)+0.2});
  \fill (2.3,0.7) circle (2pt);
  \fill (0,1.275) circle (2pt);
  \draw[thick,domain=0:2.3,smooth,variable=\x] plot ({\x},{2.15/(\x+2)+0.2});
  \node at (0,-0.3) {\small $0$};
  \node at (2.3,-0.3) {\small $T_*$};
  \node at (4.6,-0.3) {\small $T$};
\end{tikzpicture}
\caption{Instantaneous blow-up at $t=T_*$.}
\label{fig:blowup-from-the-right}
\end{figure}

\begin{remark} \label{Orlicz-type_estimate}
Properties~\ref{blow_up_theorem_type_i}--\ref{blow_up_theorem_prodi_serrin} establish that the solution is at the borderline of the known regularity criterion $L^2_tL^\infty_x$; see Section \ref{subsec-type-I} for further discussion. We remark that by a straightforward modification of the proofs, the Orlicz-type estimate \eqref{logarithmically_weakened_prodi_serrin} can be strengthened to
\[
\int_0^T\frac{\|u(t)\|^2_\infty}{1+f(\|u(t)\|_\infty)}dt<\infty
\]
for \emph{any} desired $f:[0,\infty)\to[0,\infty)$ satisfying $f(x)\to+\infty$ as $x\to +\infty$. Indeed, the construction involves an increasing sequence of frequency scales $N_k$ indexed by $k\in\mathbb N$; $f$ is allowed to grow roughly as the inverse function of $k\mapsto  N_k$. For concreteness, the proof is given with $N_k$ double-exponential, but it can be executed just as well with any sufficiently fast growth rate.

It is interesting to observe that forward self-similar solutions (discussed in Section \ref{subsub-self-similar}) saturate the Type~I bound at \emph{every} time and therefore cannot possibly obey slightly supercritical spacetime estimates of this form.
\end{remark}

\begin{remark}
    The solution $u$ lies in neither $L_t^\infty L_x^2$ nor $L_t^2H_x^1$ near the blow-up time. We emphasize that an instantaneous blow-up of the energy is not a limitation of the construction, but in fact a \emph{necessary} condition for instantaneous Type~I blow-up of the $L^\infty$ norm; see Theorem~\ref{thm:energy_blow-up}. 
\end{remark}

\begin{remark}
With $T^{u_0}\in(0,\infty]$, the maximum time of existence for the classical solution with initial data $u_0$, the time $T$ in Theorem~\ref{main-thm} can be taken as any $T<T^{u_0}$. One can apply the theorem repeatedly to construct $u$ on the full lifespan $[0,T^{u_0})$, although one loses uniform estimates in the case $T^{u_0}<\infty$.
\end{remark}

\begin{remark}
    The type of blow-up demonstrated in Theorem~\ref{main-thm} does not appear to have been identified before for the Navier--Stokes equations or other fluid models. As the closest available precedent, we point to the instantaneous loss of regularity phenomenon appearing in work of C\'ordoba, Mart\'inez-Zoroa, and O\.za\'nski for the 2D Euler equations~\cite{MR4776418} and the fractionally dissipative SQG equations~\cite{MR4800679}. They construct solutions with initially non-smooth data, say in $H^{1+\epsilon}$, which immediately drop down to lower regularity $H^s$ for some $s<1+\epsilon$. Also, there are examples of self-similar non-unique solutions to semilinear equations \cite{HW1982}. See also Section~\ref{sec:inverse_cascade} for another comparison.
\end{remark}

As a consequence of the construction leading to Theorem~\ref{main-thm}, we prove that classical solutions of the Navier--Stokes equations can spontaneously lose uniqueness by an infinite energy cascade from high frequencies, while remaining $C^\infty(\mathbb T^d)$ at each time; see Figure~\ref{fig:Nonuniqueness}.

\begin{figure}[ht]
\centering
\begin{tikzpicture}[scale=0.75,>=stealth]
  \draw[->] (0,0) -- (4.6,0) node[right] {\small $t$};
  \draw[->] (0,0) -- (0,3.2) node[above] {\small $\|u(t)\|_{L^\infty}$};
  \draw[->] (0,0) -- (0,3.2);
  \draw[dashed] (2.3,0) -- (2.3,3.15);
  \draw[domain=2.3203826:4.5,smooth,variable=\x, samples=500] plot ({\x},{1/(20*(\x-2.3))+1.65/(\x+1)+0.2});
  \draw[domain=2.3812344:4.5,smooth,variable=\x, samples=300] plot ({\x},{1/(5*(\x-2.3))+1.65/(\x+1)+0.2});
  \draw[domain=2.5017104:4.5,smooth,variable=\x] plot ({\x},{1/(2*(\x-2.3))+1.65/(\x+1)+0.2});
  \draw[domain=2.6993646:4.5,smooth,variable=\x] plot ({\x},{1/(1*(\x-2.3))+1.65/(\x+1)+0.2});
  \draw[domain=0:4.5,smooth,variable=\x] plot ({\x},{1.65/(\x+1)+0.2});
  \node at (2.3,-0.3) {\small $T_*$};  
  \fill (2.3,0.7) circle (2pt);
  \fill (0,1.85) circle (2pt);
  \draw[thick,domain=0:2.3,smooth,variable=\x] plot ({\x},{1.65/(\x+1)+0.2});
  \node at (0,-0.3) {\small $0$};
  \node at (2.3,-0.3) {\small $T_*$};
  \node at (1.1,.56) {\small $U$};
  \node at (3.4,.335) {\footnotesize $u^{(0)}$};
  \node at (3.8,1.9) {\footnotesize $u^{(1)}$};
\end{tikzpicture}
\caption{\begin{tabular}[t]{ @{} l @{} }Continuous family of solutions $u^{(\sigma)}$\\ with the same initial data.\end{tabular}}
\label{fig:Nonuniqueness}
\end{figure}
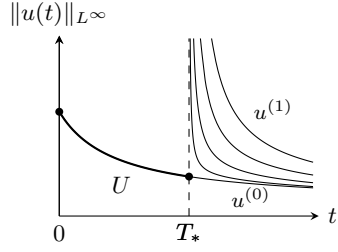

In Theorem~\ref{thm-non-unique} below, we let $U$ be any classical solution of the Navier--Stokes equations on $\mathbb T^d\times [0,T]$ for some $d\geq 2$ and $T\in(0,1]$. We will interchangeably consider $U$ as either a vector field on $\mathbb T^d$ or as a periodic solution on all of $\mathbb R^d$. 

\begin{theorem}\label{thm-non-unique}
     For any $T_*\in[0,T)$, there exists a family $(u^{(\sigma)})_{\sigma\in[0,1]}$ of weak solutions of \eqref{eq:NSE} on $\mathbb R^d\times[0,T]$ which connects the classical solution $u^{(0)}\equiv U$ to a continuum of distinct solutions,
     \[u^{(\sigma)}\big|_{[0,T_*]}\equiv U\big|_{[0,T_*]},\qquad u^{(\sigma)}\big|_{(T_*,T]}\not\equiv U\big|_{(T_*,T]},\qquad\forall \sigma\in(0,1].\]
     Additionally, $u^{(\sigma)}$ possess the regularity properties
    \[u^{(\sigma)}(t)\in C^\infty(\mathbb R^d),\qquad\forall t\in[0,T],\]
    and, for all $p<\infty$, we have
    \[
    u^{(\sigma)}\in L_t^{2,\infty}L_x^\infty\cap L_{t}^2L_{x,\mathrm{loc}}^p.
    \]
    For $d\geq3$, we additionally have the critical control
    \begin{align*}
        u^{(\sigma)}\in L_t^\infty \dot W^{-1,\infty}_x.
    \end{align*}
    Along a discrete sequence $\sigma_n\to0+$, the solutions $u^{(\sigma_n)}$ are $2\pi$-periodic and converge in weak-* $BMO^{-1}(\mathbb T^d)$ to $U$.
\end{theorem}

\begin{remark}
    The family of solutions exhibits spontaneous breaking of the periodicity. Indeed, the solutions on $[0,T_*]$ agree with $U$ and are therefore $2\pi$-periodic. At $t=T_*$, the energy entering at infinite wavenumber is localized at a lattice of frequencies that may be an irrational multiple of $2\pi$. In that case, the solutions $u^{(\sigma)}$ have features of being quasi-periodic on $\mathbb R^d$ in the sense that they are periodic to leading order with a small, low frequency aperiodic perturbation coming from $U$. As $\sigma\to0+$, the period of the ``non-unique'' part of the solution goes to zero, and the solutions homogenize to $U$. The distinguished set $\{\sigma_n\}$ are those for which the period of the anomalous part of the solution is an integer. More generally, for the dense countable subset of $\sigma\in[0,1]$ for which the period is rational, the solution is periodic but with a potentially smaller period than $2\pi$.
    
    It is very likely that the proof can be adapted so that the solutions instead decay on $\mathbb R^d$, or remain periodic and well-defined on $\mathbb T^d$, but we do not pursue these extensions here.
\end{remark}

\begin{remark}
    More strongly, the proof will show that $u^{(\sigma)}$ is in the Koch--Tataru path space $X_{[T_*,T]}$, and therefore weak-* continuous in time with values in $BMO^{-1}$. In particular, the solutions are in $L_t^\infty BMO_x^{-1}(\mathbb R^d)$ for all $\sigma\in[0,1]$.
\end{remark}

\begin{remark}
    We make the following technical observation: because the $u^{(\sigma)}$ are in $L_{t,x}^\infty$ on $\mathbb R^d\times[T',T]$ for any $T'>T_*$, there is no difficulty formulating the notions of weak and classical solutions. Indeed, in the framework of bounded mild solutions, the standard local theory carries over from $\mathbb T^d$ or decaying solutions on $\mathbb R^d$; see for instance~\cite{MR2545826}.
\end{remark}

\noindent\textbf{Structure of the paper.}
We further explain the significance of the main results and highlight the main ideas of the construction in Section \ref{sec:significance}.
Section \ref{sec:prelim} introduces the notations to be used in the paper. Sections \ref{sec:principal}, \ref{sec:est} and \ref{sec:corrector} contain the technical construction and estimates. Section \ref{sec:proof} is devoted to the proof of the main results. Auxiliary lemmas are collected in the Appendix.

\medskip

\noindent\textbf{Acknowledgments.}
M.D. is partially supported by the National Science Foundation grant DMS-2308208 and the Simons Foundation. S.P. acknowledges support from the National Science Foundation under Grant No.\ DMS-2424441.

\section{Significance of the results and main ideas}
\label{sec:significance}
In this section, we expound on the significance of the blow-up constructed in Theorem~\ref{main-thm} and non-uniqueness established in Theorem~\ref{thm-non-unique}, and their relevance to the broader question of well-posedness of the Navier--Stokes equations.

\subsection{Finite-time vs.\ instantaneous blow-up}\label{s:blow-up_left_vs_right}

 Both types of smooth blow-up as described in Section \ref{sec:def-blowup} require a cascade of energy or some higher regularity norm across infinitely many frequency scales. They are differentiated by the direction of the cascade: whether the energy is transferred from low- to high-frequency modes or vice versa characterizes the blow-up type as finite-time or instantaneous, respectively. While finite-time blow-up has received far more attention, both phenomena are relevant for understanding the global behavior of Navier--Stokes solutions, and can conceivably be exhibited by Leray--Hopf solutions.

Indeed, suppose a classical solution $u$ exhibits finite-time blow-up at some $T_*>0$. Then there exists at least one global continuation of $u$ in the class of Leray--Hopf solutions. At $t=T_*$, there are two cases: the solution either loses regularity (i.e., $u(T_*)\notin C^\infty(\mathbb T^d)$), or becomes spatially smooth due to the high-frequency components of the solution escaping to infinite wavenumber. In the latter case where $u(T_*) \in C^\infty$, a second smooth blow-up might occur, even while $u$ remains in the Leray--Hopf class. This second blow-up is an ``instantaneous blow-up'' in the sense described above. The three alternatives are depicted in Figure~\ref{fig:blowup-scenarios}.

\begin{figure}[ht]
\centering
\hspace{-2em}
\begin{subfigure}[b]{0.32\linewidth}
\begin{tikzpicture}[scale=0.75,>=stealth]
  \draw[->] (0,0) -- (4.6,0) node[right] {\small $t$};
  \draw[->] (0,0) -- (0,3.2) node[above] {\small $\|u(t)\|_{L^\infty}$};
  \draw[dashed] (2.3,0) -- (2.3,3.15);
  \draw[thick,domain=0:2.2,smooth,variable=\x] plot ({\x},{1/(3.2*(2.3-\x))});
  \draw[thick,domain=2.4:4.5,smooth,variable=\x] plot ({\x},{1/(3.2*(\x-2.3))});
  \node at (2.3,-0.3) {\small $T_*$};
\end{tikzpicture}
\caption{\begin{tabular}[t]{ @{} l @{} } Formation of a\\ spatial singularity.\end{tabular}}
\end{subfigure}
\begin{subfigure}[b]{0.32\linewidth}
\begin{tikzpicture}[scale=0.75,>=stealth]
  \draw[->] (0,0) -- (4.6,0) node[right] {\small $t$};
  \draw[->] (0,0) -- (0,3.2) node[above] {\small $\|u(t)\|_{L^\infty}$};
  \draw[->] (0,0) -- (0,3.2);
  \draw[dashed] (2.3,0) -- (2.3,3.15);
  \draw[thick,domain=0:2.2,smooth,variable=\x] plot ({\x},{1/(3.2*(2.3-\x))});
  \draw[thick,domain=2.4:4.5,smooth,variable=\x] plot ({\x},{1/(3.2*(\x-2.3))});
  \fill (2.3,0.7) circle (2pt);
  \node at (2.3,-0.3) {\small $T_*$};
\end{tikzpicture}
\caption{\begin{tabular}[t]{ @{} l @{} }No spatial singularity;\\ blow-up from the right.\end{tabular}}
\end{subfigure}
\begin{subfigure}[b]{0.32\linewidth}
\begin{tikzpicture}[scale=0.75,>=stealth]
  \draw[->] (0,0) -- (4.6,0) node[right] {\small $t$};
  \draw[->] (0,0) -- (0,3.2) node[above] {\small $\|u(t)\|_{L^\infty}$};
  \draw[->] (0,0) -- (0,3.2);
  \draw[dashed] (2.3,0) -- (2.3,3.15);
  \draw[thick,domain=0:2.2,smooth,variable=\x] plot ({\x},{1/(3.2*(2.3-\x))});
  \fill (2.3,0.7) circle (2pt);
  \draw[thick,domain=2.3:4.5,smooth,variable=\x] plot ({\x},{0.91/(\x-1)});
  \node at (2.3,-0.3) {\small $T_*$};
\end{tikzpicture}
\caption{\begin{tabular}[t]{ @{} l @{} }No spatial singularity;\\ no blow-up from the right.\end{tabular}}
\end{subfigure}
\caption{Schematic depiction of possible scenarios following a finite-time blow-up:
(A) Formation of a spatial singularity at $t=T_*$; 
(B) Absence of a spatial singularity at $t=T_*$ with instantaneous blow-up from the right;
(C) Absence of spatial singularity at $t=T_*$ with a classical solution starting from $t=T_*$.}
\label{fig:blowup-scenarios}
\end{figure}
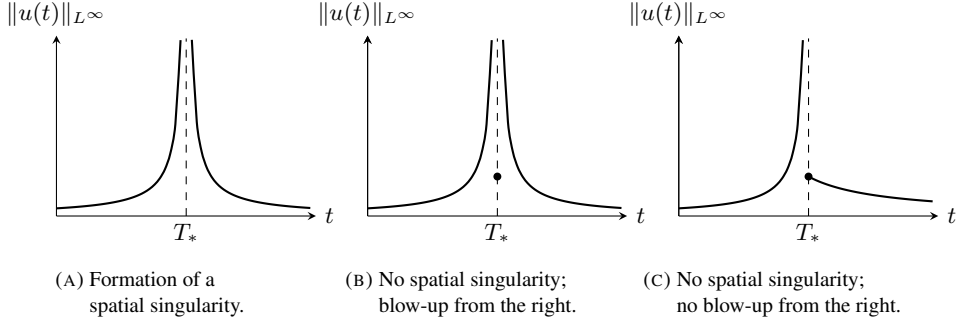

The scenario (B) in Figure~\ref{fig:blowup-scenarios}, in which both finite-time and instantaneous blow-up occur at $T_*$, is consistent with the Leray--Hopf energy inequality as long as the energy anomalously lost in the first blow-up exceeds the energy gained in the second. Even more strongly\footnote{It is presently unknown whether every Leray--Hopf solution is a restricted Leray--Hopf solution.}, it is consistent with the known behavior of a \emph{restricted} Leray--Hopf solution, i.e., the weak limit of a sequence of solutions of Galerkin-truncated Navier--Stokes, as it undergoes a potential blow-up. This highlights the fact that an instantaneous blow-up can occur in the Leray--Hopf class only \emph{after} the appearance of a finite-time blow-up. In contrast, for general weak solutions, an instantaneous blow-up can occur from an arbitrary classical solution, as demonstrated in Theorem \ref{main-thm}. 
While there are numerous regularity conditions precluding a finite-time blow-up, the analogous problem for instantaneous blow-up does not appear to have been considered. 

We emphasize that instantaneous blow-up is a fundamentally nonlinear phenomenon. For instance, suppose $u=e^{(t-T_*)\Delta}\phi$ solves the heat equation on $\mathbb T^d$ from initial data $u|_{t=T_*} =\phi\in \dot B^{-1}_{\infty,\infty}$. While $u$ does admit the Type~I bound
\[
\sqrt{t-T_*}\|u(t)\|_{L^\infty}\lesssim \|\phi\|_{\dot B^{-1}_{\infty,\infty}},
\]
if $\limsup_{t\to T_*+}\|u(t)\|_{L^\infty}=\infty$, then necessarily  the initial data is singular. An instantaneous blow-up arising from smooth initial data can occur only through a nonlinear mechanism that transfers energy across infinitely many frequency modes. 

When discussing various mechanisms of breakdown of classical solutions, we are careful to distinguish \emph{blow-up}, corresponding to unboundedness of $\|u(t)\|_{L^\infty}$ (e.g., \eqref{eq:finite_time_blow_up_definition}--\eqref{eq:instantaneous_blow_up_definition}) and \emph{spatial singularity formation}, corresponding to a loss of smoothness in the variable $x$ at a fixed time. The former refers to the classical scenario envisioned by Leray, where a smooth solution ceases to exist as $t \to T_*-$. If such a blow-up indeed occurs, the weak continuation of the solution may or may not develop a spatial singularity at $t=T_*$; see Figure~\ref{fig:blowup-scenarios}. Formation of a singularity in space, therefore, does not necessarily follow from a blow-up of a subcritical norm such as $\|u(t)\|_{L^\infty}$. However, for weakly continuous solutions, a formation of singularity in space \emph{requires} a blow-up from the left as in \eqref{eq:finite_time_blow_up_definition}.

Finally, the instantaneous blow-up and non-uniqueness constructed in this paper involve a purely \emph{temporal} loss of smoothness, where the solutions may evolve from any smooth initial data, while staying analytic in space for all $t>0$ and in time for all $t\neq T_*$. It therefore represents a fundamentally different manifestation of non-uniqueness---one that occurs entirely within the class of smooth, divergence-free vector fields (see Section \ref{subsec-nonunique} for further discussion.)

\subsection{Criticality and Type~I blow-up}
\label{subsec-type-I}

When posed on $\mathbb R^d$, the Navier--Stokes equations possess a scaling symmetry under the transformation
\[
u_\lambda(t,x)=\lambda\,u(\lambda x, \lambda^2t),
\qquad
p_\lambda(t,x)=\lambda^2 p(\lambda x, \lambda^2t),
\]
which leaves the equations invariant for any $\lambda>0$.  
A function space whose norm is invariant under this scaling is called \emph{critical}. For instance, $L_t^\infty X$ is critical in dimension $d\geq2$ if $X$ belongs to the following sequence of critical spaces:
\begin{equation}\label{eq:critical_spaces}\dot H^{\frac d2-1}\subset L^d\subset L^{d,\infty}\subset \dot B^{-1+\frac{d}{p}}_{p,\infty}\subset BMO^{-1} \subset \dot B^{-1}_{\infty,\infty},\end{equation}
for any $d<p<\infty$. Another important family of critical spaces is given by the Ladyzhenskaya--Prodi--Serrin scale, namely $L^p_tL^q_x$ where $q>d$ and $\frac2p + \frac{d}{q}=1$, with scale-invariance enforced by the latter condition.

A basic consequence of Leray's local theory is that $\|u(t)\|_{L^\infty}$ diverges at least as fast as $(T_*-t)^{-\frac12}$ before a finite-time blow-up at $t=T_*$. A blow-up that saturates this critical lower bound is referred to as \emph{Type~I}. More precisely, a classical solution is said to exhibit a \emph{finite-time Type~I blow-up} at $t=T_*$ if it obeys both \eqref{eq:finite_time_blow_up_definition} and the scale-invariant bound
\[
\|u(t)\|_{L^\infty}\leq \frac{C}{\sqrt{T_*-t}}, \quad t<T_*,
\]
for some constant $C>0$. Analogously, an \emph{instantaneous Type~I blow-up} occurs at $t=T_*$ if, in addition to \eqref{eq:instantaneous_blow_up_definition} and $u(T_*)\in C^\infty$, the solution satisfies
\[
\|u(t)\|_{L^\infty}\leq \frac{C}{\sqrt{t-T_*}}, \quad t >T_*,
\]
for some $C>0$. Observe that, near an isolated blow-up time, the Type~I bounds are essentially equivalent to the inclusion $u\in L_t^{2,\infty}L_x^\infty$, which should be compared to the Ladyzhenskaya--Prodi--Serrin space $L_t^2L_x^\infty$. The Type~I rate corresponds to the critical scaling of the heat kernel and marks
the threshold where the parabolic regularization is in balance with potential small scale creation by the nonlinearity.

The smooth instantaneous blow-up examples we construct toward Theorem~\ref{main-thm} are indeed of Type~I (Property~\ref{blow_up_theorem_type_i} in the theorem), and the bound is sharp near $T_*$ in the sense that
\begin{equation}\label{eq:type_I_lower_bound}
\liminf_{t\to T_*+}\,\sqrt{t-T_*}\|u(t)\|_{L^\infty} \geq c,
\end{equation}
for some $c>0$ by Property~\ref{blow_up_theorem_classical_2_with_lower_bound}. Let us also point out that while there is not a single agreed upon definition of Type~I blow-up, the definition given above appears to be the traditional one, first appearing in the context of the Ricci flow~\cite{MR1375255} and later for other types of parabolic and dispersive equations (\cites{MR2077706,MR2545826,MR3272053}, etc.). In the recent fluids literature, it is common to refer to any bound obeyed by self-similar solutions as Type~I, particularly those of the form $u\in L_t^\infty X$ where $X$ is in $L^{d,\infty}$ or any of the spaces in which it is embedded in \eqref{eq:critical_spaces}. Under this definition, the instantaneous blow-up constructed here is still of Type~I, in that the solution is in $L_t^\infty\dot W_x^{-1,\infty}$ or $L^\infty_t BMO^{-1}_x$.

For comparison, the existence of \emph{finite-time} Type~I blow-up for the Navier--Stokes equations with $d\geq3$ is a major open question, although it has been ruled out in special cases including backward self-similar solutions~\cites{MR1397564,MR1643650}, axisymmetric solutions~\cites{MR2545826,MR2429247,MR2512859,MR4628936}, and when the critical norm is strengthened to, say, $L^d$~\cite{Escauriaza}.

\subsection{Previous non-uniqueness constructions}
\label{subsec-nonunique}
For finite-energy initial data, Leray’s construction yields
the existence of weak solutions satisfying the energy inequality (Leray--Hopf solutions), and possessing additional desirable properties, such as Leray's structure theorem and the weak--strong uniqueness. In the classical setting, uniqueness holds under certain regularity or integrability conditions. Most notably, strong solutions, that is, solutions satisfying the Ladyzhenskaya--Prodi--Serrin condition
\[
u \in L^p_t L^q_x, \qquad \frac{2}{p}+\frac{d}{q}=1, \quad q>d,
\]
are unique in the class of Leray--Hopf solutions. In these critical regimes, the nonlinearity $u\cdot\nabla u$ is controlled by the linear smoothing effect of $\Delta u$, ensuring well-posedness. In the absence of such conditions, the situation becomes far more delicate. Whether weak solutions can fail to be unique has long been regarded as a major open problem.

\subsubsection{Convex integration}
\label{subsub-convex-integration}

The Euler equations, by which we mean \eqref{eq:NSE} with $\nu=0$, offer insight into the uniqueness problem. Following the pioneering works of Scheffer~\cite{Scheffer1993} and Shnirelman~\cites{Shnirelman1997, Shnirelman2000}, the convex integration framework brought to mathematical fluid mechanics by De~Lellis and Székelyhidi~\cites{DeLellisSzekelyhidi2009,DeLellisSzekelyhidi2013} provided a new powerful tool for constructing weak solutions of PDEs with various pathological properties motivated by physical theories of turbulence. This program led to Isett’s construction of non-conservative $C_{x,t}^{1/3-}$ solutions for the Euler equations in 3D \cite{Isett2018}, later extended to 2D by Giri and Radu in \cite{MR4809443}. Together with the proof of energy conservation above the $\frac13$-H\"older threshold~\cite{MR1298949}, these works are regarded to have resolved Onsager’s conjecture on anomalous energy dissipation.

These developments exposed a profound breakdown of uniqueness and the classical notion of determinism in fluid mechanics. Still, the weak solutions of the Euler equations arising from convex integration are not known to be vanishing-viscosity limits of Navier--Stokes solutions in any reasonable regularity class, raising the question of which ``wild'' phenomena can be transferred to the viscous setting. Indeed, the Laplacian term in~\eqref{eq:NSE} provides a smoothing mechanism that might have the potential to preclude the wild oscillatory behavior observed in the Euler case. A breakthrough in this direction came in the work of Buckmaster and Vicol~\cite{BuckmasterVicol2019}, who introduced a viscous adaptation of convex integration and constructed the first examples of non-unique weak solutions of~\eqref{eq:NSE}. Their solutions lie in $H^\epsilon(\mathbb T^3)$ for some $\epsilon>0$ and exhibit strong non-uniqueness and non-conservation of the $L^2$ norm. Subsequent constructions of wild solutions with improvements to the regularity were given, for instance, in \cites{MR4422213,2009.06596,MR4610908}. These results demonstrate that viscosity does not, in general, restore uniqueness. The recent constructions in \cites{CZZ2025,CoiculescuPalasek2025} push this point further, as the Laplacian term is essential to the emergence of non-uniqueness, and the constructions fail in the inviscid case.

\subsubsection{Self-similar solutions}
\label{subsub-self-similar}
A different mechanism for non-uniqueness of Navier--Stokes solutions was envisioned by Jia and \v{S}ver\'ak in~\cites{MR3179576,MR3341963}. Their program concerns forward self-similar solutions of \eqref{eq:NSE} which evolve from $(-1)$-homogeneous initial data. Such solutions are necessarily singular at the origin but lie in certain sharp regularity classes. More explicitly, one searches for solutions on $\mathbb R^d\times(T_*,\infty)$ of the form
\[
u(x,t)=\frac{1}{\sqrt{t-T_*}}\,U\!\left(\frac{x}{\sqrt{t-T_*}}\right),\qquad
p(x,t)=\frac{1}{t-T_*}\,P\!\left(\frac{x}{\sqrt{t-T_*}}\right),
\]
where $(U,P)$ satisfies the stationary Leray profile equation
\begin{equation}\label{eq:Leray}
-\Delta U-\tfrac12 x\cdot\nabla U-\tfrac12 U+U\cdot\nabla U+\nabla P=0,
\qquad \div U=0.
\end{equation}
If there exist two distinct profiles $U_1,U_2$ with the same far-field asymptotics,
then the corresponding self-similar flows yield two different smooth solutions of the Navier--Stokes equations arising from the same $(-1)$-homogeneous initial data at $t=T_*$. Similarly, if there exists a profile $\overline U$ which is nonlinearly unstable in the time-dependent version of \eqref{eq:Leray}, then the unstable manifold is a family of solutions converging to $\overline U$ at the initial time $T_*$. This program was carried out for the forced axisymmetric Navier--Stokes equations without swirl by Albritton, Bru\'e, and Colombo~\cite{abc}, following construction of an unstable vortex for forced 2D Euler by Vishik~\cites{vishik1,vishik2}. For the unforced Navier--Stokes equations, unstable solutions of \eqref{eq:Leray} were discovered numerically by Guillod and \v Sver\'ak~\cite{MR4593274}, and very recently a computer-assisted proof was announced by Hou, Wang, and Yang \cite{hou2025nonuniqueness}.

However, this self-similar mechanism is fundamentally different from the blow-up from the right constructed in the present work. In the self-similar case, the singularity occurs through a fixed scaling profile evolving from singular initial data. In contrast, our blow-up from the right arises from \emph{smooth} initial data, yet produces a Type~I divergence of the $L^\infty$ norm immediately after time $T_*$. Unlike the self-similar case, no singular structure is imposed at time $T_*$; the blow-up is generated dynamically by an \emph{inverse energy cascade} that transfers energy from infinitely high frequencies.

In physical terms, both mechanisms rely on an \emph{inverse energy cascade}, but in the self-similar case, specifically designed singular initial data supplies all the energy for the transfer to low scales. In our case, the initial data can be zero, so all the energy comes from the infinite wavenumber, not the initial data.

\subsubsection{Non-uniqueness based on dyadic scenarios}
\label{subsub-dyadic}

In \cite{CoiculescuPalasek2025}, Coiculescu and Palasek introduced another mechanism for non-uniqueness, based on the construction in \cite{palasek2024non} for the Obukhov dyadic model. The initial data lies at the critical regularity as it does for self-similar solutions, but is not precisely $-1$-homogeneous. This scenario is fundamentally different from the two above, as there is no energy cascade. Instead, it is based on an asymmetric ``self-distraction'' where a binary decision is made at $t=0$ at infinite wavenumber, and that information is propagated to low frequencies. The singular initial data is designed so that either the even frequency modes of the initial data annihilate the odd ones, or vice versa, resulting in two distinct solutions. Compare also to the recent construction~\cite{palasek2025arbitrary} of norm growth from bounded data which involves a finite inverse cascade, but with energy supplied by the large initial velocity rather than appearing anomalously from infinite frequencies.

\subsection{Sharpness of the non-uniqueness theorem}
\label{sec-sharp-nonunique}
To show that the constructed non-unique solution $u(t)$ in Theorem~\ref{thm-non-unique} lies on the borderline of known uniqueness criteria, we recall several classical results.

A weak solution is defined in Section~\ref{sec:weak_solutions} as a vector field $u \in L^2 (\mathbb{T}^d \times [0,T])$ satisfying the Navier--Stokes equations in a suitable averaged sense. In \cite{MR316915}, Fabes, Jones, and Rivi\`ere proved that weak solutions also satisfy the following integral equation, known as the mild formulation,
\[
u = e^{t\Delta}u_0 - \int_0^t e^{(t-s)\Delta} \mathbb{P}\div(u\otimes u) (s)\, ds,
\]
which can be formally obtained from \eqref{eq:NSE} using Duhamel's principle. 

The first unconditional weak--strong uniqueness theorem was also obtained in \cite{MR316915}, where it was proved that weak solutions are unique under the Ladyzhenskaya--Prodi--Serrin condition away from the endpoints, i.e., in the class $ L^p_tL^q_x$ with $ \frac{2}{p} + \frac{d}{q} \leq 1$ and $d < q  <\infty  $. Later, uniqueness at the endpoint $C_tL^d$ was established in \cite{MR1813331} for $d  = 3$, and subsequently extended in \cites{MR1724946,MR1680809,MR1876415} to various spatial domains using different approaches. Recently, in \cite{2009.06596} it was shown that the method of \cite{MR1876415} can be extended to cover the whole Ladyzhenskaya--Prodi--Serrin line including the other endpoint $L^2_tL^\infty_x$:

\begin{theorem}
Let $U$ be a classical solution of \eqref{eq:NSE} on $\mathbb T^d\times [T_*,T]$, $d\geq 2$. If $u$ is a weak solution of \eqref{eq:NSE} with $u(T_*)=U(T_*)$ such that 
\[
\displaystyle u\in L_t^{2}L_x^\infty([T_*,T]\times\mathbb T^d),
\]
then $u\equiv U$ for all $t\in [T_*,T]$.
\end{theorem}
The non-unique solutions constructed in this paper lie at the borderline of $L^2_tL^\infty_x$, as they satisfy the Type~I bound and even Orlicz-type estimates as discussed in Remark~\ref{Orlicz-type_estimate}.

Another relevant uniqueness condition was proved in 2001 by Koch and Tataru (see Section~\ref{sec:KT} for details):
\begin{theorem}[\cite{KochTataru2001}]
The Navier--Stokes equations \eqref{eq:NSE} have a unique small global
solution in $X$ for all initial data $u_0$ with $\div u_0=0$ which are small in $BMO^{-1}$.
\end{theorem}
The non-unique solutions constructed here belong to the Koch--Tataru path space $X$ and are allowed to have small (even zero) initial data in $BMO^{-1}$. While the Koch--Tataru argument shows that a small ball in $X$ contains only one solution, which follows from Banach’s fixed point theorem, our non-unique solutions lie outside of this ball. More precisely, if the classical solution $U$ in Theorem~\ref{thm-non-unique} is such that $\|U(T_*)\|_{BMO^{-1}}$ is small enough, then $U$ is a unique small Koch--Tataru solution on $[T_*,T]$, and can in fact be extended to a global classical solution on $[T_*,\infty)$. The rest of the family $(u^{(\sigma)})_{\sigma \in (0,1]}$ are of order one in the Koch--Tataru path space.

We also have the following uniqueness conditions that follow from the energy equality for weak solutions in some Onsager-type spaces.
\begin{theorem} \label{thm:Onsager_conditions}
Let $U$ be a classical solution of \eqref{eq:NSE} on $\mathbb T^d\times [T_*,T]$, $d\geq 2$. If $u$ is a weak solution of \eqref{eq:NSE} with $u(T_*)=U(T_*)$ such that one of the following holds
\begin{enumerate}[1.,ref=\arabic*,left=1em]
\item
$\displaystyle u\in L^2_t \dot H^{1}_x([T_*,T]\times\mathbb T^d) \cap L_t^\infty B^{-1}_{\infty,\infty}([T_*,T]\times\mathbb T^d)$,
\item
$\displaystyle u\in L^\infty_t L^2_x([T_*,T]\times\mathbb T^d) \cap L^2_t \dot H^{1}_x([T_*,T]\times\mathbb T^d) \cap L_t^{2,\infty} L^\infty_x([T_*,T]\times\mathbb T^d)$,
\end{enumerate}
then $u\equiv U$ for all $t\in [T_*,T]$.
\end{theorem}
\begin{proof}
The first condition simply implies that $u\in L^3(T_*,T;B^{\frac13}_{3,c_0})$ via interpolation. Then the estimate on the energy flux from \cite{MR2422377} yields the energy equality 
\[
\tfrac12 \|u(t)\|_{L^2}^2 = \tfrac12 \|u(T_*)\|_{L^2}^2 - \nu \int_{T_*}^t \|\nabla  u(s)\|_{L^2}^2 \, ds, \qquad T_* \leq t \leq T,
\]
which means that $u$ is a Leray--Hopf weak solution on $[T_*,T]$.

The second condition does not directly imply that $u\in L^3(T_*,T;B^{\frac13}_{3,c_0})$. Nevertheless, in \cite{1802.05785}, Theorem 1.1, using an argument that takes advantage of the Laplacian, it was shown that $u$ still satisfies the energy equality and hence it is a Leray--Hopf solution on $[T_*,T]$. Note that Theorem 1.1 is stated for $d=3$ only, but the proof is dimension independent.

To conclude, since $u(T_*)$ is smooth, we can apply Leray's weak--strong uniqueness principle, which holds for all $d\geq 2$, and says that smooth solutions are unique in the class of Leray--Hopf weak solutions. Thus $u \equiv U$ for all $t\in[T_*,T]$.

\end{proof}

Finally, we can show that for any weak solution exhibiting an instantaneous Type I blow-up of the $L^\infty$ norm, the energy must blow up as well. 

\begin{theorem} \label{thm:energy_blow-up}
Let $u$ be a weak solution of \eqref{eq:NSE} on $\mathbb T^d\times [0,T]$, $d\geq 2$, such that 
\begin{enumerate}[left=1em]
\item $u$ is a classical solution of \eqref{eq:NSE} on $\mathbb T^d\times([0,T]\setminus\{T_*\})$ for some $T_* \in [0,T)$,
\item $u(T_*)\in C^\infty(\mathbb T^d)$,
\item $u$ exhibits instantaneous Type I blow-up at $T_*$:
\[
\limsup_{t\to T_*+} \|u(t)\|_{L^\infty} = \infty, \qquad \limsup_{t\to T_*+}\,\sqrt{t-T_*}\|u(t)\|_{L^\infty} < \infty.
\]
\end{enumerate}
Then the energy of $u$ instantaneously blows up at $t=T_*$:
\[
\lim_{t\to T_*+} \|u(t)\|_{L^2} = \infty.
\]
\end{theorem}
\begin{proof}
Suppose to the contrary that $\|u(t)\|_{L^2}$ is bounded. Since $u$ is a classical solution on $(T_*,T]$, it satisfies energy equality on every interval $[t_0,T]$ with $t_0\in(T_*,T]$, and hence
\[
2\nu\int_{T_*}^T \|\nabla u(t)\|^2_{L^2} \, dt \leq \lim_{t \to T_*+}  \|u(t)\|^2_{L^2} < \infty.
\]
Thus $u$ satisfies the second condition of Theorem~\ref{thm:Onsager_conditions} and has to coincide with a smooth solution arising from $u(T_*)$ on some short time interval. Then 
\[
\limsup_{t\to T_*+} \|u(t)\|_{L^\infty} < \infty,
\]
a contradiction.
\end{proof}

\subsection{Inverse energy cascade}\label{sec:inverse_cascade}

Theorem~\ref{thm-non-unique} reveals a new mechanism for non-uniqueness that does not rely on the presence of a spatial singularity. Even when the velocity field remains smooth and even analytic in space for every fixed $t\geq0$, uniqueness can fail spontaneously.   
In other words, the breakdown of uniqueness does not require rough initial data, nor finite-time blow-up; instead, it arises from the delicate transfer of energy across scales that can be generated within a regular flow.

The underlying mechanism is an \emph{inverse energy cascade} originating at infinitely high wavenumbers.  
To quantify this process, consider the energy flux through frequency~$N$ defined via a projection $P_N$ to low modes:
\[
\Pi_N(I) = -\int_I \int_{\mathbb{T}^d} (u \cdot \nabla u) \cdot P_N u \, dx dt.
\]
For instance, $P_N$ can be the Littlewood--Paley projection defined in Section \ref{sec:notation} below. The sign convention is chosen so that a positive $\Pi_N(I)$ represents a net transfer of energy from high to low frequencies. For any weak solution $u$ of the Navier--Stokes equations, the following energy balance holds:
\[
\tfrac12 \|P_Nu(t)\|_{L^2}^2 = \tfrac12 \|P_Nu(t_0)\|_{L^2}^2 +\Pi_N([t_0,t])- \nu \int_{t_0}^t \|\nabla P_N u(s)\|_{L^2}^2 \, ds.
\]
For classical solutions on $[0,T]$, the energy flux across frequency $N$ vanishes
\[
\lim_{N\to \infty} \Pi_N([t_0,t]) =0, \qquad  0 \leq t_0 \leq t \leq T,
\]
ruling out anomalous dissipation, and hence yielding the standard energy equality. In fact, the same conclusion also holds for weak solutions in the Onsager class $u \in L^3([0,T];B^{\frac13}_{3,c_0})$, see \cite{MR2422377}, or satisfying one of the conditions of Theorem~\ref{thm:Onsager_conditions}.

In contrast, the solutions constructed in Theorem~\ref{thm-non-unique} exhibit a strong inverse energy cascade at time $t=T_*$, despite being space-time smooth away from $T_*$:
\[
\lim_{N\to \infty} \Pi_N([T_*,t]) = +\infty, \qquad \lim_{N\to \infty} \Pi_N([t,T]) =0,
\]
for any $t \in(T_*,T)$. This singular behavior of the flux provides an instantaneous injection of energy from infinitely high frequencies and initiates a bifurcation from a classical solution creating a new branch of smooth evolution. Some details of the cascade mechanism are discussed in Subsection \ref{sec:idea}.

\noindent
\textbf{Comparison to ``corner singularities'' and ``initial layers''.} 
The constructed blow-up from the right in Theorem \ref{main-thm} appears to be the first of its kind not only for fluid equations, but in the area of partial differential equations in general. The closest analog might be \emph{corner singularities}. A canonical example is the heat equation with mismatched boundary and initial conditions. Consider the heat equation on the half-line $\{x>0\}$:
\[
\partial_t u = u_{xx}, \qquad  u(x,0)=0, \qquad u(0,t)=1.
\]
This problem admits the explicit solution
\[
u(x,t) = \operatorname{erfc}\!\left(\frac{x}{2\sqrt{t}}\right)=\frac2{\sqrt \pi}\int_{x/{2\sqrt{t}}}^{\infty} e^{-s^2}\, ds.
\]
Differentiating yields
\[
|u_x(0,t)| = (\pi t)^{-\frac12}.
\]
Although this boundary-driven creation of fine scales resembles an inverse cascade across frequencies, it differs from our mechanism in two essential ways. First, in the heat equation example, the small scales are injected by the boundary, whereas in our construction there is no injection of the energy by any external force or boundary; the cascade is produced solely by nonlinear interactions. Second, for the heat equation problem, there does not exist a classical or smooth on $[0,\infty)$ matching both the initial and boundary data; accordingly, this phenomenon does \emph{not} lead to a loss of uniqueness. In contrast, the solutions constructed in the present work are smooth for $t>0$ and undergo instantaneous Type~I blow-up at $t=0$ driven by a complete inverse energy cascade from arbitrarily high frequencies (or zero scales), thereby producing a new smooth branch that immediately departs from the classical evolution with the same initial data.

\subsection{Main ideas of the construction}\label{sec:idea}
Here we give a heuristic overview of the multi-scale construction leading to Theorems~\ref{main-thm} and~\ref{thm-non-unique}, which, at leading order, involves infinitely many interacting components with strong localization in time and frequency, and weak localization in space. The goal is to build a smooth solution of the Navier--Stokes equations that stays regular
for $t>0$, yet whose $L^\infty$ norm diverges as $t\to0^+$. To achieve this, we construct an infinite hierarchy of oscillatory building blocks indexed by $(j,k)\in\mathcal J_d\times\mathbb N$ (with $\mathcal J_d$ some finite set), each corresponding to a shell in Fourier space. The localized elements are designed so that the solution components in shells with index $k+1$ interact nonlinearly to generate those in shells with index $k$ on a particular short time scale. 

\noindent\textbf{Setup and notation.}
Recall that in both of the main theorems, we are given a background $U$ which is classical on $[0,T]$. For convenience, we may translate in time to arrange that $U$ is instead a solution on $[-T_*,T-T_*]$. Thus, the blow-up time is transferred to $t=0$, and we are left to construct solutions of \eqref{eq:NSE} 
that agree with $U$ until $t=0$ and diverge immediately thereafter.

For $k\in\mathbb N$ and $j\in \mathcal J_d$ (see Section~\ref{sec:notation} for the definition), the frequency scales $N_{j,k}$ will be designed to rapidly increase according to the ordering
\[
A^{b^k}\sim N_{1,k} \ll N_{2,k} \ll \dots \ll N_{J_d,k} \ll N_{1,k+1} \sim A^{b^{k+1}}.
\]
For a decreasing sequence of times $N_{1, k+1}^{-2} \ll t_k \ll N_{J_d,k}^{-2}$, the energy transfer from the frequency levels $\{N_{j,k+1}\}_{j\in \mathcal J_d}$ to $\{N_{j,k}\}_{j\in \mathcal J_d}$ occurs on the time interval $[t_{k+1},t_k]$; see Figure~\ref{fig:time_evolution}.

\noindent\textbf{Profiles and amplitudes.}
We introduce localized oscillatory profiles which, up to technical modifications, are of the form
\[
\psi_{j,k}(x)\approx N_{j,k}^{-2}a_{j,k}(x)\theta_j\sin(N_{j,k}\eta_j\cdot x),
\]
where $\theta_j,\eta_j\in\mathbb S^{d-1}$ satisfy $\theta_j\cdot\eta_j=0$ and $N_{j,k}\eta_j\in\mathbb Z^d$ for all $j\in \mathcal J_d$ and $k\in\mathbb N$. The slowly varying amplitudes $a_{j,k}$ are bounded and will be chosen recursively to ensure the creation of the desired velocity at the adjacent lower frequency shells. 

\noindent\textbf{The ansatz.}
The full velocity field $u$ is defined to be the sum of a leading order critical part $v$ and a small, slightly subcritical perturbation $w$. The leading part is assembled as a series\footnote{Here and elsewhere, sums over $k$ and $j$ are implicitly taken over $\mathbb N=\{0,1,2\ldots\}$ and $\mathcal J_d$, respectively.} which roughly takes the form
\[
v(x,t) = \sum_k v_k(x,t), \qquad v_k(x,t) \approx -\sum_j N_{j,k} e^{-N_{j,k}
^2t} \mathbb P\Delta \psi_{j,k}(x),  \qquad t \geq t_k,
\]
where $\mathbb{P}$ denotes the Leray projector. By the slowness of $a_{j,k}$ relative to the frequency scale $N_{j,k}$, we have $$-\mathbb P\Delta\psi_{j,k}(x)\approx a_{j,k}(x)\theta_j\sin(N_{j,k}\eta_j\cdot x)=O(1),$$ with errors that are lower order in the sense that they involve the ratio of the successive frequency scales, or decay in time much faster than the natural scale-invariant rate.

Each component $v_k$ is generated over the short interval $[t_{k+1},t_k]$ by self-interactions of higher modes $v_{k+1}\otimes v_{k+1}$, thereby transferring energy
down the frequency ladder, producing the inverse energy cascade. To justify this ansatz, we need to show that
\begin{equation} \label{ansatz-intro}
v_k(x,t) \approx -\int_0^{t} e^{(t-s)\Delta }\mathbb P\div ( v_{k+1}\otimes v_{k+1})(s) \, ds.
\end{equation}
When $d=2$, the separation of time scales ensures that the following integral is dominated by local interactions:
\begin{equation}\label{scale-separation}
\int_0^{t} N_{j,k} N_{j',k} e^{-N_{j,k}^2 s-N_{j',k}^2 s} \, ds \sim 
\begin{cases}
1, &j=j',\\
N_{j,k}^{-1} N_{j',k} \ll 1, &j > j',
\end{cases}
\end{equation}
for $t\geq t_{k-1}$. A precise computation showing the smallness of nonlocal interactions can be found in \eqref{eq:NonlocalInteractions}. When $d\geq3$, these interactions can be minimized by selection of the supports of the building blocks. Thus we have
\[
\begin{split}
\int_0^t P_{< N_{1,k+1}}\mathbb P\div v_{k+1}\otimes v_{k+1} \, ds &\approx P_{< N_{1, k+1}}\mathbb P\div \sum_j \mathbb P\Delta \psi_{j,k+1} \otimes \mathbb P\Delta \psi_{j,k+1}\\
& \approx \mathbb P\div\left(\sum_ja_{j,k+1}^2\theta_j\otimes\theta_j\right)
\end{split}
\]
for $t\geq t_k$, where we only take into account the leading order terms in the approximation.
Choosing the coefficients $a_{j,k}$ so that
\[
\sum_ja_{j,k+1}^2\theta_j\otimes\theta_j \approx \mathcal D \sum_j N_{j,k}\psi_{j,k}+p\Id,
\]
for some scalar function $p(x)$, where $\mathcal D$ is as defined in Section~\ref{sec:notation} and ensures that
\[
\int_0^t P_{< N_{1,k+1}}\mathbb P\div v_{k+1}\otimes v_{k+1} \, ds \approx \sum_j N_{j,k} \Delta \psi_{j,k} + \nabla p,
\]
for $t\geq t_k$, which justifies \eqref{ansatz-intro}. The smallness of nonlocal interactions $v_k \otimes v_{k'}$ for $k\ne k'$ follows as well from the separation of time scales.

\begin{figure}[h!]
\centering
\begin{tikzpicture}[scale=1.5, >=stealth]

\draw[->, thick] (0,0) -- (6,0) node[right] {\small $t$};
\draw[->, thick] (0,0) -- (0,1.8);

\draw[dashed] (1/6,0) -- (1/6,1.3);
\draw[dashed] (2.2,0) -- (2.2,0.485);

\node[below] at (1/6,0) {\small $t_{k+1}$};
\node[below] at (2.2,0) {\small $t_k$};
\node at (3.3,0.6) {\small $\|v_k(t)\|_{L^\infty}$};
\node at (0.9,1.2) {\small $\|v_{k+1}(t)\|_{L^\infty}$};

\draw[thick, smooth, domain=0:6, samples=300]
  plot (\x,{(2/(1+exp(-2*\x))-1)*0.8*exp(-0.2*\x)});

\draw[thick, smooth, domain=0:6, samples=300]
  plot (\x,{(2/(1+exp(-30*\x))-1)*1.9*exp(-2*\x)});

\end{tikzpicture}
\caption{Time evolution of $\|v_k(t)\|_{L^\infty}$ and $\|v_{k+1}(t)\|_{L^\infty}$.}
\label{fig:time_evolution}
\end{figure}
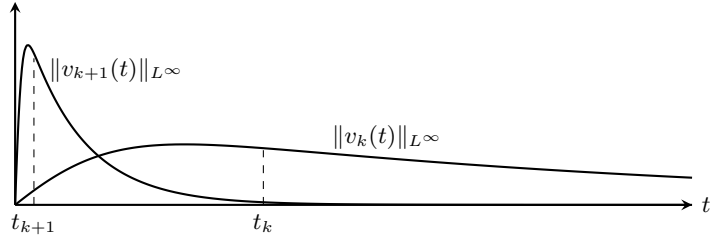

\begin{remark}
To ensure that the solution is locally $L^2$ in space and time (which is crucial for the weak solution formulation; see Section~\ref{sec:weak_solutions}), we additionally include a cutoff function $\varphi_{j,k}$ in the ansatz. (See the precise definition in Subsection \ref{sec:potential} below.) As a result, the component of the solution in the frequency shell near $N_{j,k}$ occupies, to leading order, a spatial domain of volume $O(2^{-k})$.
\end{remark}

\addtocontents{toc}{\protect\setcounter{tocdepth}{1}}

\section{Preliminaries}
\label{sec:prelim}

\subsection{Notation}\label{sec:notation}

Let $\mathbb T^d$ denote the $d$-dimensional torus $(\mathbb R/2\pi\mathbb Z)^d$. We define the spaces $L^p(\mathbb T^d)$ in the standard way for $1\leq p\leq\infty$, and write the norm as $\|f\|_{L^p(\mathbb T^d)}$ or $\|f\|_p$ when there is no ambiguity. We also denote by $L^{p,\infty}$ the Lorentz space (or weak $L^p$ space) defined in the standard way via distribution function  $\lambda(s)=|\{x\in\mathbb T^d:|f(x)|>s\}|$. Here and elsewhere, $|\Omega|$ denotes the volume of $\Omega$ restricted to the fundamental domain of $\mathbb T^d$. We point out the inclusion 
\[L^p \subset L^{p,\infty},\]
which is strict in the sense that, while every function in $L^p$ belongs to $L^{p,\infty}$, the converse fails due to examples resembling $|x|^{-d/p}$ in a neighborhood of $x=0$.

We also write $\|\cdot\|_{C^\alpha}$ to denote the $\alpha$-H\"older seminorm on $\mathbb R^d$ or $\mathbb T^d$.

For $f\in L^2(\mathbb T^d)$, we define the Fourier series and inverse according to
\begin{align*}
    \hat f(\xi)\coloneqq \fint_{\mathbb T^d}f(x)e^{-ix\cdot\xi}dx,\qquad\mathcal F^{-1}g(x)\coloneqq\sum_{\xi\in\mathbb Z^d}g(\xi)e^{ix\cdot \xi},
\end{align*}
for $\xi\in\mathbb Z^d$ and $x\in\mathbb T^d$, where $\fint_\Omega\coloneqq |\Omega|^{-1}\int_{\Omega}$ for $\Omega\subset\mathbb T^d$. 

Then the following operators can be defined by multiplication on the Fourier side: the heat propagator $$e^{t\Delta}f\coloneqq\mathcal F^{-1}(e^{-t|\xi|^2}\hat f(\xi)),$$ the Leray projection $$\mathbb Pf\coloneqq \mathcal F^{-1}\left(\left(\Id-\frac{\xi\otimes \xi}{|\xi|^2}\right)\hat f(\xi)\right),$$ for $f$ vector-valued, and the Littlewood--Paley projections
$$P_Nf\coloneqq\mathcal F^{-1}(\psi(|\xi|/N)\hat f(\xi)),$$
for all dyadic numbers $N\in 2^\mathbb N$, where $\psi\in C_c^\infty((2/3,3/2))$ is a bump function such that $\sum_N\psi(r/N)\equiv 1$ for all $r\in[3/4,\infty)$, say.

Let $\mathrm{Sym}^{d}(\mathbb R)$ be the space of real symmetric $d\times d$ matrices, equipped with some norm. Let $J_d\coloneqq \dim \mathrm{Sym}^{d}(\mathbb R)=d(d+1)/2$, and define $\mathcal J_d\coloneqq\{1,2,\ldots,J_d\}$. For two vectors $v,w\in\mathbb R^d$, we define the tensor product $(v\otimes w)_{i,j}=v_iw_j$ and symmetrized tensor product $(v\odot w)_{i,j}=\frac12(v_iw_j+v_jw_i)$. For brevity, we apply this algebraic notation to operators as well; for instance, if $f$ is a vector function, then we denote by $\nabla\odot f$ the symmetric tensor field with components $\frac12(\partial_if_j+\partial_jf_i)$.

We define the additional Fourier multipliers
\begin{align*}
    \mathcal Df&\coloneqq2\nabla\odot f-2(\div f)\Id,\\
    \newD f&\coloneqq2\nabla\odot f-(\div f)\Id,\\
    \mathcal R&\coloneqq \Delta^{-1}\newD,\\
    \mathbb Q&\coloneqq2\Delta^{-1}\nabla\odot \mathbb P\div,
\end{align*}
for $f$ a vector field on $\mathbb T^d$, and record the following essential identities: for $\mathcal D$,
\begin{align*}
    \div \mathcal D=\Delta\mathbb P,
\end{align*}
for $\newD$,
\begin{align*}
    \div\newD = \Delta,\qquad \newD=\mathbb Q\mathcal D+\left(2\frac{\nabla\otimes\nabla}{\Delta}-\Id\right)\div,
\end{align*}
for $\mathcal R$,
\begin{align*}
    \div \mathcal R=P_{\neq0},
\end{align*}
where $P_{\neq0}$ is the projection operator to frequencies $\xi\neq0$, and for $\mathbb Q$,
\begin{align}\label{eq:Q_identities}
    \mathbb Q=\mathcal R\mathbb P\div,\qquad \mathbb Q\mathcal D=2\nabla\odot\mathbb P.
\end{align}
This version of the anti-divergence $\mathcal R$ is standard and improves on the original definition from~\cite{DeLellisSzekelyhidi2013} in the sense that $\Delta\mathbb P$ is now a classical differential operator. This is important in the sequel because it guarantees locality and $C^1\to L^\infty$ boundedness.

\subsection{Koch--Tataru spaces} \label{sec:KT}
We state two equivalent definitions of the space $BMO^{-1}(\mathbb R^d)$. Then the periodic version $BMO^{-1}(\mathbb T^d)$ is defined as the space of distributions on $\mathbb T^d$ whose periodic extensions are in $BMO^{-1}(\mathbb R^d)$, with the norm defined as the norm of the extension.

In practice, we use the definition that a distribution $f$ is in $BMO^{-1}(\mathbb R^d)$ if there exist $g_1,\ldots,g_d\in BMO$ with $f=\partial_1g_1+\cdots+\partial_dg_d$, with the norm $\|f\|_{BMO^{-1}}\coloneqq \inf\sum_{i=1}^d\|g_i\|_{BMO}$, the infimum taken over all such $g$. Here, $BMO$ is defined as the space of $f\in L^1_{loc}$ such that
\begin{align*}
    \|f\|_{BMO}\coloneqq \sup_{Q}\fint_Q\left|f(x)-\fint_Qf(y)dy\right|dx<\infty,
\end{align*}
where the supremum is taken over cubes $Q$ in $\mathbb R^d$. An equivalent definition of the $BMO^{-1}$ norm is
\begin{align*}
    \|f\|_{BMO^{-1}(\mathbb R^d)}=\sup_{x_0\in\mathbb R^d}\sup_{R>0}\left(\int_0^{R^2}\fint_{B(x_0,R)}|e^{t\Delta}f(x)|^2dxdt\right)^\frac12.
\end{align*}
We write $BMO^{-1}_{T}$ to denote the space whose norm is defined analogously, but with the supremum taken over only  $R\in(0,T^\frac12)$. For periodic functions, it is clear that inclusion in $BMO^{-1}$ and $BMO^{-1}_{T}$ is equivalent.

We endow $BMO^{-1}(\mathbb T^d)$ with the weak-* topology, in the sense that it is the dual of the Triebel--Lizorkin space $\dot F^1_{1,2}$ with norm\footnote{We supply the definition in the periodic case only in order to avoid technicalities that appear in $\mathbb R^d$ due to the zero frequency.}
\begin{align*}
\|f\|_{\dot F_{1,2}^1(\mathbb T^d)}&=\left\|\left(\sum_NN^2|P_Nf|^2\right)^{\frac12}\right\|_{L^1(\mathbb T^d)}.
\end{align*}
This topology is the natural one in the sense that it exactly captures the time-continuity of the heat propagator from $BMO^{-1}$ initial data.

The second formulation of $BMO^{-1}$ motivates the definition of the natural path space for the heat equation from such data, first introduced by Koch and Tataru~\cite{KochTataru2001}. For $T\in(0,\infty]$, let
\begin{align*}
    X_T(\mathbb R^d)\coloneqq \left\{u\in L_{loc}^2(\mathbb R^d\times[0,T];\mathbb R^d):\div u=0,\,\|u\|_{X_T}<\infty\right\},
\end{align*}
where
\begin{align*}
    \|u\|_{X_T}\coloneqq \sup_{t\in(0,T]}t^\frac12\|u(t)\|_{L^\infty}+\sup_{x_0\in\mathbb R^d}\sup_{R\in(0,T^\frac12]}\Big(\int_0^{R^2}\fint_{B(x_0,R)}|u(x,t)|^2dxdt\Big)^\frac12.
\end{align*}
By time translation, we can analogously define $X_{I}$ for an interval $I\subset\mathbb R$.

Then $X_T(\mathbb T^d)$ is defined as those functions $u\in L^2(\mathbb T^d\times[0,T];\mathbb R^d)$ whose periodic extensions to $\mathbb R^d$ lie in $X_T(\mathbb R^d)$.

\subsection{Weak solutions}\label{sec:weak_solutions}

The notion of a \emph{weak solution} provides the minimal analytic framework in which the
Navier--Stokes equations can be interpreted for arbitrary finite-energy initial data.
Since the pioneering work of Leray~\cite{Leray1934}, it has been known that weak solutions
exist globally in time for any $d\ge2$.

\begin{definition}\label{def:weak_solutions}
Denote by $\mathcal{D}_T$ the space of divergence-free test functions $\varphi \in C^\infty (\mathbb{T}^d \times \mathbb{R}) $ such that  $\varphi =0$ if $t\geq T$.
Let $ u_0 \in L^2(\mathbb{T}^d)$  be weakly divergence-free. A vector field $ u \in L^2 (\mathbb{T}^d \times [0,T])$ is a weak solution of \eqref{eq:NSE} with initial data $u_0$ if the following hold:
\begin{enumerate}[1.,ref=\arabic*,left=1em]
    \item For $a.e.$ $t\in [0,T]$, $u$ is weakly divergence-free;
    
    \item For any $\varphi \in \mathcal{D}_T$,
\begin{equation}
\int_{\mathbb{T}^d} u_0(x)\cdot \varphi(x,0) \, dx = - \int_0^T \int_{\mathbb{T}^d} u\cdot \big(  \partial_t \varphi+ \Delta \varphi +  u \cdot \nabla \varphi  \big) \, dx dt .
\end{equation}
\end{enumerate}
\end{definition}

Note that weak solutions also satisfy the mild formulation, i.e,
\[
u = e^{t\Delta}u_0 - \int_0^t e^{(t-s)\Delta} \mathbb{P}\div(u\otimes u) (s)\, ds,
\]
in the sense of distributions, \cite{MR316915}.

\addtocontents{toc}{\protect\setcounter{tocdepth}{2}}

\section{Construction of the principal part}\label{sec:principal}

\subsection{Frequency scales}\label{frequency_scales_section}

In Sections~\ref{sec:principal}--\ref{sec:est}, we construct fields on the $2\pi$-periodic torus $\mathbb T^d\coloneqq \mathbb R^d/(2\pi\mathbb Z)^d$. These fields will later be rescaled and perturbed aperiodically in Sections~\ref{sec:corrector}--\ref{sec:proof}.

We introduce the parameters $b>1$, to be chosen large depending only on $d\geq2$, and $A>1$, to be chosen large depending on all other parameters. Fix also an $m_*\in\mathbb N$ to be specified later. From here we define the frequency scales as follows: fix the minimum frequency $N_{1,0}=1$ at $(j,k)=(1,0)$, while for all other $k\in\mathbb N$ and $j\in\mathcal J_d$, define
\begin{equation}\label{N_definition}
N_{j,k}=\left\{
\begin{aligned}
&m_*\big\lceil A^{b^{k+(j-1)/J_d}}\big\rceil, \quad &d=2,\\
&m_*\big\lceil A^{b^{k}}\big\rceil, \quad &d \geq 3,
\end{aligned}
\right.
\end{equation}
where $\lceil \cdot\rceil$ is the ceiling function. We assume $A$ is sufficiently large depending on $b$, $d$, and $m_*$ so that $N_{j_1,k_1}\ll N_{j_2,k_2}$ if either $k_1<k_2$, or $k_1=k_2$ and $j_1<j_2$.

For $k\in\mathbb N_{\geq1}$ and $j\in\mathcal J_d$, we define also an interceding sequence of frequency scales
\begin{equation}\label{M_definition}
    M_{j,k}=\left\{
\begin{aligned}
&\big\lceil A^{\gamma b^{k}}\big\rceil, \quad &j=1\text{ or }d\geq 3,\\
&\big\lceil A^{\gamma b^{(j-1)/J_d}}\big\rceil M_{1,k}, \quad &j\geq2\text{ and }d=2,
\end{aligned}
\right.
\end{equation}
for a parameter $\gamma\in(0,1)$ to be specified. Assuming additionally that
\begin{align*}
    \gamma> b^{-1/J_d},
\end{align*}
and taking $A$ still larger as needed, we arrange that for $d=2$,
\begin{equation}\begin{aligned}\label{N_and_M_ordering}
    A^cN_{j-1,k}\leq M_{j,k}\leq A^{-c}N_{j,k}&,\quad 2\leq j\leq J_d,\\
    A^cN_{J_d,k-1}\leq M_{1,k}\leq A^{-c} N_{1,k}&
\end{aligned}\end{equation}
and for $d\geq 3$,
\begin{equation}\begin{aligned}\label{N_and_M_ordering_3d}
   A^cN_{j,k-1}\leq M_{j,k}\leq A^{-c}N_{j,k}&,\quad 1\leq j\leq J_d
\end{aligned}\end{equation}
for some $c>0$ (depending on $b$ and $\gamma$).

\subsection{Geometry of the building blocks}

Choose some $\theta_j,\eta_j\in\mathbb S^{d-1}\cap\mathbb Q^d$ such that $\theta_j\cdot\eta_j=0$ for all $j\in \mathcal J_d$, and $(\theta_j\otimes\theta_j)_{j\in\mathcal J_d}$ are linearly independent and generate $\Id$ with positive coefficients. Fixing $m_*$ in Section~\ref{frequency_scales_section} to be such that $m_*\eta_j\in\mathbb Z^d$ for all $j$, we ensure that $N_{j,k}\eta_j\in\mathbb Z^d$ for all $j$ and $k$.

For $\rho<1/10$, define $\mathcal C_{j,k}(\rho)$ to be the $2\pi/M_{j,k}$-periodic cylinder of radius $\rho M_{j,k}^{-1}$ whose axis is the periodic line $\mathbb R\theta_j\pmod{2\pi\mathbb Z^d/M_{j,k}}$. Fix $\delta_0>0$ small enough (depending on $d$ but not $k$) so that
\begin{align}\label{pipe_volume_bound}
|\mathcal C_{j,k}(4\delta_0)|\leq1/(10J_d).
\end{align}
When $d\geq3$, taking $\delta_0>0$ even smaller as necessary, we may translate each cylinder so that
\begin{align*}
    \mathcal C_{j_1,k}(4\delta_0)\cap\mathcal C_{j_2,k}(4\delta_0)=\emptyset\qquad\forall j_1\neq j_2.
\end{align*}

For $k\geq1$, let $\varphi_{j,k}\in C_c^\infty(\mathcal C_{j,k}(\delta_0))$ be a $2\pi/M_{j,k}$-periodic cutoff function with $\theta_j\cdot\nabla\varphi_{j,k}=0$, normalized so that
\[
(2\pi)^{-d}\int_{\mathbb T^d}\varphi_{j,k}^2(x)\sin^2(N_{j,k}\eta_j\cdot x) \, dx =1, \qquad \forall j,k,
\]
and with
\begin{equation} \label{eq:varphi_bound}
\|\nabla^n \varphi_{j,k}\|_\infty\lesssim M_{j,k}^{n}, \qquad \forall j,k.
\end{equation}
Then define the following regions in $\mathbb T^d$:
\begin{align*}
    \Omega_k=\bigcap_{k'=1}^k\bigcup_{j\in\mathcal J_d}\mathcal C_{j,k'}((3-2^{-(k-k')})\delta_0),\quad \widetilde\Omega_k=\bigcap_{k'=1}^k\bigcup_{j\in \mathcal J_d}\mathcal C_{j,k'}((3-\frac342^{-(k-k')})\delta_0),
\end{align*}
with the convention $\Omega_0=\widetilde\Omega_0=\mathbb T^d$.

\begin{lemma}
For any $k\geq0$, we have
    \begin{align}
    |\Omega_{k}|\leq2^{-k}|\mathbb T^d|.\label{Omega_volume_estimate}
\end{align}
Moreover, there is a constant $C_0>1$ such that if $Q\subset\mathbb T^d$ is a cube of length $\ell(Q)\in [C_0M_{1,k_0}^{-1},2\pi)$, then
\begin{align}\label{cube_intersection_volume_bound}
    |\Omega_{k}\cap Q|\leq2^{-(k-k_0)}|Q|\qquad\forall k\geq k_0.
\end{align}
\end{lemma}

\begin{proof}
The proof is similar to \cite{CoiculescuPalasek2025}*{Lemma 3.3}. Recall that, by definition, $\mathcal C_{j,k}(r)$ is a $2\pi/M_{j,k}$-periodic cylinder. Since $M_{1,k}$ divides $M_{j,k}$ for all $j\in\mathcal J_d$ (by \eqref{M_definition}), the cylinders are all in fact $2\pi/M_{1,k}$-periodic. Thus we may arrange cutoff functions $\rho_k\in C^\infty(\mathbb T^d;[0,1])$ that are $2\pi/M_{1,k}$-periodic with
\begin{align*}
    \rho_k\equiv1\quad\text{on }  \bigcup_{j\in\mathcal J_d} \mathcal C_{j,k}(3\delta_0),
\end{align*}
and
\begin{align*}
    \supp\rho_k\subset\bigcup_{j\in\mathcal J_d}\mathcal C_{j,k}(4\delta_0).
\end{align*}
Due to the pointwise inequality $\mathbf1_{\Omega_k}\leq\rho_1\cdots\rho_k$, we can estimate
\begin{align*}
    |\Omega_k|\leq\|\rho_1\rho_2\cdots\rho_k\|_{L^1}.
\end{align*}
Note that by \eqref{pipe_volume_bound}, we have
\begin{align*}
    \|\rho_k\|_{L^1}\leq \sum_{j\in\mathcal J_d}|\mathcal C_{j,k}(4\delta_0)|\leq1/10.
\end{align*}
By improved H\"older's inequality (see, e.g., Lemma 2.1 in \cite{ModenaSz18}),
\begin{align*}
    \|\rho_1\rho_2\cdots\rho_{k+1}\|_{L^1}&\leq \|\rho_1\rho_2\cdots\rho_k\|_{L^1}\|\rho_{k+1}\|_{L^1}+O(M_{1,k+1}^{-1}\|\rho_1\rho_2\cdots\rho_k\|_{C^1}\|\rho_{k+1}\|_{L^1})\\
    &\leq \|\rho_1\rho_2\cdots\rho_k\|_{L^1}/10+O(M_{J_d,k}/M_{1,k+1}).
\end{align*}
Combining with the trivial estimate $\|\rho_1\|_{L^1}\leq|\mathbb T^d|$, we conclude \eqref{Omega_volume_estimate} by induction. The same argument, additionally introducing a cutoff for $Q$, yields \eqref{cube_intersection_volume_bound}.
\end{proof}

Clearly there exist cutoff functions $\chi_k\in C_c^\infty(\widetilde\Omega_{k-1})$ with $\chi_k\equiv1$ on $\Omega_{k-1}$ obeying
\begin{align} \label{eq:chi_k_bound}
    \|\nabla^n\chi_k\|_\infty\lesssim_nM_{J_d,k-1}^n.
\end{align}

\subsection{The velocity potential}\label{sec:potential}

To avoid issues related to nonlocal operators in the construction, we define the principal part of the velocity field as the Laplacian of a particular potential.

Let $\phi_k$ be a standard mollifier at length scale $\ell_k:=N_{1,k}^{-\frac12}N_{1,k+1}^{-\frac12}$. We will need $\ell_k \leq  N_{J_d,k}^{-1}$, which clearly holds provided $b$ is large enough. Define the sequence of profiles
\begin{align}
\psi_{j,0}(x)&\coloneqq N_{j,0}^{-2}a_{j,0}(x)\theta_j\sin(N_{j,0}\eta_j\cdot x)\label{def_psi_0},\\
\psi_{j,k}(x)&\coloneqq N_{j,k}^{-2}\phi_k*(a_{j,k}(x)\varphi_{j,k}(x)\theta_j\sin(N_{j,k}\eta_j\cdot x)),\qquad k\geq1,\label{def_psi}
\end{align}
where the coefficients $a_{j,k}$ are to be constructed in such a way that
\begin{equation}\label{a-convex-integration-identity}
\sum_ja_{j,k+1}^2\theta_j\otimes\theta_j=2\mathcal D \sum_j N_{j,k}\psi_{j,k}+p\Id,
\end{equation}
for some scalar function $p(x)$,
\begin{equation} \label{eq:a_bounds}
    \|\nabla^na_{j,k}\|_\infty\lesssim N_{J_d,k-1}^{n}, \qquad \forall j,k,
\end{equation}
and
\begin{equation}\label{a-support}
    \supp a_{j,k}\subset \widetilde\Omega_{k-1}.
\end{equation}
The existence of such coefficient functions $a_{j,k}$ is guaranteed by the following lemma.
\begin{lemma}
    For $j\in\mathcal J_d$ and $k\geq0$, there exist coefficient functions $a_{j,k}\in C^\infty(\mathbb T^d;\mathbb R)$ satisfying \eqref{a-convex-integration-identity} and \eqref{eq:a_bounds}. Further, $\supp\psi_{j,k}\subset \Omega_k$ and
    \begin{align}\label{eq:Dpsi_bounds}
        \|\nabla^n\psi_{j,k}\|_{L^\infty}\lesssim_n N_{j,k}^{-2+n},
    \end{align}
    for $n\geq0$.
\end{lemma}

\begin{proof}
    Recall the well-known rank-one decomposition
    \begin{align*}
        S=\sum_{j\in\mathcal J_d}\Gamma_j^2(S)\theta_j\otimes\theta_j,
    \end{align*}
    for all $S\in \mathrm{Sym}^{d}(\mathbb R)$ in a $c_0$-neighborhood of $\Id$, where $\Gamma_j:B(\Id,c_0)\to\mathbb R$ are smooth.
    
We define $a_{1,0}(x)=1$, $a_{j,0}(x)=0$ for $j\neq1$, and, recursively, for $k\geq 0$,
\[
a_{j,k+1}(x)=c^{-\frac12}\chi_{k+1}(x)\Gamma_j \big(\Id+cS_k(x)\big),
\]
where
\[
S_k\coloneqq2\mathcal D \sum_j N_{j,k}\psi_{j,k},
\]
and $\psi_{j,k}$ is defined in \eqref{def_psi_0}--\eqref{def_psi} in terms of $a_{j,k}$. The absolute constant $c$ will be chosen 
small enough so that
\begin{equation} \label{eq:condition_on_c}
c\|S_k\|_\infty \leq c_0,
\end{equation}
holds for all $k$, and hence the rank-one decomposition can be applied. Note that
\[
\|a_{j,k}\|_\infty \lesssim c^{-\frac12},
\]
for all $j$ and $k$. We can also get rough estimates on the derivatives of $\psi_{j,k}$ by shifting all the derivatives onto the mollifier: 
\[
\|\nabla^n \psi_{j,k}\|_\infty \lesssim N^{-2}_{j,k} \ell^{-n}_k  \|a_{j,k}\|_\infty \lesssim N^{-2}_{j,k} \ell^{-n}_k  c^{-\frac12}.
\]
Then
\[
\|\nabla^n S_k\|_\infty \lesssim \sum_j N_{j,k}\|\nabla^{n+1} \psi_{j,k}\|_\infty \lesssim N^{-1}_{1,k} \ell_k^{-n-1} c^{-\frac12}.
\]

Before we obtain optimal bounds on the derivatives, we continue with rough estimates. Using the Fa\`a di Bruno formula,
\begin{align*}
\|\nabla^n_x\Gamma_j(\Id+cS_k(x))\|_\infty &\lesssim_n \sum_{i=0}^n\sum_{\alpha_1+\cdots+\alpha_i=n}c^i\|\nabla^{\alpha_1} S_k\|_\infty\cdots\|\nabla^{\alpha_i} S_k\|_\infty \|\nabla^i\Gamma_j\|_{L^\infty}\\
&\lesssim \sum_{i=0}^n\sum_{\alpha_1+\cdots+\alpha_i=n}c^i(N_{1,k}^{-1}\ell_k^{-\alpha_1-1}c^{-\frac12})\cdots(N_{1,k}^{-1}\ell_k^{-\alpha_i-1}c^{-\frac12})\\
&\lesssim \sum_{i=0}^nc^\frac i2 N_{1,k}^{-i}\ell_k^{-n-i}\\
&\lesssim c^{\frac{n}2}N^{-n}_{1,k} \ell_k^{-2n}.
\end{align*}
Since $M_{J_d,k} \leq N_{J_d,k} \leq \ell_k^{-1}$, we may neglect the terms in which derivatives fall on $\chi_{k+1}$, which yields
\[
\begin{split}
\|\nabla^n a_{j,k+1}\|_\infty &= \big\|\nabla^n \Big(c^{-\frac12}\chi_{k+1}(x)\Gamma_j\big(\Id+cS_k(x)\big)\Big)\big\|_\infty\\
&\lesssim_n c^{-\frac12} \|\nabla^n_x\Gamma_j(\Id+cS_k(x))\|_\infty\\
&\lesssim_n c^{\frac{n-1}{2}}N^{-n}_{1,k} \ell_k^{-2n}.
\end{split}
\]
Now we are in a position to obtain optimal bounds on the derivatives of $S_k$. Since $\ell_{k-1}^{-1}  =  N_{1,k-1}^{\frac12}N_{1,k}^{\frac12}$ and $M_{j,k} \leq  N_{j,k}$,
\[
\begin{split}
\|\nabla^{n} \psi_{j,k}\|_\infty&\lesssim_n N_{j,k}^{-2} \sum_{i=0}^{n} \|\nabla^i(a_{j,k}\varphi_{j,k})\|_\infty N_{j,k}^{n-i}\\
&\lesssim_n  N_{j,k}^{-2}\sum_{i=0}^n (c^{\frac{i-1}{2}}N^{-i}_{1,k-1} \ell_{k-1}^{-2i} + c^{-\frac12} M_{j,k}^i)N_{j,k}^{n-i}\\
&\lesssim_n N_{j,k}^{-2} c^{-\frac{1}{2}} N_{j,k}^{n},
\end{split}
\]
which in turn improves the estimate of $S_k$:
\[
\begin{split}
\|\nabla^nS_k\|_\infty &\lesssim \sum_j N_{j,k}\|\nabla^{n+1} \psi_{j,k}\|_\infty\\
&\lesssim \sum_j N_{j,k} N_{j,k}^{-2} c^{-\frac{1}{2}}  N_{j,k}^{n+1}\\
&\lesssim \sum_j c^{-\frac{1}{2}}  N_{j,k}^{n}\\
&\lesssim c^{-\frac{1}{2}}  N_{J_d,k}^{n},
\end{split}
\]
and also on $a$:
\[
\begin{split}
\|\nabla^n a_{j,k+1}\|_\infty &= \big\|\nabla^n \Big(c^{-\frac12}\chi_{k+1}(x)\Gamma_j\big(\Id+cS_k(x)\big)\Big)\big\|_\infty\\
&\lesssim_n c^{-\frac12} \|\nabla^n_x\Gamma_j(\Id+cS_k(x))\|_\infty\\
&\lesssim_n N^{n}_{J_d,k}.
\end{split}
\]
Thus bounds \eqref{eq:a_bounds} and \eqref{eq:Dpsi_bounds} hold.

The inclusion \eqref{a-support} is immediate from the definition of the cutoff $\chi_{k+1}$.

    Next, we compute
    \begin{align*}
        \sum_ja_{j,k+1}^2\theta_j\otimes\theta_j&=c^{-1}\chi_{k+1}^2(x) \sum_j\Gamma_j^2\big(\Id+cS_k(x)\big)\theta_j\otimes\theta_j\\
        &=\chi_{k+1}^2(x)(c^{-1}\Id+S_k(x)),
    \end{align*}
    which verifies \eqref{a-convex-integration-identity}, as long as $\chi_{k+1}\equiv1$ on $\supp S_k$. Indeed, recall that by definition $\chi_{k+1}\equiv1$ on $\Omega_k$; thus it suffices to verify that $\psi_{j,k}$, and therefore $S_k$, are supported in $\Omega_k$. We have
    \begin{align*}
        \supp\psi_{j,k}\subset (\supp{a_{j,k}}\cap\supp\varphi_{j,k})+B(0,\ell_k),
    \end{align*}
    the sum being taken in the Minkowski sense, as a result of the convolution. From \eqref{a-support}, the fact that $\supp\varphi_{j,k}\subset\mathcal C_{j,k}(\delta_0)$, and the definitions of $\Omega_k$ and $\widetilde\Omega_k$, it is easy to verify that this is contained in $\Omega_k$, upon choosing parameters such that $\ell_kM_{j,k'}\ll2^{-(k-k')}$ for all $j$ and $k'\leq k$. 
\end{proof}

We define, for $k\geq 0$,
\[
\bar v_k(x,t)= \sum_j \bar v_{j,k}(x,t) = \sum_j -N_{j,k} e^{-N_{j,k}
^2t} \Delta \psi_{j,k}(x),
\]
and
\[
v_k(x,t) = -\int_0^{t} e^{(t-s)\Delta }\mathbb P\div (\bar v_{k+1}\otimes \bar v_{k+1})(s) \, ds.
\]

Then the full principal part is
\begin{align*}
    v(x,t)=\sum_{k=0}^\infty v_k(x,t).
\end{align*}
We analogously define the approximate principal part,
\begin{align*}
    \bar v(x,t)=\sum_{k=0}^\infty \bar v_k(x,t).
\end{align*}

We can also define tensor fields $R_k$ and $\bar R_k$ such that $v_k=\div R_k$ and $\bar v_k=\div\bar R_k$.

Finally we define the tensor fields 
\[
\bar R_k(x,t)= \sum_j \bar R_{j,k}(x,t) = \sum_j -N_{j,k} e^{-N_{j,k}
^2t}  \newD  \psi_{j,k}(x),
\]
and
\[
R_k(x,t) = -\int_0^{t} e^{(t-s)\Delta }\mathcal R\mathbb P\div (\bar v_{k+1}\otimes \bar v_{k+1})(s) \, ds,
\]
recalling the definitions from Section~\ref{sec:notation}.

\section{Estimates on the leading order solution}\label{sec:est}
In this section we show that the principal part $v$ and the approximate principal part $\bar v$ are indeed close. As a consequence we show that $v$ satisfies a forced Navier--Stokes system with the forcing term appropriately small.

\subsection{Error bounds on the velocity}

We start with the following estimate for $\bar R_k$. There are two time regimes separated by a time scale $t_k$ which can be chosen arbitrarily in 
\begin{equation}\label{time-sequence}
N_{1, k+1}^{-2} \ll t_k \ll N_{J_d,k}^{-3}.
\end{equation}
For concreteness, let us define $t_k\coloneqq N_{J_d,k}^{-4}$.
\begin{proposition}\label{barv-I_estimate_proposition}
Let
\[
\mathcal{I}_k = -\int_0^{t} e^{(t-s)\Delta }\sum_j N_{j,k+1}^2 e^{-2N_{j,k+1}
^2s} \mathbb Q (a_{j,k+1}^2 \theta_j\otimes \theta_j) \, ds.
\]
Then for any $\varepsilon_0 >0$, $\alpha \in(0, \frac{1}{10})$, and $\bar n \in \mathbb{N}$, the following bounds hold for all $t\geq 0$:
\[
\|\nabla^n(\bar R_k(t) - \mathcal{I}_k(t))\|_{L^\infty} \leq \varepsilon_0 N^{-\alpha}_{1,k}(t^{-\frac{n}{2} +\alpha}+1), \qquad  n=1,2,\dots,\bar n,
\]
\[
\|\bar R_k(t) - \mathcal{I}_k(t)\|_{L^\infty} \lesssim_\varepsilon \varepsilon_0N_{1,k}^{-\alpha}+\mathbbm{1}_{t\leq t_k}N_{J_d, k}^{\varepsilon}, \qquad \text{if} \qquad d=2,
\]
\[
\|\bar R_k(t) - \mathcal{I}_k(t)\|_{L^\infty} \lesssim \varepsilon_0N_{1,k}^{-\alpha}+\mathbbm1_{t\leq t_k}, \qquad \text{if} \qquad d\geq 3,
\]
provided $A$ and $b$ are large enough.
\end{proposition}
\begin{proof}
First we obtain a rough estimate which we will use on the time interval $[0,t_k]$.
\[
\begin{split}
\|\nabla^n\mathcal{I}_k\|_{L^\infty} &\lesssim \int_0^{t} \sum_j N_{j,k+1}^2 e^{-2N_{j,k+1}
^2s} \|\mathbb Q \nabla^n (a_{j,k+1}^2(x) \theta_j\otimes \theta_j)\|_{L^\infty} \, ds,\\
&\lesssim_\varepsilon \int_0^{t} \sum_j N_{j,k+1}^2 e^{-2N_{j,k+1}
^2s} \|\nabla^{n}(a_{j,k+1}^2)\|_{C^\varepsilon} \, ds,\\
&\lesssim_\varepsilon N_{J_d, k}^{n+\varepsilon},
\end{split}
\]
for any $\varepsilon>0$. In the case $d\geq 3$ the above bound can be improved. Indeed, in this case $N_{j,k}=N_{1,k}$ for all $j$. Then the summation in $j$ becomes, by \eqref{a-convex-integration-identity},
\[
\begin{split}
\nabla^n\mathbb Q\sum_j (a_{j,k+1}^2 \theta_j\otimes \theta_j) &=
\nabla^n\mathcal R\mathbb P\div \left(2\mathcal D \sum_j N_{1,k}\psi_{j,k}+p\Id\right)\\
&=2\nabla^n\mathcal R \Delta \mathbb P \sum_j N_{1,k}\psi_{j,k}\\
&=2\nabla^n\mathcal R \Delta \mathbb P \sum_j N_{1,k}^{-1}\phi_k*(a_{j,k}(x)\varphi_{j,k}(x)\theta_j\sin(N_{1,k}\eta_j\cdot x))\\
&=2\nabla^n\mathcal R \Delta \sum_j N_{1,k}^{-1}\phi_k*(a_{j,k}\varphi_{j,k}\theta_j\sin(N_{1,k}\eta_j\cdot x))\\
&\quad - 2\nabla^n\mathcal R \nabla \sum_j N_{1,k}^{-1}\phi_k*(\nabla(a_{j,k}\varphi_{j,k})\cdot \theta_j\sin(N_{1,k}\eta_j\cdot x)),
\end{split}
\]
which we can estimate without any loss in $L^\infty$ since $\mathcal R\Delta$ is a local operator: 
\[
\Big\|\nabla^n\mathbb Q\sum_j (a_{j,k+1}^2 \theta_j\otimes \theta_j)\Big\|_\infty \lesssim N_{J_d, k}^n.
\]
Hence for $d\geq 3$ we obtain
\[
\begin{split}
\|\nabla^n\mathcal{I}_k\|_\infty &\lesssim \int_0^{t}  N_{1,k+1}^2 e^{-2N_{1,k+1}^2s} \Big\| \nabla^n\mathbb Q \sum_j(a_{j,k+1}^2 \theta_j\otimes \theta_j)\Big\|_\infty \, ds\\
&\lesssim \int_0^{t}  N_{1,k+1}^2 e^{-2N_{1,k+1}^2s} N_{J_d, k}^n \, ds\\
&\lesssim N_{J_d, k}^n.
\end{split}
\]

Since we also have
\[
\|\nabla^n\bar R_k(t)\|_\infty\lesssim \sum_j N_{j,k}^{n}e^{-N_{j,k}^2t}\lesssim N_{J_d,k}^{n},
\]
we obtain the desired bounds for $t\in [0,t_k]$:
\begin{equation} \label{eq:bound_before_t_k_2D}
\|\nabla^n(\bar R_k(t) - \mathcal{I}_k(t))\|_\infty  \lesssim_\varepsilon N_{J_d, k}^{n+\varepsilon}, \qquad \text{if} \qquad d=2,
\end{equation}
and
\begin{equation} \label{eq:bound_before_t_k_3D}
\|\nabla^n(\bar R_k(t) - \mathcal{I}_k(t))\|_\infty  \lesssim_\varepsilon N_{J_d, k}^n, \qquad \text{if} \qquad d\geq 3.
\end{equation}

For $t\geq t_k$ we will get improved estimates. We split the integral $\mathcal{I}_k$ into
\begin{equation}\notag
\begin{split}
\mathcal{I}_k(t) &=-\int_0^{t_k} e^{(t-s)\Delta }\sum_j N_{j,k+1}^2 e^{-2N_{j,k+1}
^2s} \mathbb Q (a_{j,k+1}^2(x) \theta_j\otimes \theta_j) \, ds\\
&\quad-\int_{t_k}^t e^{(t-s)\Delta }\sum_j N_{j,k+1}^2 e^{-2N_{j,k+1}
^2s} \mathbb Q (a_{j,k+1}^2(x) \theta_j\otimes \theta_j) \, ds\\
&=:\mathcal{I}_k^{<t_k} + \mathcal{I}_k^{>t_k}.
\end{split}
\end{equation}
We shall show that the first term is close to $\bar R_k$ while the second one being a lower order error term. More precisely, define
\[
\tilde{\mathcal{I}}_k^{<t_k} = -e^{(t-t_k)\Delta}\int_0^{t_k} \sum_j N_{j,k+1}^2 e^{-2N_{j,k+1}
^2s} \mathbb Q  \left(a_{j,k+1}^2(x) \theta_j\otimes \theta_j\right) \, ds.
\]
Note that, thanks to \eqref{eq:a_bounds},
\[
\begin{split}
\| \nabla^n(e^{t\Delta} -\Id) \mathbb Q (a^2_{j,k+1}\theta_j\otimes \theta_j) \|_\infty &\lesssim t \| \nabla^n\Delta \mathbb Q (a^2_{j,k+1}\theta_j\otimes \theta_j)\|_{L^\infty}\\
& \lesssim_\varepsilon t N_{J_d,k}^{n+2+\varepsilon }.
\end{split}
\]
Then we estimate, for $t \geq t_k$,
\[
\begin{split}
\| \nabla^n(\mathcal{I}_k^{<t_k} - \tilde{\mathcal{I}}_k^{<t_k})\|_\infty &=
\Big\|e^{(t-t_k)\Delta}\int_0^{t_k} e^{(t_k-s)\Delta }\sum_j N_{j,k+1}^2 e^{-2N_{j,k+1}
^2s} \mathbb Q \left(a_{j,k+1}^2(x) \theta_j\otimes \theta_j\right) \, ds \\
& \quad -e^{(t-t_k)\Delta}\int_0^{t_k} \sum_j N_{j,k+1}^2 e^{-2N_{j,k+1}
^2s} \mathbb Q \left(a_{j,k+1}^2(x) \theta_j\otimes \theta_j\right) \, ds\Big\|_{L^\infty}\\
& \lesssim_\varepsilon \int_0^{t_k} \sum_j N_{j,k+1}^2 e^{-2N_{j,k+1}
^2s}  t_k N_{J_d,k}^{n+2+\varepsilon}  \, ds\\
&\lesssim t_k N_{J_d,k}^{n+2+\varepsilon},
\end{split}
\]
where we used the fact that $\|e^{t\Delta}f\|_{L^\infty}\leq \|f\|_{L^\infty}$. 

On the other hand, using \eqref{eq:Q_identities},
\begin{align}\label{eq:quadratic_identity}
\frac{1}{2}\sum_j  \mathbb Q \left(a_{j,k+1}^2(x) \theta_j\otimes \theta_j\right) &= 2\sum_j N_{j,k} \nabla\odot\mathbb P \psi_{j,k}\nonumber\\
&= \sum_j N_{j,k} ( \newD  \psi_{j,k}+(\Id-2\Delta^{-1}\nabla\otimes\nabla)\div\psi_{j,k}),
\end{align}
by \eqref{a-convex-integration-identity}, we have
\[
\begin{split}
\tilde{\mathcal{I}}_k^{<t_k} &=-e^{(t-t_k)\Delta}\int_0^{t_k} \sum_j N_{j,k+1}^2 e^{-2N_{j,k+1}
^2s} \mathbb Q  \left(a_{j,k+1}^2(x) \theta_j\otimes \theta_j\right) \, ds\\
&=-\frac{1}{2}
e^{(t-t_k)\Delta}\sum_j (1- e^{-2N_{j,k+1}
^2t_k} ) \mathbb Q \left(a_{j,k+1}^2(x) \theta_j\otimes \theta_j\right)\\
&=-\frac{1}{2}e^{(t-t_k)\Delta} \sum_j  \mathbb Q \left(a_{j,k+1}^2(x) \theta_j\otimes \theta_j\right) +\frac{1}{2}e^{(t-t_k)\Delta}\sum_j e^{-2N_{j,k+1}
^2t_k} \mathbb Q \left(a_{j,k+1}^2(x) \theta_j\otimes \theta_j\right)\\
&=-\sum_j N_{j,k} e^{(t-t_k)\Delta} \newD  \psi_{j,k} + \frac{1}{2}e^{(t-t_k)\Delta}\sum_j e^{-2N_{j,k+1}
^2t_k} \mathbb Q \left(a_{j,k+1}^2(x) \theta_j\otimes \theta_j\right)\\
&\qquad-(\Id-2\frac{\nabla\otimes\nabla}{\Delta})p\\
&= -\sum_j N_{j,k} e^{-N_{j,k}^2t} \newD  \psi_{j,k} - \sum_j N_{j,k} (e^{-N_{j,k}^2(t-t_k)}-e^{-N_{j,k}^2t}) \newD \psi_{j,k} \\
&\qquad + \sum_j N_{j,k} (e^{-N_{j,k}^2(t-t_k)}-e^{(t-t_k)\Delta}) \newD  \psi_{j,k} \\
&\qquad + \frac{1}{2}e^{(t-t_k)\Delta}\sum_j e^{-2N_{j,k+1}
^2t_k} \mathbb Q \left(a_{j,k+1}^2(x) \theta_j\otimes \theta_j\right) -(\Id-2\frac{\nabla\otimes\nabla}{\Delta})p\\
&= \bar R_k + \tilde{\mathcal{I}}_k^1 +\tilde{\mathcal{I}}_k^2 + \tilde{\mathcal{I}}_k^3 -(\Id-2\frac{\nabla\otimes\nabla}{\Delta})p.
\end{split}
\]
The term $\tilde{\mathcal{I}}_k^2$ can be expressed as the commutator
\[\tilde{\mathcal{I}}_k^2= - \phi_k*\sum_j N_{j,k}^{-1}  \newD [e^{(t-t_k)\Delta},a_{j,k}\varphi_{j,k}](\theta_j\sin(N_{j,k}\eta_j\cdot x))\]
between the heat kernel with multiplication by the coefficients. Meanwhile, the term with $p$ results from \eqref{eq:quadratic_identity}; thus
\begin{align*}
    -p&=\sum_jN_{j,k}e^{(t-t_k)\Delta}\div\psi_{j,k}\\
    &=\phi_k*\sum_jN_{j,k}^{-1}e^{(t-t_k)\Delta}\big(\theta_j\cdot\nabla(a_{j,k}\varphi_{j,k})\theta_j\sin(N_{j,k}\eta_j\cdot x)\big).
\end{align*}
By \eqref{heat-decay-estimate} and the fact that $\Delta^{-1}\nabla\otimes\nabla$ is Calder\'on--Zygmund, we have
\begin{align*}
    &\|\nabla^n (\Id-2\frac{\nabla\otimes\nabla}{\Delta})p\|_\infty\\
    &\qquad\lesssim \sum_j \Big(\|\nabla(a_{j,k}\varphi_{j,k})\|_\infty N_{j,k}^{n-1+\varepsilon}e^{-N_{j,k}^2(t-t_k)/4}+ N_{j,k}^{n-m+\varepsilon}\|\nabla^{m}(a_{j,k}\varphi_{j,k})\|_\infty\Big)\\
    &\qquad\lesssim \sum_j\big(M_{j,k}N_{j,k}^{n-1+\varepsilon}e^{-N_{j,k}^2(t-t_k)/4}+N_{j,k}^{n-m+\varepsilon}M_{j,k}^{m}\big),
\end{align*}
which is small enough after taking $m$ sufficiently large.

We continue to show that $\tilde{\mathcal{I}}_k^1$--$\tilde{\mathcal{I}}_k^3$ are lower order terms. First,
\begin{equation}
\begin{split}
\|\nabla^n\tilde{\mathcal{I}}_k^1\|_{L^\infty} &\lesssim \sum_j N_{j,k}^{n}e^{-N_{j,k}^2t}(e^{N_{j,k}^2t_k}-1)\lesssim \sum_j N_{j,k}^{n+2}t_ke^{-N_{j,k}^2t},
\end{split}
\end{equation}
which is acceptable as long as $t_k\ll N_{J_d,k}^{-2}$. 
By Lemma~\ref{l:commutator},
\begin{align*}
    \|\nabla^n\tilde{\mathcal I}_k^2\|_\infty&\lesssim \sum_jN_{j,k}^{n}\left(M_{j,k}N_{j,k}^{-1}e^{-N_{j,k}^2t/4} + N_{j,k}^{-m}M_{j,k}^m\right),
\end{align*}
which can be small enough again by taking $m$ sufficiently large. (Note the non-decaying term, which we can take extremely small.) Finally,
\begin{equation}\notag
\begin{split}
\|\nabla^n\tilde{\mathcal{I}}_k^3\|_{L^\infty} &=\frac12\Big\|\nabla^ne^{(t-t_k)\Delta}\sum_j e^{-2N_{j,k+1}
^2t_k}\mathbb Q  \left(a_{j,k+1}^2 \theta_j\otimes \theta_j\right)\Big\|_{L^\infty}\\
&\lesssim_\varepsilon \Big\|\sum_j  e^{-2N_{j,k+1}
^2t_k} \nabla^n  \left(a_{j,k+1}^2 \theta_j\otimes \theta_j\right)\Big\|_{C^\varepsilon}\\
&\lesssim \sum_j N_{J_d, k}^{n+\varepsilon} e^{-2N_{j,k+1}^2t_k},
\end{split}
\end{equation}
where we used \eqref{eq:a_bounds} and \eqref{time-sequence}. 

The last residual term enjoys the following estimate:
\[
\begin{split}
\|\nabla^n\mathcal{I}_k^{>t_k}\|_{L^\infty} &\leq \left\|\int_{t_k}^t \sum_j N_{j,k+1}^2 e^{-2N_{j,k+1}
^2s} \nabla^n \mathbb Q \left(a_{j,k+1}^2 \theta_j\otimes \theta_j\right) \, ds\right\|_{L^\infty}\\
&\lesssim_\varepsilon \int_{t_k}^t \sum_j N_{j,k+1}^2 e^{-2N_{j,k+1}
^2s} N_{J_d, k}^{n+\varepsilon}  \, ds\\
&\lesssim \sum_j N_{J_d, k}^{n+\varepsilon} e^{-2N_{j,k+1}^2t_k},
\end{split}
\]
where we used \eqref{eq:a_bounds} and \eqref{time-sequence}.

To conclude we combine all the terms and obtain, using the triangle inequality,
for $t\geq t_k$,
\[
\begin{split}
\|\nabla^n(\bar R_k(t) - \mathcal{I}_k(t))\|_\infty & \leq \|\nabla^n(\bar R_k - \tilde{\mathcal{I}}_k^{<t_k})\|_\infty + \|\nabla^n( \tilde{\mathcal{I}}_k^{<t_k} - \mathcal{I}_k)\|_\infty\\
&= \|\nabla^n(\tilde{\mathcal{I}}_k^1 +\tilde{\mathcal{I}}_k^2 +\tilde{\mathcal{I}}_k^3 -(\Id-2\frac{\nabla\otimes\nabla}{\Delta})p)\|_\infty \\
&\quad   +\| \nabla^n(\tilde{\mathcal{I}}_k^{<t_k} - \mathcal{I}_k^{<t_k})\|_\infty + \|\nabla^n\mathcal{I}_k^{>t_k}\|_\infty \\
&\lesssim_{\varepsilon,m} \sum_j \big( N_{j,k}^{n}e^{-N_{j,k}^2t/4}( N_{j,k}^{2}t_k + M_{j,k}N_{j,k}^{-1}) + N_{j,k}^{n-m}M_{j,k}^m \big)\\
&\qquad +\sum_j\big(M_{j,k}N_{j,k}^{n-1}e^{-N_{j,k}^2(t-t_k)/4}+N_{j,k}^{n-m}M_{j,k}^{m}\big)\\
&\qquad + \sum_j N_{J_d, k}^{n+\varepsilon} e^{-2N_{j,k+1}^2t_k} +t_k N_{J_d,k}^{n+2+\varepsilon}.
\end{split}
\]
Choosing $\gamma=\frac12$ in the definition of $M_{j,k}$ \eqref{M_definition} and $t_k= N_{J_d,k}^{-4}$,
we can ensure that
\[
\begin{split}
N_{j,k}^{n}e^{-N_{j,k}^2t/4}(N_{j,k}^{2}t_k + M_{j,k}N_{j,k}^{-1} )&\leq  N_{j,k}^{n}e^{-N_{j,k}^2t/4} \times 2N_{j,k}^{-\frac12}\\
&\lesssim N_{j,k}^{n-\frac12}e^{-N_{j,k}^2t/4},
\end{split}
\]
with similar bounds for the remaining terms. 
Also recall that, for $t \in [0,t_k]$,
\[
\begin{split}
\|\nabla^n(\bar R_k(t) - \mathcal{I}_k(t))\|_{L^\infty_{x,t}}  &\lesssim_\varepsilon N_{J_d, k}^{n+\varepsilon}\\
& = t_k^{-(n+\varepsilon)/4}.
\end{split}
\]
Then for any $\varepsilon_0>0$ and $0<\alpha < \frac{1}{10}$, we can choose $A$, $b$, and $m$ large enough and $\varepsilon$ small enough so that
\[
\|\nabla^n(\bar R_k(t) - \mathcal{I}_k(t))\|_\infty \leq_n \varepsilon_0 N^{-\alpha}_{1,k}(t^{-\frac{n}{2} +\alpha}+1), \qquad n=1,2,\dots, \bar n
\]
for all $t\geq 0$.

For $n=0$ the above bound breaks down on the time interval $[0,t_k]$, and we apply estimates \eqref{eq:bound_before_t_k_2D} and \eqref{eq:bound_before_t_k_3D} instead.
\end{proof}

Applying Proposition \ref{barv-I_estimate_proposition}, we further show that $R_k-\bar R_k$ is small as follows. Consequently, $v_k-\bar v_k$ is also small since $v_k=\div R_k$ and $\bar v_k=\div\bar R_k$.

\begin{proposition}\label{difference_estimate_proposition}
For all $\alpha>0$ sufficiently small, $\epsilon_0>0$, and $\bar n \in \mathbb{N}$, parameters $A$ and $b$ can be taken sufficiently large such that
\[
\|\nabla^n(R_k(t)- \bar R_k(t) )\|_{L^\infty} \leq_n \epsilon_0 N_{1,k}^{-\alpha} (t^{-\frac{n}{2}+\alpha}+1), \qquad n=1,2,\dots, \bar n,
\]
\[
\|R_k(t)- \bar R_k(t)\|_{L^\infty} \lesssim_\varepsilon \varepsilon_0N_{1,k}^{-\alpha}+\mathbbm{1}_{t\leq t_k}N_{J_d, k}^{\varepsilon}, \qquad \text{if} \qquad d=2,
\]
\[
\|R_k(t)- \bar R_k(t)\|_{L^\infty} \lesssim \varepsilon_0N_{1,k}^{-\alpha}+\mathbbm1_{t\leq t_k}, \qquad \text{if} \qquad d\geq 3.
\]
\end{proposition}

\begin{proof}
We start by noting that
\begin{equation}\notag
\begin{split}
\Delta \psi_{j,k}(x)&=-a_{j,k}(x)\varphi_{j,k}(x)\sin(N_{j,k}\eta_j\cdot x)\theta_j\\
&\quad+2N_{j,k}^{-1}\nabla(a_{j,k}\varphi_{j,k})\cdot\eta_j\cos(N_{j,k}\eta_j\cdot x)\theta_j\\
&\quad +N_{j,k}^{-2}\Delta (a_{j,k}\varphi_{j,k})\sin(N_{j,k}\eta_j\cdot x)\theta_j\\
&\quad +\Delta(\psi_{j,k}-N_{j,k}^{-2}a_{j,k}\varphi_{j,k}\theta_j\sin(N_{j,k}\eta_j\cdot x)),
\end{split}
\end{equation}
with the first term the main contribution and the other three lower order errors. We split 
\begin{equation}\label{bar-v-k1}
\bar v_{k+1}\otimes \bar v_{k+1}
=\sum_j\bar v_{j,k+1}\otimes \bar v_{j,k+1}+\sum_{j\neq j'}\bar v_{j,k+1}\otimes \bar v_{j',k+1}.
\end{equation}
We compute the unidirectional terms
\begin{equation}\label{unidirectional_terms}
\begin{split}
\sum_j\bar v_{j,k+1}\otimes \bar v_{j,k+1}&=\sum_{j} N_{j,k+1}^2 e^{-2N_{j,k+1}
^2t} \Delta \psi_{j,k+1}\otimes \Delta \psi_{j,k+1}\\
&=\sum_{j} N_{j,k+1}^2 e^{-2N_{j,k+1}
^2t} a_{j,k+1}^2 \theta_j\otimes \theta_j\\
&\hspace{-8em}+ \sum_{j} N_{j,k+1}^2 e^{-2N_{j,k+1}
^2t} a_{j,k+1}^2(\varphi^2_{j,k+1} \sin^2(N_{j,k+1}\eta_j\cdot x) -1)\theta_j\otimes \theta_j + \mathfrak{E}_k^1,
\end{split}
\end{equation}
where $\mathfrak{E}_k^1$ contains lower order error terms involving derivatives of $a_{j,k+1}\varphi_{j,k+1}$ as well as terms coming from removing the mollifier.

The residual terms with $j\ne j'$ in \eqref{bar-v-k1} are
\begin{equation} \label{different_directions_terms}
\begin{split}
\sum_{j\neq j'}&\bar v_{j,k+1}\otimes \bar v_{j',k+1}\\
=&\sum_{j\neq j'} N_{j,k+1}N_{j',k+1} e^{-N_{j,k+1}
^2t-N_{j',k+1}^2t} \Delta \psi_{j,k+1}\otimes \Delta \psi_{j',k+1}\\
=&\sum_{j\neq j'} N_{j,k+1}N_{j',k+1} e^{-N_{j,k+1}
^2t-N_{j',k+1}^2t} \\
&\cdot a_{j,k+1}a_{j',k+1}\varphi_{j,k+1}\varphi_{j',k+1} \sin(N_{j,k+1}\eta_j\cdot x)\sin(N_{j',k+1}\eta_{j'}\cdot x)\theta_j\otimes \theta_{j'}\\
&+\mathfrak{E}_k^2,
\end{split}
\end{equation}
where $\mathfrak{E}_k^2$ contains lower order error terms involving derivatives of $a_{j,k+1}a_{j',k+1}\varphi_{j,k+1}\varphi_{j',k+1}$ as well as terms coming from removing the mollifier. We decompose the sum of all the lower order terms as follows:
\begin{equation}\label{E1+E2terms}
\begin{split}
\mathfrak{E}_k^1 +\mathfrak{E}_k^2 &= \sum_{j} N_{j,k+1}^2 e^{-2N_{j,k+1}
^2t} E_{j,k}\\
&+\sum_{j\neq j'} N_{j,k+1}N_{j',k+1} e^{-N_{j,k+1}
^2t-N_{j',k+1}^2t}E_{j,j',k}\\
&+\sum_{j,j'}N_{j,k+1}N_{j',k+1}e^{-N_{j,k+1}^2t-N_{j',k+1}^2t}F_{j,j',k},
\end{split}
\end{equation}
where $E_{j,k}$ and $E_{j,j',k}$ contain all the terms involving derivatives of $a_{j,k+1}\varphi_{j,k+1}$ coming from the sums $\sum_j$ and $\sum_{j\neq j'}$ respectively, while $F_{j,j',k}$ comes from removing the mollifier. The term $E_{j,k}$ is specified below.
\begin{equation}\notag
\begin{split}
E_{j,k}&=-2N_{j,k+1}^{-1}a_{j,k+1}\varphi_{j,k+1}\nabla(a_{j,k+1}\varphi_{j,k+1})\cdot\eta_j\sin(2N_{j,k+1}\eta_{j}\cdot x)\theta_j\otimes \theta_j\\
&\quad+4N_{j,k+1}^{-2}(\nabla(a_{j,k+1}\varphi_{j,k+1})\cdot\eta_j)^2\cos^2(N_{j,k+1}\eta_{j}\cdot x)\theta_j\otimes \theta_j\\
&\quad-2N_{j,k+1}^{-2}a_{j,k+1}\varphi_{j,k+1}\Delta(a_{j,k+1}\varphi_{j,k+1})\sin^2(N_{j,k+1}\eta_{j}\cdot x)\theta_j\otimes \theta_j\\
&\quad+2N_{j,k+1}^{-3}\Delta(a_{j,k+1}\varphi_{j,k+1})\nabla(a_{j,k+1}\varphi_{j,k+1})\cdot\eta_j\sin(2N_{j,k+1}\eta_{j}\cdot x)\theta_j\otimes \theta_j\\
&\quad+N_{j,k+1}^{-4}(\Delta(a_{j,k+1}\varphi_{j,k+1}))^2\sin^2(N_{j,k+1}\eta_{j}\cdot x)\theta_j\otimes \theta_j\\
&=:E^1_{j,k}+E^2_{j,k}+E^3_{j,k}+E^4_{j,k}+E^5_{j,k}.
\end{split}
\end{equation}
The term $E_{j,j',k}$ collects all the residual errors in the sum of $j\neq j'$, which are lower order terms. 
The mollifier error is given by
\begin{align*}
    F_{j,j',k}&=\Delta(\psi_{j,k+1}-N_{j,k+1}^{-2}a_{j,k+!}\varphi_{j,k+1}\theta_j\sin(N_{j,k+1}\eta_j\cdot x))\otimes\Delta\psi_{j',k+1}\\
    &\quad+\Delta(N_{j,k+1}^{-2}a_{j,k+1}\varphi_{j,k+1}\theta_j\sin(N_{j,k+1}\eta_j\cdot x))\\
    &\qquad\otimes \Delta(\psi_{j',k+1}-N_{j',k+1}^{-2}a_{j',k+1}\varphi_{j',k+1}\theta_{j'}\sin(N_{j',k+1}\eta_{j'}\cdot x)).
\end{align*}
Combining \eqref{unidirectional_terms} and \eqref{different_directions_terms} we obtain
\[
\begin{split}
R_k =& -\int_0^{t} e^{(t-s)\Delta }\sum_j N_{j,k+1}^2 e^{-2N_{j,k+1}
^2s} \mathbb Q (a_{j,k+1}^2 \theta_j\otimes \theta_j) \, ds \\
& -\int_0^{t} e^{(t-s)\Delta }\sum_j N_{j,k+1}^2 e^{-2N_{j,k+1}
^2s} \mathbb Q (a_{j,k+1}^2(\varphi^2_{j,k+1} \sin^2(N_{j,k+1}\eta_j\cdot x) -1) \theta_j\otimes \theta_j) \, ds \\
&- \int_0^{t} e^{(t-s)\Delta }\sum_{j\ne j'}N_{j,k+1}N_{j',k+1} e^{-N_{j,k+1}
^2s-N_{j',k+1}^2s} \\
&\qquad \quad\cdot \mathbb Q (a_{j,k+1}a_{j',k+1}\varphi_{j,k+1}\varphi_{j',k+1} \sin(N_{j,k+1}\eta_j\cdot x)\sin(N_{j',k+1}\eta_{j'}\cdot x)\theta_j\otimes \theta_{j'}) \, ds\\
&-\int_0^{t} e^{(t-s)\Delta }\mathbb Q(\mathfrak{E}_k^1 +\mathfrak{E}_k^2) \, ds\\
=& \mathcal{I}_k+\mathcal{J}_k^1+\mathcal{J}_k^2 + \mathcal{E}_k.
\end{split}
\]
Among the terms on the right hand side of the equation, the main term $\mathcal{I}_k \approx \bar R_k$ as was already shown in Proposition~\ref{barv-I_estimate_proposition}. It remains to prove that the other ones obey acceptable small upper bounds.


\noindent
{\bf Residual terms $\mathcal{J}_k^1$ and $\mathcal{J}_k^2$.}

Now we proceed with estimates of $\mathcal{J}_k^1$ and $\mathcal{J}_k^2$. First, using the fact that a parallel shear flow is a stationary solution of the Euler equation, as $\eta_j \cdot \theta_j=0$,
\begin{multline}\label{div-osc}
\div (a_{j,k+1}^2(x)(\varphi^2_{j,k+1} \sin^2(N_{j,k+1}\eta_j\cdot x) -1) \theta_j\otimes \theta_j)\\
\quad= \nabla a_{j,k+1}^2 \cdot \theta_j (\varphi^2_{j,k+1} \sin^2(N_{j,k+1}\eta_j\cdot x) -1)\theta_j.
\end{multline}
 
By Lemma~\ref{l:oscillation_estimate}, we have
\begin{multline}\label{est-anti-div-osc}
\big\|e^{(t-s)\Delta }\mathcal R\mathbb P\big(\nabla a_{j,k+1}^2 \cdot \theta_j (\varphi^2_{j,k+1} \sin^2(N_{j,k+1}\eta_j\cdot x) -1)\big)\big\|_{C^\varepsilon}\\
\lesssim_\varepsilon M_{j,k+1}^{-1}\|\nabla(a_{j,k+1}^2)\|_{C^\varepsilon}e^{-M_{j,k+1}^2(t-s)/4}+M_{j,k+1}^{-m}\|\nabla(a_{j,k+1}^2)\|_{C^{m,\varepsilon}},
\end{multline}
and, for $n\geq1$,
\begin{equation}\label{est-derivative-anti-div-osc}
\begin{split}
&\big\|\nabla^n e^{(t-s)\Delta }\mathcal R\mathbb P\big(\nabla a_{j,k+1}^2 \cdot \theta_j (\varphi^2_{j,k+1} \sin^2(N_{j,k+1}\eta_j\cdot x) -1)\big)\big\|_{C^\varepsilon}\\
&\quad\lesssim_\varepsilon N_{j,k+1}^{n-1}\|\nabla(a_{j,k+1}^2)\|_{C^\varepsilon}e^{-M_{j,k+1}^2(t-s)/4}+M_{j,k+1}^{-m+n}\|\nabla(a_{j,k+1}^2)\|_{C^{m,\varepsilon}}\\
&\quad\lesssim N_{j,k+1}^{n-1}N_{J_d,k}^{1+\varepsilon}e^{-M_{j,k+1}^2(t-s)/4}+M_{j,k+1}^{-m+n}N_{J_d,k}^{m+1+\varepsilon}
\end{split}
\end{equation}
using that $\varphi^2_{j,k+1}(x) \sin^2(N_{j,k+1}\eta_j\cdot x)-1$ is zero mean and $2\pi/M_{j,k+1}$-periodic.
Recalling $\mathbb Q=\mathcal R\mathbb P\div$, in view of \eqref{div-osc} and the estimate \eqref{est-anti-div-osc}, we obtain
\[
\begin{split}
\| \mathcal{J}_k^1 \|_{L^\infty} &= \Big\| \int_0^{t}  e^{(t-s)\Delta } \sum_{j} N_{j,k+1}^2 e^{-2N_{j,k+1}^2s} \\
&\qquad \quad \times \mathcal R \mathbb P \div \big(a_{j,k+1}^2 (\varphi^2_{j,k+1} \sin^2(N_{j,k+1}\eta_j\cdot x) -1)\theta_j\otimes \theta_j\big) \, ds\Big\|_\infty\\
&\lesssim_\varepsilon  \int_0^{t}  \sum_{j}  N_{j,k+1}^2 e^{-2N_{j,k+1}^2s}
\left(  M_{j,k+1}^{-1}\|\nabla(a_{j,k+1}^2)\|_{C^\varepsilon}e^{-M_{j,k+1}^2(t-s)/4} \right.\\
&\qquad \quad \left.+M_{j,k+1}^{-m}\|\nabla(a_{j,k+1}^2)\|_{C^{m,\varepsilon}} \right)\, ds\\
&\lesssim_\varepsilon  \int_0^{t}  \sum_{j}  N_{j,k+1}^2 e^{-2N_{j,k+1}^2s}
N_{J_d,k}^{1+\varepsilon}\big(  M_{j,k+1}^{-1}e^{-M_{j,k+1}^2(t-s)/4}+M_{j,k+1}^{-m}N_{J_d,k}^{m} \big)\, ds\\
&\lesssim \sum_{j}  
N_{J_d,k}^{1+\varepsilon}\big(  M_{j,k+1}^{-1}e^{-M_{j,k+1}^2t/4}+M_{j,k+1}^{-m}N_{J_d,k}^{m} \big).
\end{split}
\]

For $n\geq 1$, applying \eqref{div-osc} and the estimate \eqref{est-derivative-anti-div-osc} yields
\[
\begin{split}
\| \nabla^n \mathcal{J}_k^1 \|_{L^\infty} &= \Big\| \int_0^{t} \nabla^n e^{(t-s)\Delta } \sum_{j} N_{j,k+1}^2 e^{-2N_{j,k+1}^2s} \\
&\qquad \quad \times \mathcal R \mathbb P \div \big(a_{j,k+1}^2 (\varphi^2_{j,k+1} \sin^2(N_{j,k+1}\eta_j\cdot x) -1)\theta_j\otimes \theta_j\big) \, ds\Big\|_\infty\\
&\lesssim_\varepsilon  \int_0^{t}  \sum_{j}  N_{j,k+1}^2 e^{-2N_{j,k+1}^2s}
N_{J_d,k}^{1+\varepsilon}\big(  N_{j,k+1}^{n-1}e^{-M_{j,k+1}^2(t-s)/4}+M_{j,k+1}^{-m+n}N_{J_d,k}^{m} \big)\, ds\\
&\lesssim \sum_{j}
N_{J_d,k}^{1+\varepsilon}\big( N_{j,k+1}^{n-1} e^{-M_{j,k+1}^2t/4}+M_{j,k+1}^{-m+n}N_{J_d,k}^{m} \big).
\end{split}
\]

Now we focus on $\mathcal{J}_k^2$. 
Bounding $L^\infty$ by $C^\varepsilon$, using that $\mathbb Q$ is a Calder\'on--Zygmund operator, and estimating the heat propagator with \eqref{heat-decay-estimate}, we have
\begin{equation}
\begin{split}
&\|\nabla^n  \mathcal{J}_k^2(t) \|_{L^\infty}\\
&\quad=
\Big\| \int_0^{t} \nabla^n e^{(t-s)\Delta } \sum_{j\neq j'} N_{j,k+1}N_{j',k+1} e^{-N_{j,k+1}
^2s-N_{j',k+1}^2s}\\
&\qquad\times\mathbb Q(a_{j,k+1}a_{j',k+1}\varphi_{j,k+1}\varphi_{j',k+1} \sin(N_{j,k+1}\eta_j\cdot x)\sin(N_{j',k+1}\eta_{j'}\cdot x)\theta_j\otimes \theta_{j'}) \, ds \Big\|_\infty\\
&\quad\lesssim \int_0^{t} \sum_{j' < j} N_{j,k+1}N_{j',k+1} e^{-N_{j,k+1}
^2s-N_{j',k+1}^2s} N_{j,k+1}^{n+\varepsilon}e^{-\frac14 N_{j,k+1}
^2(t-s)}  \, ds\\
&\qquad +\int_0^{t} \sum_{j' < j} N_{j,k+1}N_{j',k+1} e^{-N_{j,k+1}
^2s-N_{j',k+1}^2s} N_{j,k+1}^{n-m}M_{j,k+1}^{m+\varepsilon}  \, ds\\
&\quad\eqcolon\mathcal{J}_k^{2,1} + \mathcal{J}_k^{2,2}.
\end{split}
\end{equation}
Computing the major term $\mathcal{J}_k^{2,1}$ explicitly we obtain
\begin{equation} \label{eq:NonlocalInteractions}
\begin{split}
\mathcal{J}_k^{2,1}&= \sum_{j'<j} e^{-\frac14N_{j,k+1}^2t}N^{n+1+\varepsilon}_{j,k+1}N_{j',k+1}\frac{1-e^{-(\frac34N_{j,k+1}^2+N_{j',k+1}^2)t}}{\frac34N_{j,k+1}^2+N_{j',k+1}^2}\\
& \lesssim \sum_{j'<j} e^{-\frac14N_{j,k+1}^2t}N^{n-1+\varepsilon}_{j,k+1}N_{j',k+1},
\end{split}
\end{equation}
while the non-decaying term $\mathcal{J}_k^{2,2}$ is very small again:
\[
\mathcal{J}_k^{2,2} \lesssim \sum_{j'<j} N_{j,k+1}^{-1}N_{j',k+1} N_{j,k+1}^{n-m}M_{j,k+1}^{m+\varepsilon}.
\]

\noindent
{\bf Residual term $\mathcal{E}_k$.}

Recall from \eqref{E1+E2terms} that

\[
\begin{split}
\mathcal{E}_k&=-\int_0^{t} e^{(t-s)\Delta } \mathbb Q(\mathfrak{E}_k^1 +\mathfrak{E}_k^2) \, ds\\
&= -\int_0^{t} e^{(t-s)\Delta }\sum_j N_{j,k+1}^2 e^{-2N_{j,k+1}
^2s} \mathbb Q E_{j,k}\, ds\\
&\quad -\int_0^{t} e^{(t-s)\Delta }\sum_{j\ne j'}N_{j,k+1}N_{j',k+1} e^{-N_{j,k+1}
^2s-N_{j',k+1}^2s} \mathbb Q E_{j,j',k}\, ds\\
&\quad-\int_0^{t} e^{(t-s)\Delta }\sum_{j,j'}N_{j,k+1}N_{j',k+1} e^{-N_{j,k+1}
^2s-N_{j',k+1}^2s} \mathbb Q F_{j,j',k}\, ds\\
&=:  \mathcal{E}_k^1+ \mathcal{E}_k^2 + \mathcal{E}_k^3.
\end{split}
\]
Since $\mathcal E^2_k$ collects the lower order terms from the interactions of different directions $j$ and $j'$, $\mathcal E^2_k$ obeys the same estimate as $\mathcal J^2_k$.

Next we estimate $\mathcal E^1_k$ by considering the items within it separately. In particular, we estimate the portion involving $E^1_{j,k}$ by taking advantage of the heat kernel and Lemma~\ref{l:commutator}, 
\begin{equation}\notag
\begin{split}
\Big\| &\nabla^n\int_0^t\sum_{j}N_{j,k+1}^2e^{-2N_{j,k+1}^2s}e^{(t-s)\Delta}\mathbb Q E^1_{j,k}\,ds \Big\|_\infty\\
&\lesssim \int_0^t\sum_{j}N_{j,k+1}^2e^{-2N_{j,k+1}^2s}N_{j,k+1}^{-1}\\
&\quad \times \big\|\nabla^ne^{(t-s)\Delta}\mathbb Q (a_{j,k+1}\varphi_{j,k+1}\nabla(a_{j,k+1}\varphi_{j,k+1})\cdot\eta_j\sin(2N_{j,k+1}\eta_j\cdot x)\theta_j\otimes \theta_j) \big\|_\infty\,ds\\
&\lesssim_\varepsilon \int_0^t\sum_{j}N_{j,k+1}e^{-2N_{j,k+1}^2s} N_{j,k+1}^{n+\varepsilon}e^{-2N_{j,k+1}^2(t-s)/4}\,ds\\
&\quad + \int_0^t\sum_{j}N_{j,k+1}e^{-2N_{j,k+1}^2s} N_{j,k+1}^{n-m}M_{j,k+1}^{m+1+\epsilon}\, ds\\
&\lesssim \sum_{j} N_{j,k+1}^{n-1+\varepsilon}e^{-\frac14 N_{j,k+1}^2t} (1-e^{-\frac74 N_{j,k+1}^2t}).
\end{split}
\end{equation}
Note that $E^2_{j,k}$ and $E^3_{j,k}$ are of the same order and both contain a zero-frequency part, it is thus not necessary to use the heat kernel decay. Instead we estimate them by using  
$\|e^{t\Delta} f\|_\infty\leq \|f\|_\infty$, obtaining
\begin{equation}\notag
\begin{split}
\Big\| \nabla^n &\int_0^t\sum_{j}N_{j,k+1}^2e^{-2N_{j,k+1}^2s}e^{(t-s)\Delta}\mathbb Q (E^2_{j,k}+E^3_{j,k})\,ds\Big\|_\infty\\
&\lesssim_\varepsilon\int_0^t\sum_{j}N_{j,k+1}^2e^{-2N_{j,k+1}^2s}\big\|\nabla^n \big(N_{j,k+1}^{-2}(\nabla(a_{j,k+1}\varphi_{j,k+1})\cdot\eta_j)^2\big)\big\|_{C^\varepsilon}\,ds\\
&\quad+\int_0^t\sum_{j}N_{j,k+1}^2e^{-2N_{j,k+1}^2s}\big\|\nabla^n  \big(N_{j,k+1}^{-2}a_{j,k+1}\varphi_{j,k+1}\Delta(a_{j,k+1}\varphi_{j,k+1})\big)\big\|_{C^\varepsilon}\,ds\\
&\lesssim \int_0^t\sum_{j}e^{-2N_{j,k+1}^2s}M_{j,k+1}^{n+2+\varepsilon}\,ds\\
&\lesssim \sum_{j}N_{j,k+1}^{-2}M_{j,k+1}^{n+2+\varepsilon}.
\end{split}
\end{equation}
Noting that $E^4_{j,k}$ obeys the same estimate as $E^1_{j,k}$ and $E^5_{j,k}$ obeys the same estimate as $E^2_{j,k}$, we conclude that
\[\|\nabla^n \mathcal E^1_k\|_\infty\lesssim \sum_{j} \big(N_{j,k+1}^{n-1+\varepsilon}e^{-\frac14 N_{j,k+1}^2t} +N_{j,k+1}^{-2}M_{j,k+1}^{n+2+\varepsilon}\big).
\]

Toward estimating $\mathcal E_k^3$, we recall that
\begin{align*}
    F_{j,j',k}&=\Delta(\psi_{j,k+1}-N_{j,k+1}^{-2}a_{j,k+1}\varphi_{j,k+1}\theta_j\sin(N_{j,k+1}\eta_j\cdot x))\otimes\Delta\psi_{j',k+1}\\
&\quad+\Delta(N_{j,k+1}^{-2}a_{j,k+1}\varphi_{j,k+1}\theta_j\sin(N_{j,k+1}\eta_j\cdot x))\\
    &\qquad\otimes \Delta(\psi_{j',k+1}-N_{j',k+1}^{-2}a_{j',k+1}\varphi_{j',k+1}\theta_{j'}\sin(N_{j',k+1}\eta_{j'}\cdot x)),
\end{align*}
and use the fact that $\|\nabla^n(f-\phi_{k+1}*f)\|_\infty\lesssim \ell_{k+1}\|\nabla^{n+1}f\|_\infty$ along with \eqref{eq:a_bounds} to obtain
\begin{align*}
    &\|\nabla^n\Delta(\psi_{j,k+1}-N_{j,k+1}^{-2}a_{j,k+1}\varphi_{j,k+1}\theta_j\sin(N_{j,k+1}\eta_j\cdot x))\|_\infty\\
    &\qquad\lesssim N_{j,k+1}^{-2}\ell_{k+1}\|\nabla^{3+n}(a_{j,k+1}\varphi_{j,k+1}\sin(N_{j,k+1}\eta_j\cdot x))\|_\infty\\
    &\qquad\lesssim N_{j,k+1}^{1+n}\ell_{k+1}.
\end{align*}
Similarly,
\begin{align*}
    \|\nabla^n\Delta\psi_{j,k+1}\|_\infty+\|\nabla^n\Delta(N_{j,k+1}^{-2}a_{j,k+1}\varphi_{j,k+1}\theta_j\sin(N_{j,k+1}\eta_j\cdot x))\|_\infty&\lesssim N_{j,k+1}^n.
\end{align*}
It follows by the Leibniz rule that
\begin{align*}
    \|\nabla^n\mathbb Q F_{j,j',k}\|_\infty&\lesssim N_{j\vee j',k+1}^{n+1+\varepsilon}\ell_{k+1},
\end{align*}
and therefore
\begin{align*}
    \|\nabla^n\mathcal E_k^3\|_\infty&\lesssim \sum_{j,j'}N_{j,k+1}N_{j',k+1}N_{j\vee j',k+1}^{-2}N_{j\vee j',k+1}^{n+1+\varepsilon}\ell_{k+1}\lesssim\sum_jN_{j,k+1}^{n+1+\varepsilon}\ell_{k+1}.
\end{align*}

\noindent
{\bf Final bound.}
By the triangle inequality,
\[
\|\nabla^n(R_k-\bar R_k)\|_\infty \leq \|\nabla^n(R_k-\mathcal{I}_k)\|_\infty + \|\nabla^n(\bar R_k-\mathcal{I}_k)\|_\infty .
\]
By Proposition~\ref{barv-I_estimate_proposition}, for any $\varepsilon_0 >0$, $\alpha \in(0, \frac{1}{10})$, and $\bar n \in \mathbb{N}$, 
\[
\|\nabla^n(\bar R_k - \mathcal{I}_k)\|_\infty \leq \varepsilon_0 N^{-\alpha}_{1,k+1}(t^{-\frac{n}{2} +\alpha}+1), \qquad n=1,2,\dots,\bar n,
\]
for $A$ and $b$ large enough.

Since
\[
R_k-\mathcal{I}_k=\mathcal{J}_k^1+\mathcal{J}_k^2 + \mathcal{E}_k,
\]
it remains to show such a bound for the residual terms $\mathcal{J}_k^1+\mathcal{J}_k^2 + \mathcal{E}_k$. Indeed, by the triangle inequality we have
\[ 
\begin{split}
\|\nabla^n(\mathcal{J}_k^1+\mathcal{J}_k^2 + \mathcal{E}_k)\|_\infty&\leq \|\nabla^n\mathcal{J}_k^1\|_\infty+\|\nabla^n\mathcal{J}_k^2\|_\infty+\|\nabla^n\mathcal{E}_k^1\|_\infty+\|\nabla^n\mathcal{E}_k^2\|_{L^\infty}+\|\nabla^n\mathcal{E}_k^3\|_\infty\\
&\lesssim_{\varepsilon,m} \sum_{j} N_{j,k+1}^{n -1+ \varepsilon} N_{J_d,k} \big(e^{-\frac14 M_{j,k+1}^2t} + M_{j,k+1}^{-m}N_{J_d,k}^{m} \big)\\
&\quad \, \,  +\sum_{j'<j} N^{n-1+\varepsilon}_{j,k+1}N_{j',k+1}e^{-\frac14N_{j,k+1}^2t}\\
&\quad \, \, +\sum_{j'<j} N_{j,k+1}^{-1}N_{j',k+1} N_{j,k+1}^{n-m}M_{j,k+1}^{m+\varepsilon}\\
&\quad \, \, +\sum_{j}\big( N_{j,k+1}^{n-1+\varepsilon}e^{-\frac14 N_{j,k+1}^2t}+N_{j,k+1}^{-2}M_{j,k+1}^{n+ 2+\varepsilon}\big)
+\sum_jN_{j,k+1}^{n+1+\varepsilon}\ell_{k+1}.
\end{split}
\]
As before, it follows that there exists $\alpha> 0$, such that for any $\varepsilon_0 >0$ and $\bar n \in \mathbb{N}$,
\[
\|\nabla^n(\mathcal{J}_k^1+\mathcal{J}_k^2 + \mathcal{E}_k)\|_\infty \leq \varepsilon_0 N^{-\alpha}_{1,k+1}(t^{-\frac{n}{2} +\alpha}+1), \qquad n=1,2,\dots, \bar n,
\]
for $A$, $b$, and $m$ large enough. For $n=0$ the bound becomes worse on $[0,t_k]$, as in Proposition~\ref{barv-I_estimate_proposition}.
\end{proof}

\subsection{Bounds on the solution and residual}

Now we are ready to prove that the principal part $v$ satisfies a forced Navier--Stokes system with the forcing term small.

\begin{proposition}\label{prop_f_estimate}
For any $\epsilon_0>0$, we can arrange that the principal part $v$ satisfies
    \begin{align}\label{v_equation}\partial_tv-\Delta v+\mathbb P\div v\otimes v=\mathbb P\div f\end{align}
    for an error $f$ satisfying
\begin{align}\label{f_bound}
    \|\nabla^n f\|_{C^\kappa}\lesssim_n \epsilon_0 (t^{-1 - \frac{n}{2} +\alpha}+1)
\end{align}
for some $0<\kappa<\alpha<1$.
Moreover,
\begin{align}
    \|\nabla^nv_k(t)\|_{L^p}&\lesssim ((t/N_{1,k+1})^{\alpha}+2^{-k/p})t^{-\frac12(1+n)}+N_{1,k+1}^{-\alpha},\label{vk-pointwise-bounds}
    \end{align}
    and
    \begin{align}
    \|\nabla^nv(t)\|_{L^\infty}+\|\nabla^n\bar v(t)\|_{L^\infty}&\lesssim t^{-\frac12(1+n)}\label{v-pointwise-bounds}
\end{align}
for $n=0,1,2,\ldots,\overline n$, and
\begin{align}\label{v-critical-bounds}
    \|v\|_{L^1([t',t], t^{-\frac12}dt; L^\infty)}+\|v\|_{L^2([t',t]; L^\infty)}^2\lesssim 1+(\log A)^{-1}\log(t/{t'}).
\end{align}
\end{proposition}

\begin{proof}
From the definition and \eqref{eq:Dpsi_bounds}, we easily have
\begin{align}\label{vk_bar_estimate}
    \|\nabla^n\overline v_k(t)\|_{p}\lesssim |\Omega_k|^\frac1p\|\nabla^n\overline v_k(t)\|_{\infty}&\lesssim 2^{-k/p}\sum_jN_{j,k}^{1+n}e^{-N_{j,k}^2t}
\end{align}
which yields the estimate for $\overline v$ in \eqref{v-pointwise-bounds} after summing in $k\geq0$. Combining this estimate for $\bar v_k$ with Proposition~\ref{difference_estimate_proposition},
\begin{equation}\label{alt_vk_bound}
\begin{split}
    \|\nabla^nv_k(t)\|_p&\leq \|\nabla^n(v_k(t)-\bar{v}_k(t))\|_p+\|\nabla^n\bar{v}_k(t)\|_p\\
    &\lesssim \|\nabla^{n+1}(R_k(t)-\bar{R}_k(t))\|_p+\|\nabla^n\bar{v}_k(t)\|_p\\
    &\lesssim N_{1,k+1}^{-\alpha}(t^{-\frac12-\frac n2+\alpha}+1)+2^{-k/p}\sum_jN_{j,k}^{1+n}e^{-N_{j,k}^2t},
\end{split}
\end{equation}
which implies \eqref{vk-pointwise-bounds}--\eqref{v-pointwise-bounds}.

To construct $f$, recall that
\[
\begin{split}
v_k(t) &= -\int_0^{t} e^{(t-s)\Delta }\mathbb P\div (\bar v_{k+1}\otimes \bar v_{k+1})(s) \, ds\\
&=\int_0^{t} e^{(t-s)\Delta }\mathbb P\div (-v_{k+1}\otimes  v_{k+1}+f_{k+1})(s) \, ds
\end{split}
\]
for $k\geq0$, with
\[
f_k=v_{k} \otimes  v_{k}-\bar v_{k}\otimes \bar v_{k}=(v_{k}-\bar v_{k}) \otimes v_{k} + \bar v_{k} \otimes (v_{k}-\bar v_{k}).
\]
Summing in $k\geq0$, we compute that $(v,f)$ satisfies \eqref{v_equation}, which we express equivalently in mild form as
\begin{align*}
    v(t)=\int_0^te^{(t-s)\Delta}\mathbb P\div (-v(s)\otimes v(s)+f(s))ds,
\end{align*}
as long as
\begin{align*}
    \mathbb P\div f&=\mathbb P\div\left(v\otimes v+\sum_{k\geq0}(-v_{k+1}\otimes v_{k+1}+f_{k+1})\right)\\
    &=\mathbb P\div\left(v_0\otimes v_0+\sum_{k\geq1} f_k   +\sum_{k_1\neq k_2}v_{k_1}\otimes v_{k_2}\right).
\end{align*}
This requirement is met by defining
\[f=\sum_{k\geq0} f_k   +\sum_{k_1\neq k_2}v_{k_1}\otimes v_{k_2}.\]
Indeed, $\bar v_0$ is a shear flow due to \eqref{def_psi_0} and the fact that $a_{j,0}$ are constant, so we have
\begin{align*}
    \mathbb P\div v_0\otimes v_0=\mathbb P\div (\bar v_0\otimes\bar v_0+f_0)=\mathbb P\div f_0.
\end{align*}

Thanks to Proposition~\ref{difference_estimate_proposition} and \eqref{v-pointwise-bounds}, and taking $\alpha>0$ smaller as needed,
\begin{align*}
\|\nabla^nf_k(t)\|_{C^\kappa} &\lesssim \sum_{i=0}^n\|\nabla^i(v_{k}-\bar v_{k})\|_{C^\kappa}(\|\nabla^{n-i}v_{k}\|_{C^\kappa}+\|\nabla^{n-i}\bar v_{k}\|_{C^\kappa})\\
&\lesssim\sum_{i=0}^n\epsilon_0N_{1,k+1}^{-2\alpha}(t^{-\frac12(1+i+\kappa)+2\alpha}+1)(t^{-\frac12(1+n-i+\kappa)}+1)\\
&\lesssim \epsilon_0N_{1,k+1}^{-2\alpha}(t^{-1-\frac n2+\alpha+(\alpha-\kappa)}+1).
\end{align*}
Summing in $k$, we reach the bound claimed in \eqref{f_bound} for the first sum in $f$.

For the other, we exploit the symmetry between $k_1$ and $k_2$ to estimate
\begin{align*}
    \Big\|\nabla^n\sum_{k_1\neq k_2}v_{k_1}\otimes v_{k_2}\Big\|_{C^\kappa}&\lesssim \sum_{i=0}^n\sum_{k_1<k_2}\|\nabla^iv_{k_1}\|_{C^\kappa}\|\nabla^{n-i}v_{k_2}\|_{C^\kappa}\\
    &\eqcolon I_1+I_2+I_3+I_4,
\end{align*}
decomposing the upper bound for both factors based on the two terms in \eqref{alt_vk_bound}. We have
\begin{align*}
    I_1&=\sum_i\sum_{k_1<k_2}(N_{1,k_1+1}^{-\alpha}t^{-\frac12-\frac {i+\kappa}2+\alpha})(N_{1,k_2+1}^{-\alpha}t^{-\frac12-\frac {n-i+\kappa}2+\alpha})\lesssim t^{-1-\frac n2+2\alpha-\kappa},
\end{align*}
an acceptable bound considering $\alpha-\kappa>0$ and $t\in(0,1)$. Next, again making $\alpha$ smaller as necessary,
\begin{align*}
    I_2&=\sum_i\sum_{k_1<k_2}(N_{1,k_1+1}^{-2\alpha}t^{-\frac12-\frac {i+\kappa}2+2\alpha})\Big(\sum_jN_{j,k_2}^{1+n-i+\kappa}e^{-N_{j,k_2}^2t}\Big)\\
    &\lesssim \sum_i \Big(t^{-\frac12-\frac {i+\kappa}2+2\alpha}\sum_j\sum_{k_2}N_{j,k_2}^{1+n-i+\kappa}e^{-N_{j,k_2}^2t}\Big)\\
    &\lesssim\sum_i\Big(t^{-\frac12-\frac {i+\kappa}2+2\alpha}\sum_j t^{-\frac12(1+n-i+\kappa)}\Big)\\
    &\lesssim t^{-1-\frac n2+2\alpha-\kappa},
\end{align*}
which leads to the claimed bound. Note that $I_3$ is the complementary term to $I_2$, in which $k_1$ and $k_2$ are interchanged, is analogous and we omit it. Finally,
\begin{align*}
    I_4&=\sum_i\sum_{k_1<k_2}\Big(\sum_jN_{j,k_1}^{1+n-i+\kappa}e^{-N_{j,k_1}^2t}\Big)\Big(\sum_jN_{j,k_2}^{1+i+\kappa}e^{-N_{j,k_2}^2t}\Big)\\
    &\lesssim \sum_i\sum_j\sum_{k_2}N_{J_d,k_2-1}^{1+i+\kappa}N_{j,k_2}^{1+n-i+\kappa}e^{-N_{j,k_2}^2t}.
\end{align*}
Based on \eqref{N_definition}, one computes
\begin{align*}
    N_{J_d,k_2-1}&\sim N_{j,k_2}^{b^{-j/J_d}}
\end{align*}
and therefore, noting that the dominant contribution occurs at $i=0$,
\begin{align*}
    I_4&\lesssim \sum_i\sum_j\sum_{k_2}N_{j,k_2}^{1+n-i+\kappa+(1+i+\kappa)b^{-j/J_d}}e^{-N_{j,k_2}^2t}\\
    &\lesssim \sum_j\sum_{k_2}N_{j,k_2}^{n+2-(1-b^{-j/J_d})+2\kappa}e^{-N_{j,k_2}^2t}\\
    &\lesssim t^{-1-\frac n2+\frac12(1-b^{-1/J_d}-4\kappa)}.
\end{align*}
Taking $\alpha,\kappa>0$ smaller as needed, we conclude \eqref{f_bound}.

To show \eqref{v-critical-bounds}, it suffices to prove the estimate on $L^1(t^{-\frac12}dt;L^\infty)$, from which the $L^2(L^\infty)$ estimate follows by interpolation with \eqref{v-pointwise-bounds}. From \eqref{alt_vk_bound}, we have
\begin{align*}
    \|v\|_{L^1([t',t],t^{-\frac12}dt;L^\infty)}&\lesssim\sum_k\int_{t'}^t\left(N_{1,k+1}^{-\alpha}s^{-1+\alpha}+\sum_jN_{j,k}s^{-\frac12}e^{-N_{j,k}^2s}\right)ds\eqcolon J_1+J_2.
\end{align*}
For $J_1$, we have
\begin{align*}
    J_1=\sum_kN_{1,k+1}^{-\alpha}\int_{t'}^ts^{-1+\alpha}ds&\lesssim t^\alpha\lesssim1.
\end{align*}
For the other, we decompose:
\begin{align*}
J_2=&\sum_j\Big(\sum_{k:N_{j,k}< t^{-1/2}}\int_{t'}^tN_{j,k}s^{-\frac12}ds+\sum_{t^{-1/2}\leq N_{j,k}\leq (t')^{-1/2}}\int_{t'}^tN_{j,k}s^{-\frac12}e^{-N_{j,k}^2s}ds\\
&\qquad+\sum_{k:N_{j,k}>(t')^{-1/2}}\int_{t'}^tN_{j,k}^2e^{-N_{j,k}^2s}ds\Big).
\end{align*}
Explicitly carrying out the integral in the first term, clearly the sum is $O(1)$. For the second, we expand the region of integration to $(0,\infty)$. By scaling, the integral is $O(1)$ uniformly in $k$. Thus the sum is bounded, up to a constant multiple, by the number of terms. For the third term, the integrals are controlled by $O(\exp(-N_{j,k}^2t'))$, the sum of which over $N_{j,k}>(t')^{-1/2}$ is clearly $O(1)$. In total,
\begin{align*}
    J_2&\lesssim 1+\sup_j\#\{k:t^{-1/2}\leq N_{j,k}\leq (t')^{-1/2}\}.
\end{align*}
Observe that
\begin{align*}
    \frac{N_{j,k+1}}{N_{j,k}}\sim A^{b^{k+1+(j-1)/J_d}-b^{k+(j-1)/J_d}}\gtrsim A^{b-1}.
\end{align*}
Thus, the number of $k$'s with $(t')^{-1/2}\leq N_{j,k}\leq t^{-1/2}$ is bounded by the corresponding quantity for geometric series with ratio $A^{b-1}$. We conclude
\begin{align*}
J_2\lesssim 1+\log_{A^{b-1}}((t')^{-\frac12}/t^{-\frac12})
\end{align*}
which, combined with the estimate for $J_1$, implies \eqref{v-critical-bounds}.
\end{proof}

\section{Construction of the corrector}\label{sec:corrector}

Let $U$ be any classical solution to the Navier--Stokes equations on the time interval $[0,T]$ which, once again, we replace with $[-T_*, T-T_*]$ by translation. We may define
\begin{align*}
    \sup_{0\leq n\leq 10}\|\nabla^nU\|_{L_{t,x}^\infty(\mathbb T^d\times[0,T-T_*])}\eqcolon C_U<\infty.
\end{align*}
The family of solutions $u^{(\sigma)}$ claimed in Theorem~\ref{thm-non-unique} will arise by modulating a frequency scale $N_0>0$ (see Section~\ref{sec:proof}). Having constructed $v$ on $\mathbb T^d\times[0,\infty)$, we extend it by $0$ to the full time interval $\mathbb R$, then rescale it to $v^{N_0}$, which we regard as a $2\pi/N_0$-periodic vector field on $\mathbb R^d\times[0,\infty)$. Here, for an arbitrary vector field on $\mathbb R^d$, we define the natural Navier--Stokes rescaling
\begin{align*}
    V^{N_0}(x,t)\coloneqq N_0V(N_0x,N_0^2t).
\end{align*}
Recall that critical norms such as $L^\infty((0,\infty); \dot W^{-1,\infty}(\mathbb R^d))$ are invariant under the transformation $V\mapsto V^{N_0}$.

From there, we construct the exact Navier--Stokes solution as
\[u = U + v^{N_0} + w^{N_0}\]
with $w$ a small corrector. To simplify the computation, we instead construct the rescaled solution $u^{1/N_0}=U^{1/N_0}+v+w$. The corrector $w$ should satisfy
\[
\begin{split}
\partial_tw-\Delta w+\mathbb P\div(w\otimes w+2(U^{1/N_0}+v)\odot w)&=-\mathbb P\div (f+2 U^{1/N_0}\odot v),\\
w|_{t=0}&=0.
\end{split}
\]
After the translation and rescaling, the lifetime of $U^{1/N_0}$ is $[-N_0^2T_*,N_0^2(T-T_*)]$.

\subsection{Semigroup estimates}

We apply the semigroup theory and fixed point argument to construct $w$. Let $\bar T=\min\{N_0^2(T-T_*),\bar C\}$ for some large $\bar C>1$ to be specified. For suitable $\alpha$ and $\kappa$ we define the Banach spaces
\begin{equation}\label{def-X}
\begin{split}
X&=\Big\{V\in C^0((0,\bar T]; C^{1,\kappa}(\mathbb T^d; \mathbb R^d)): \\
& \qquad\qquad \|V\|_{X}:= \sup_{t\in(0,\bar T]}(t^{\frac{1-\alpha}2} \|V\|_{L^\infty}+t^{\frac{2-\alpha}2} \|\nabla V\|_{C^{\kappa}})<\infty\Big\}.
\end{split}
\end{equation}
and
\begin{equation}\label{def-Y}
\begin{split}
Y&=\Big\{\phi\in C^0((0,\bar T]; C^{1,\kappa}(\mathbb T^d; \mathrm{Sym}^{d})): \\
& \qquad\qquad \|\phi\|_{Y}:= \sup_{t\in(0,\bar T]}(t^{1-\alpha} \|\phi\|_{L^\infty}+t^{\frac32-\alpha} \|\nabla \phi\|_{C^{\kappa}})<\infty\Big\}.
\end{split}
\end{equation}
Then we have the following product rule, which is immediate from the definitions and the estimate $\|gh\|_{C^{1,\kappa}}\lesssim\|g\|_{C^{1,\kappa}}\|h\|_{L^\infty}+\|g\|_{L^\infty}\|h\|_{C^{1,\kappa}}$.
\begin{lemma}\label{product_X_Y_lemma}
    If $g,h\in X$, then $g\odot h\in Y$ with
    \begin{align*}
        \|g\odot h\|_Y\lesssim \|g\|_X\|h\|_X.
    \end{align*}
\end{lemma}

For $\phi\in Y$ and $0<t'\leq t\leq \bar T$, let the semigroup $S(t,t')\phi$ be the solution to the linearized equation 
\begin{equation}\label{semi-group}
\begin{split}
\partial_t S(t,t')\phi-\Delta S(t,t')\phi+2\mathbb P\div ((U^{1/N_0}+v)(t)\odot S(t,t')\phi)&=0,\\
S(t',t')\phi&=\mathbb P\div \phi(t').
\end{split}
\end{equation}

Note that we suppress the dependence of $S(t,t')$ on the parameter $1/N_0$. We shall show that the semigroup $S(t,t')$ acts approximately as $e^{(t-t')\Delta}\mathbb P\div$ up to a mild loss of $(t/t')^\epsilon$. 

\begin{proposition}\label{prop-semi}
For any $\phi\in Y$ and $0<t'\leq t\leq \bar T$, the estimate
\begin{equation}\notag
\|S(t,t')\phi\|_{L^\infty(\mathbb T^d)}+(t-t')^{\frac12} \|\nabla S(t,t')\phi\|_{C^{\kappa}(\mathbb T^d)}\lesssim_{\bar C} t^{-\frac12}(t')^{-1+\alpha}(t/t')^{\epsilon} \|\phi\|_Y
\end{equation}
holds.
\end{proposition}

\begin{proof}
Recall that $\bar T\leq \bar C$ by definition. In the subsequent proof, we freely use that $t'<t\leq\bar C$ without explicitly indicating the dependence.

Applying Duhamel's formula to \eqref{semi-group} gives 
\begin{equation}\notag
\begin{split}
S(t,t')\phi &= e^{(t-t')\Delta}\mathbb P\div \phi(t')-2\int_{t'}^t e^{(t-s)\Delta}\mathbb P\div ((U^{1/N_0}+v)(s) \odot  S(s,t')\phi)\, ds\\
&=e^{(t-t')\Delta}\mathbb P\div \phi(t')-2\int_{t'}^{t' \vee {\frac t2}} e^{(t-s)\Delta}\mathbb P\div ((U^{1/N_0}+v)(s) \odot  S(s,t')\phi)\, ds\\
&\quad-2\int_{t' \vee {\frac t2}}^t e^{(t-s)\Delta}\mathbb P\div ((U^{1/N_0}+v)(s) \odot  S(s,t')\phi )\, ds\\
&=: I_1+I_2+I_3.
\end{split}
\end{equation}

To handle the Leray projection $\mathbb P$ in $I_1$, we either estimate $I_1$ in H\"older space $C^{\kappa}$ or exploit the $C^1\to L^\infty$ boundedness of the heat kernel. We take the smaller upper bound among the two types of estimates. First, we have 
\begin{equation}\notag
\|I_1\|_{L^\infty} \lesssim \|\nabla \phi(t')\|_{C^{\kappa}}\lesssim (t')^{-\frac32+\alpha}\|\phi\|_Y= (t')^{-\frac12}(t')^{-1+\alpha}\|\phi\|_Y,
\end{equation}
and 
\begin{equation}\notag
\|I_1\|_{L^\infty} \lesssim (t-t')^{-\frac12}\| \phi(t')\|_{L^\infty}\lesssim (t-t')^{-\frac12}(t')^{-1+\alpha}\|\phi\|_Y.
\end{equation}
Note that if $0<t'<\frac t2<\bar T$, we have $(t')^{-\frac12}>(t-t')^{-\frac12}\approx t^{-\frac12}$; if $0<\frac t2< t'<\bar T$, we also have
$(t-t')^{-\frac12}>(t')^{-\frac12}\approx t^{-\frac12}$. Therefore, we obtain from the previous two estimates that
\begin{equation}\notag
\|I_1\|_{L^\infty} \lesssim t^{-\frac12}(t')^{-1+\alpha}\|\phi\|_Y.
\end{equation}

The integral $I_2=0$ for $t'\geq \frac t2$. While for $t'<\frac t2$, we have
\begin{equation}\notag
\begin{split}
\|I_2\|_{L^\infty} &\lesssim \int_{t'}^{t' \vee {\frac t2}} (t-s)^{-\frac12} (C_U+\|v(s)\|_{L^\infty})\|S(s,t')\phi\|_{L^\infty} \, ds\\
&\lesssim t^{-\frac12}\int_{t'}^{t' \vee {\frac t2}}  (C_U+\|v(s)\|_{L^\infty})\|S(s,t')\phi\|_{L^\infty} \, ds\\
&= t^{-\frac12}\int_{t'}^{t' \vee {\frac t2}}  \left(s^{-\frac 12}(C_U+\|v(s)\|_{L^\infty})\right)\left( s^{\frac12}\|S(s,t')\phi\|_{L^\infty} \right)\, ds\\
\end{split}
\end{equation}
since $(t-s)^{-\frac12}\lesssim t^{-\frac12}$. 

On the other hand, we have
\begin{equation}\notag
\begin{split}
\|I_3\|_{L^\infty} &\lesssim \int_{t' \vee {\frac t2}}^t (t-s)^{-\frac12} (C_U+\|v(s)\|_{L^\infty})\|S(s,t')\phi\|_{L^\infty} \, ds\\
&\lesssim t^{-\frac12}\int_{t' \vee {\frac t2}}^t  \left((t-s)^{-\frac 12}(C_U+\|v(s)\|_{L^\infty})\right)\left( s^{\frac12}\|S(s,t')\phi\|_{L^\infty} \right)\, ds.
\end{split}
\end{equation}
Combining the estimates above gives
\begin{equation}\notag
\begin{split}
\|S(t,t')\phi\|_{L^\infty}&\lesssim t^{-\frac12}(t')^{-1+\alpha}\|\phi\|_Y\\
&\quad+ t^{-\frac12}\int_{t'}^{t' \vee {\frac t2}}  \left(s^{-\frac 12}(C_U+\|v(s)\|_{L^\infty})\right)\left( s^{\frac12}\|S(s,t')\phi\|_{L^\infty} \right)\, ds\\
&\quad+ t^{-\frac12}\int_{t' \vee {\frac t2}}^t  \left((t-s)^{-\frac 12}(C_U+\|v(s)\|_{L^\infty})\right)\left( s^{\frac12}\|S(s,t')\phi\|_{L^\infty} \right)\, ds.
\end{split}
\end{equation}

Denote $h(t)=t^{\frac12}\|S(t,t')\phi\|_{L^\infty(\mathbb T^d)}$. It then follows
\begin{equation}\notag
h(t)\lesssim (t')^{-1+\alpha}\|\phi\|_Y+\int_{t'}^t (s^{-\frac 12}+(t-s)^{-\frac 12})(C_U+\|v(s)\|_{L^\infty}) h(s) \,ds.
\end{equation}
Applying the fractional Gr\"onwall's inequality leads to
\begin{align*}\notag
h(t)&\lesssim (t')^{-1+\alpha}\|\phi\|_Y \exp \Big(O(\int_{t'}^t s^{-\frac12}(C_U+\|v(s)\|_{L^\infty})\,ds\\
&\qquad+(C_Ut^{\frac12}+\|s^{\frac12} v(s)\|_{L^\infty([t',t])})\int_{t'}^t (C_U+\|v(s)\|_{L^\infty})^2\, ds ) \Big).
\end{align*}
Therefore we need an estimate for $\|v\|_{L^1([t',t], t^{-\frac12}dt; L^\infty)}$ and $\|v\|_{L^2([t',t]; L^\infty)}$. Using \eqref{v-critical-bounds}, we obtain
\[h(t)\lesssim (t')^{-1+\alpha}\|\phi\|_Y \exp \left(O_U(1+(\log A)^{-1}\log(t/{t'})) \right), \]
emphasizing that the implicit constant in the exponential depends on the background $U$. Choosing $A\gg 1$ implies 
\[\|S(t,t')\phi\|_{L^\infty(\mathbb T^d)}\lesssim t^{-\frac12}(t')^{-1+\alpha}(t/{t'})^{\frac{\epsilon}2}\|\phi\|_Y.\]

Next, in order to estimate $\|\nabla S(t,t')\phi\|_{C^{\kappa}(\mathbb T^d)}$, we rewrite $\nabla S(t,t')\phi$ as
\begin{equation}\notag
\begin{split}
\nabla S(t,t')\phi &= e^{(t-t')\Delta}\nabla \mathbb P\div \phi(t')-2\int_{t'}^t e^{(t-s)\Delta}\nabla \mathbb P\div ((U^{1/N_0}(s)+v(s)) \odot  S(s,t')\phi)\, ds\\
&=e^{(t-t')\Delta}\nabla \mathbb P\div \phi(t')-2\int_{t'}^{\frac12(t+t')} e^{(t-s)\Delta}\nabla\mathbb P\div ((U^{1/N_0}(s)+v(s)) \odot  S(s,t')\phi)\, ds\\
&\quad-2\int_{\frac12(t+t')}^t e^{(t-s)\Delta}\nabla\mathbb P\div ((U^{1/N_0}(s)+v(s)) \odot  S(s,t')\phi )\, ds\\
&=: I_4+I_5+I_6.
\end{split}
\end{equation}
Applying the standard heat operator estimate
\begin{equation}\label{heat}
\|e^{t\Delta}\nabla^m \mathbb P g\|_{C^{r}}\lesssim t^{-\frac{m+r-s}2}\| g\|_{C^{s}}
\end{equation}
for all $m\geq 0$, $r,s\in \mathbb R$ satisfying $m+r-s\geq 0$, $r\notin\mathbb N$,
we have
\begin{align*}
\|I_4\|_{C^{\kappa}}&\lesssim (t-t')^{-\frac12} \|\mathbb P\div \phi(t')\|_{C^{\kappa}}\\
&\lesssim (t-t')^{-\frac12} \|\nabla \phi(t')\|_{C^{\kappa}}\\
&\lesssim (t-t')^{-\frac12}(t')^{-\frac32+\alpha} \|\phi\|_{Y}\\
&\lesssim (t-t')^{-\frac12}(t')^{-\frac{3+\kappa}2+\alpha}\|\phi\|_{Y},
\end{align*}
in the last line using that $t'\in(0,\bar C)$ to insert a factor of $(t')^{-\frac\kappa2}$. Alternatively, we may estimate 
\begin{equation}\notag
\|I_4\|_{C^{\kappa}}\lesssim (t-t')^{-\frac{2+\kappa}2} \| \phi(t')\|_{C^0}\lesssim (t-t')^{-\frac{2+\kappa}2}(t')^{-1+\alpha} \|\phi\|_{Y}.
\end{equation}
Taking whichever of the two estimates is stronger, we find
\begin{align*}
    \|I_4\|_{C^\kappa}&\lesssim\min\big\{(t-t')^{-\frac12}(t')^{-\frac{3+\kappa}2+\alpha},(t-t')^{-\frac{2+\kappa}2}(t')^{-1+\alpha}\big\}\|\phi\|_Y\\
    &=(t-t')^{-\frac12}(t')^{-1+\alpha} \min\big\{(t')^{-\frac{1+\kappa}2},(t-t')^{-\frac{1+\kappa}2}\big\}\|\phi\|_{Y}\\
    &\lesssim (t-t')^{-\frac12}(t')^{-1+\alpha}t^{-\frac{1+\kappa}2} \|\phi\|_{Y}.
\end{align*}

The terms $I_5$ and $I_6$ can be estimated similarly.
\end{proof}

\subsection{Fixed point argument}

Next we apply the fixed point argument to construct the corrector.
\begin{proposition}\label{w-exists-fixed-point-proposition}
There exists $\varepsilon>0$ such that for all $\delta\in(0, \varepsilon)$ and $A\gg 1$, there exists $w\in B_X(0,\delta)$ with $u^{1/N_0}=U^{1/N_0}+v+w$ solving \eqref{eq:NSE} on $[0,\bar T]\times\mathbb R^d$. Further, choosing $\bar C$ and $N_0$ large depending on $U$ and $T$, the solution $u^{1/N_0}$ extends to a classical solution on $[0,T]$.
\end{proposition}

\begin{proof}
We construct $w$ as a fixed point of the map
\begin{align*}
    F(w)=-\int_0^tS(t,t')(w\otimes w+f+2 U^{1/N_0}\odot v)(t')\,dt'
\end{align*}
in $B_X(0,\delta)$. By Proposition~\ref{prop-semi}, for any $w\in B_X(0,\delta)$,
\begin{align*}
    \|F(w)\|_X&\leq\sup_{t\in(0,\bar T]}\int_0^t\Big(t^{\frac12(1-\alpha)}\|S(t,t')(w\otimes w+f+2 U^{1/N_0}\odot v)(t')\|_\infty\\
    &\qquad+t^{1-\frac\alpha2}\|\nabla S(t,t')(w\otimes w+f+2 U^{1/N_0}\odot v)(t')\|_{C^\kappa}\Big)\,dt'\\
    &\lesssim_{\bar C}\sup_{t\in(0,\bar T]}\int_0^t\Big(t^{\frac12(1-\alpha)}+t^{1-\frac\alpha2}(t-t')^{-\frac12}\Big)t^{-\frac12}(t')^{-1+\alpha}\left(\frac t{t'}\right)^\epsilon\\
    &\qquad \qquad \cdot\|w\otimes w+f+2 U^{1/N_0}\odot v\|_Y\,dt'\\
    &\lesssim \|w\otimes w\|_Y+\|f\|_Y+\|U^{1/N_0}\odot v\|_Y,
\end{align*}
choosing $\epsilon\in(0,\alpha)$. It follows from Lemma~\ref{product_X_Y_lemma} that 
\[ \|w\otimes w\|_Y\lesssim \|w\|_X^2.\]
On the other hand, in view of the definition \eqref{def-Y} and Proposition~\ref{prop_f_estimate}, we have
\begin{equation*}
\begin{split}
\|f\|_Y&=\sup_{t\in(0,\bar T]}(t^{1-\alpha}\|f\|_{L^\infty}+t^{\frac32-\alpha}\|\nabla f\|_{C^\kappa})\\
&\lesssim \epsilon_0\sup_{t\in(0,\bar T]}(t^{1-\alpha}(t^{-1+\alpha}+1)+t^{\frac32-\alpha}(t^{-\frac32+\alpha}+1))\\
&\lesssim \epsilon_0,
\end{split}
\end{equation*}
and 
\begin{equation*}
\begin{split}
\|U^{1/N_0}\odot v\|_Y&=\sup_{t\in(0,\bar T]}(t^{1-\alpha}\|U^{1/N_0}\odot v\|_{L^\infty}+t^{\frac32-\alpha}\|\nabla(U^{1/N_0}\odot v)\|_{C^\kappa})\\
&\lesssim 1/N_0\sup_{t\in(0,\bar T]}(t^{1-\alpha}\| v\|_{L^\infty}+t^{\frac32-\alpha}\| v\|_{C^\kappa}+t^{\frac32-\alpha}\|\nabla v\|_{C^\kappa})\\
&\lesssim 1/N_0\sup_{t\in(0,\bar T]}(t^{1-\alpha-\frac12}+t^{\frac32-\alpha-\frac12(1+\kappa)}+t^{\frac32-\alpha-\frac12(2+\kappa)})\\
&\lesssim 1/N_0. 
\end{split}
\end{equation*}
Therefore, for large enough $N_0>1$, we conclude
\begin{align*}
    \|F(w)\|_X&\leq C(\|w\|_X^2+\|f\|_Y+\|U^{1/N_0}\odot v\|_Y)\leq C'(\delta^2+\epsilon_0),
\end{align*}
for some constants $C,C'$ depending on $\bar C$. It follows that $F$ takes $B_X(0,\delta)$ to itself if $\delta>0$ is chosen smaller than $1/(2C')$, then $\epsilon_0>0$ is chosen smaller than $\delta/(2C')$. Similarly, if $w_1,w_2\in B_X(0,\delta)$,
\begin{align*}
    \|F(w_1)-F(w_2)\|_X\lesssim \delta\|w_1-w_2\|_X.
\end{align*}
Taking $\delta>0$ small enough, we conclude the existence of a fixed point $w\in B_X(0,\delta)$ of $F$. This proves the claim about $u$ on $[0,\bar T]$. For $N_0$ sufficiently large, it will be the case that $N_0^2(T-T_*)>\bar T=\bar C$. To extend the lifespan, we argue that global solution theory can be applied starting from initial data $u(\bar T)$. Observe that by \eqref{v-pointwise-bounds} and the fact that $w\in X$, the perturbation $v(\bar T)+w(\bar T)$ can be made to be arbitrarily small by arranging $\bar T$ to be sufficiently large (which in turn is achieved with a sufficiently large choice of $N_0$ and $\bar C$). Moreover, $U^{1/N_0}(\bar T)$ stays bounded in $BMO^{-1}$ by the scaling invariance. Taking $\bar T$ as large as needed depending on $U$, we conclude existence of $u$ up to time $T-T_*$ by (for instance) the stability theorem of Auscher, Dubois, and Tchamitchian~\cite{MR2062638}.
\end{proof}

\section{Proof of main results}\label{sec:proof}

Finally we show that the construction in Section \ref{sec:corrector} yields a solution to the Navier--Stokes equations satisfying the properties stated in Theorem \ref{main-thm} and Theorem \ref{thm-non-unique} as well, and hence complete the proof of the main results.

We first argue that the constructed velocity field $u$ in Section \ref{sec:corrector} is indeed a weak solution, by establishing the following general result,  which is of independent interest. 

\begin{lemma}\label{le-weak}
Let $u\in L^2(\mathbb{T}^d \times [0,T])\cap L^\infty([0,T];H^s(\mathbb{T}^d))$ for some $s\in \mathbb{R}$ be such that $u|_{[0,T_*)}$ and $u|_{(T_*,T]}$ are classical solutions of \eqref{eq:NSE} on $[0,T_*)$ and $(T_*,T]$ respectively. If
\[
\lim_{t\to T_*^-} u(t) = \lim_{t\to T_*^+} u(t), \qquad \text{in} \quad \mathcal{D}'(\mathbb{T}^d),
\]
then $u$ is a weak solution of the Navier--Stokes equations on $[0,T]$.
\begin{proof}
For any test function $\varphi \in \mathcal{D}_T$ and $t_0\in(T_*,T)$, we have
\[
\begin{split}
\int_{t_0}^T \int_{\mathbb{T}^d} &u\cdot \big(  \partial_t \varphi+ \Delta \varphi +  u \cdot \nabla \varphi  \big) \, dx dt +
\int_{\mathbb{T}^d} u(x,t_0) \cdot \varphi(x, t_0) \, dx\\
&= - \int_{t_0}^T \int_{\mathbb{T}^d} \varphi \cdot \big(  \partial_t u - \Delta u +  u \cdot \nabla u  +\nabla p\big) \, dx dt\\
&=0.
\end{split}
\]
On the other hand,
\[
\begin{split}
\lim_{t_0 \to T_*^+} &\int_{\mathbb{T}^d} u(x,t_0) \cdot \varphi(x, t_0) \, dx\\
&= \lim_{t_0 \to T_*^+} \int_{\mathbb{T}^d} u(x,t_0) \cdot \varphi(x, T_*) \, dx
+ \lim_{t_0 \to T_*^+} \int_{\mathbb{T}^d} u(x,t_0) \cdot (\varphi(x, t_0)-\varphi(x, T_*)) \, dx\\
&= \lim_{t_0 \to T_*^+} \int_{\mathbb{T}^d} u(x,t_0) \cdot \varphi(x, T_*) \, dx.
\end{split}
\]
Hence,
\[
\int_{T_*}^T \int_{\mathbb{T}^d} u\cdot \big(  \partial_t \varphi+ \Delta \varphi +  u \cdot \nabla \varphi  \big) \, dx dt = -\lim_{t_0 \to T_*^+} \int_{\mathbb{T}^d} u(x,t_0) \cdot \varphi(x, T_*) \, dx.
\]
Similarly,
\[
\begin{split}
\int_0^{T_*} \int_{\mathbb{T}^d} u\cdot \big(  \partial_t \varphi+ \Delta \varphi +  u \cdot \nabla \varphi  \big) \, dx dt &=-\int_{\mathbb{T}^d} u(x,0) \cdot \varphi(x, 0) \, dx\\
&\quad +\lim_{t_0 \to T_*^-} \int_{\mathbb{T}^d} u(x,t_0) \cdot \varphi(x, T_*) \, dx.
\end{split}
\]
Combining both integrals we get
\[
\begin{split}
\int_0^{T} \int_{\mathbb{T}^d} u\cdot \big(  \partial_t \varphi+ \Delta \varphi +  u \cdot \nabla \varphi  \big) \, dx dt &=-\int_{\mathbb{T}^d} u(x,0) \cdot \varphi(x, 0) \, dx\\
&\hspace{-2in} -\lim_{t_0 \to T_*^+} \int_{\mathbb{T}^d} u(x,t_0) \cdot \varphi(x, T_*) \, dx +\lim_{t_0 \to T_*^-} \int_{\mathbb{T}^d} u(x,t_0) \cdot \varphi(x, T_*) \, dx\\
&=-\int_{\mathbb{T}^d} u(x,0) \cdot \varphi(x, 0) \, dx,
\end{split}
\]
and thus $u$ satisfies the weak formulation in Definition~\ref{def:weak_solutions}.
\end{proof}

\end{lemma}

Moreover, we show that $u$ belongs to the claimed borderline spaces as follows.

\begin{proposition}\label{prop:borderline}
    With $u$ as defined in Section \ref{sec:corrector}, we have
        \[
    u\in L_t^{2,\infty}L_x^\infty\cap L_{t}^2L_x^p
    \]
    for all $p<\infty$. Furthermore, $u$ lies in the Koch--Tataru path space $X_{T-T_*}$ and is weak-* continuous in time into $BMO^{-1}$. When $d\geq3$, we additionally have
    \begin{align*}
        u\in L_t^\infty \dot W^{-1,\infty}_x.
    \end{align*}
\end{proposition}

\begin{proof}
First we address the boundedness in $\dot W^{-1,\infty}$ when $d\geq 3$. Recall that, by definition, $u=U+v^{N_0}+w^{N_0}$. Due to the criticality of the norm, we may equivalently estimate $w$ which, by construction, lies in $B_X(0,\delta)$. From this, \eqref{v-pointwise-bounds}, and \eqref{f_bound}, it follows
\begin{align*}
    \|w(t)\|_{\dot W^{-1,\infty}}&\lesssim_{\varepsilon} \int_0^t\|w\otimes w+2(v+U^{1/N_0})\odot w+f+2U^{1/N_0}\odot v\|_{C^\varepsilon}ds\\
    &\lesssim \int_0^t (s^{-1+\alpha-\varepsilon}+s^{-\frac12(1+\varepsilon)}+1)ds\\
    &\lesssim t^{\alpha-\epsilon}+t.
\end{align*}
Taking $\varepsilon$ sufficiently small, this implies $w\in L_t^\infty \dot W_x^{-1,\infty}(\mathbb R^d\times[-T_*,T-T_*])$.

Recall $v_k=\div R_k$. Immediately from the definition and \eqref{eq:Dpsi_bounds}, we estimate
\begin{align*}
    \|\bar R_k\|_\infty&\lesssim \sum_je^{-N_{j,k}^2t}.
\end{align*}
Combining with Proposition~\ref{difference_estimate_proposition}, when $d\geq 3$, this yields
\begin{align*}
    \|R_k(t)\|_\infty&\lesssim N_{1,k+1}^{-\alpha}+\mathbbm1_{t\leq N_{1,k}^{-2}}.
\end{align*}
Alternatively,
\begin{align*}
    \|R_k(t)\|_\infty&\lesssim_{\varepsilon} \int_0^t\|\bar v_{k+1}\otimes\bar v_{k+1}\|_{C^\varepsilon}\lesssim \int_0^t\sum_jN_{j,k+1}^{2+\varepsilon}e^{-2N_{j,k+1}^2t}ds\lesssim tN_{J_d,k+1}^{2+\varepsilon}.
\end{align*}
Thus we have $v=\div R$ where
\begin{align*}
    \|R(t)\|_\infty&\lesssim \sum_{k:tN_{J_d,k+1}^{2+2\varepsilon}\leq1}tN_{J_d,k+1}^{2+\varepsilon}+\sum_{\substack{k:tN_{J_d,k+1}^{2+2\varepsilon}>1\\\wedge N_{1,k}^2t\leq1}}1+\sum_{k:N_{1,k}^2t>1}N_{1,k+1}^{-\alpha}\\
    &\lesssim 1+\#\{k:N_{J_d,k+1}^{-2-2\epsilon}<t\leq N_{1,k}^{-2}\}+t^{b\alpha/2}.
\end{align*}
By \eqref{N_definition}, this quantity is $O(1)$ uniformly in $t>0$, and we conclude the claimed bound on $v$.

Next, we address the weak-* continuity in time into $BMO^{-1}$. For every $t\in[-T_*,T-T_*]\setminus\{0\}$, $u$ is a classical solution and continuity is obvious. At $t=0$, left continuity is clear because $u$ agrees with $U$ on $[-T_*,0]$. (Note that left continuity is vacuous if $T_*=0$.) To show right continuity at $t=0$, it suffices to show $u|_{[0,T-T_*]}\in X_{T-T_*}$ due to the recent result~\cite{hou2024regularity} of Hedong Hou. To this end, we decompose
\begin{align*}
    u=U+w+u_1+u_2+u_3
\end{align*}
where
\begin{align*}
    u_1&=\sum_{j,k}\left(e^{t\Delta}N_{j,k}\Delta\psi_{j,k}+\bar v_{j,k}\right),\\
    u_2&=v-\bar v,\\
    u_3&=-e^{t\Delta}\sum_{j,k}N_{j,k}\Delta\psi_{j,k}.
\end{align*}
All the contributions except for $u_3$ can be controlled crudely using the fact that
\begin{align}\label{X_T_criterion}
\sup_{t\in(0,T-T_*]}t^{\frac12-\epsilon}\|f(t)\|_{L^\infty}<\infty\implies f\in X_{T-T_*}
\end{align}
for any $\epsilon>0$, which is clear from the definition of $X_T$ and H\"older's inequality. The criterion \eqref{X_T_criterion} clearly applies to $U$, which is smooth, as well as to $\|w(t)\|_{\infty}\lesssim t^{-\frac12(1-\alpha)}$ by Proposition~\ref{w-exists-fixed-point-proposition}.

For $u_1$,
\begin{align*}
    e^{t\Delta}N_{j,k}\Delta\psi_{j,k}+\bar v_{j,k}&=N_{j,k}\Delta(e^{t\Delta}-e^{-N_{j,k}^2t})\psi_{j,k}\\
    &=N_{j,k}^{-1}\Delta\phi_k*[e^{t\Delta},a_{j,k}\varphi_{j,k}]\sin(N_{j,k}\eta_j\cdot x)
\end{align*}
and Lemma~\ref{l:commutator} yields
\begin{align*}
    \|u_1\|_{\infty}&\lesssim \sum_{j,k}N_{j,k}(M_{j,k}N_{j,k}^{-1}e^{-N_{j,k}^2t/4}+(M_{j,k}/N_{j,k})^m).
\end{align*}
Arguing as in the proof of Proposition~\ref{prop_f_estimate}, the first term is controlled by $O(t^{-\frac12+\epsilon})$ for $\epsilon>0$ sufficiently small. Taking $m$ sufficiently large, the second term is $O(1)$.

That $u_2\in X_{T-T_*}$ is immediate from Proposition~\ref{difference_estimate_proposition} and \eqref{X_T_criterion}.

For $u_3$, we have
\begin{align*}
    \|u_3\|_{X_{T-T_*}}&\sim\Big\|\div\sum_{j,k}N_{j,k}\nabla\psi_{j,k}\Big\|_{BMO^{-1}_{(T-T_*)^{1/2}}}
\end{align*}
by the definitions of the spaces. Thus it suffices to show
\begin{align*}
    \sum_{j,k}N_{j,k}\nabla\psi_{j,k}\in BMO.
\end{align*}
To this end, we regard $\psi_{j,k}$ as a periodic function on $\mathbb R^d$, and claim that
\begin{align*}
    \sup_{Q}\sum_{j,k}\fint_Q\Big|N_{j,k}\nabla\psi_{j,k}(x)-\Big(\fint_QN_{j,k}\nabla\psi_{j,k}\Big)\Big|dx<\infty
\end{align*}
where the supremum is taken over cubes $Q\subset\mathbb R^d$; clearly, by periodicity, it suffices to consider only cubes below side length $1$, say. We proceed as in \cite{CoiculescuPalasek2025}*{Proposition 3.8}. Let $k(Q)$ and $j(Q)$ be such that
\begin{align*}
    (k(Q),j(Q))=\max\{(k,j):\mathbb N\times\mathcal J_d:N_{j,k}\leq \ell(Q)^{-1}\}
\end{align*}
where $\mathbb N\times\mathcal J_d$ is endowed with the standard dictionary order, and $\ell(Q)$ denotes the side length of the cube.

For a fixed $Q$, there are two cases: if $(k,j)\leq(k(Q),j(Q))$, then we estimate
\begin{align*}
    \Big|N_{j,k}\nabla\psi_{j,k}(x)-\Big(\fint_QN_{j,k}\nabla\psi_{j,k}\Big)\Big|dx&=N_{j,k}\Big|\fint_Q(\nabla\psi_{j,k}(x)-\nabla\psi_{j,k}(y))dy\Big|\\
    &\lesssim N_{j,k}\fint_Q\ell(Q)\|\nabla^2 \psi_{j,k}\|_\infty dy\\
    &\lesssim N_{j,k}\ell(Q)
\end{align*}
by \eqref{eq:Dpsi_bounds}.

If instead $(k,j)>(k(Q),j(Q))$, then
\begin{align*}
    \fint_Q\Big|N_{j,k}\nabla\psi_{j,k}(x)-\Big(\fint_QN_{j,k}\nabla\psi_{j,k}\Big)\Big|dx&\lesssim |Q|^{-1}N_{j,k}\|\nabla\psi_{j,k}\|_{L^1(Q)}\\
    &\lesssim N_{j,k}\|\nabla\psi_{j,k}\|_\infty\frac{|Q\cap\supp\psi_{j,k}|}{|Q|}.
\end{align*}
By \eqref{eq:Dpsi_bounds}, we have that $N_{j,k}\|\nabla\psi_{j,k}\|_\infty=O(1)$.

Let $k_*$ be the minimal choice such that $M_{1,k_*}\geq C_0\ell(Q)^{-1}$. When $k\geq k_*$, the volume ratio is controlled by $2^{-(k-k_*)}$ according to \eqref{Omega_volume_estimate}; otherwise we use the trivial bound $1$. We estimate
\begin{align*}
    \Big\|\sum_{j,k}N_{j,k}\nabla\psi_{j,k}\Big\|_{BMO}        &\lesssim \sup_Q\Big(\sum_{(k,j)\leq(k(Q),j(Q))}N_{j,k}\ell(Q)\\
    &\qquad+\sum_{(k(Q),j(Q))<(k,j)<(k_*,1)}1+\sum_{(k,j)\geq(k_*,1)}2^{-(k-k*)}\Big).
\end{align*}
The first sum is $O(N_{j(Q),k(Q)}\ell(Q))=O(1)$ by definition of $j(Q),k(Q)$. The third sum is clearly bounded as a geometric series.

We claim that the second sum contains boundedly-many terms. Indeed, by definition of $(k(Q),j(Q))$, $(k,j)>(k(Q),j(Q))$ implies $N_{j,k}>\ell(Q)^{-1}$. Similarly, by definition of $k_*$, it contains only terms with $M_{1,k}<C_0\ell(Q)^{-1}$. Thus, the sum contains only $k$'s such that $N_{J_d,k}>\ell(Q)^{-1}$ and $M_{1,k}<C_0\ell(Q)^{-1}$. By \eqref{N_and_M_ordering}, taking $A$ large as needed, there can be at most one such $k$.

Toward showing $u\in L_t^{2,\infty}L_x^\infty\cap L_{t}^2L_x^p$, we decompose $u=U+v+w$. For $U$ and $w$, it once again suffices, by H\"older's inequality, to observe that the $L^\infty$ norms are $O(t^{-\frac12+\epsilon})$. Furthermore, $v\in L_t^{2,\infty}L_x^\infty$ is immediate from \eqref{v-pointwise-bounds} with $n=0$. It suffices to show $v\in L_t^2L_x^p$. By \eqref{vk-pointwise-bounds}, we have
\begin{align*}
    \|v_{k}\|_{L_t^2L_x^p}&\lesssim N_{1,k+1}^{-\alpha}+2^{-k/p}.
\end{align*}
Summing in $k$ proves the claim.
\end{proof}

\textbf{Proof of Theorem \ref{main-thm}:}
It follows immediately from Lemma \ref{le-weak} that the solution $u=U+v+w$ constructed in Section \ref{sec:corrector} is a weak solution of \eqref{eq:NSE} with initial data $u_0=U(0,\cdot)$. Since the background solution $U$ is classical, the upper bound $\|U(t)\|_{L^\infty}\lesssim \frac{1}{\sqrt{t}}$ holds on $t\in[0,T)$ for some $T>0$.  It follows directly from the estimate \eqref{v-pointwise-bounds} that the principal part $v$ also satisfies this upper bound. Regarding the corrector $w$, we have from the fixed point argument and the definition of space $X$ as in \eqref{def-X} 
\[\|w(t)\|_{L^\infty} \lesssim t^{-\frac12+\frac12\alpha}\lesssim \frac{1}{\sqrt{t}}, \ \ 0\leq t<\min\{1,T\}.\]
Thus, the upper bound $\|u(t)\|_{L^\infty}\lesssim \frac{1}{\sqrt{t}}$ on $[0,T)$. To obtain the lower bound, we note that there exists a sequence of times $t_n\to 0+$ such that the principal velocity field $v$ satisfies 
\[v(t_n)\sim \frac{1}{\sqrt{t_n}}e^{-1}.\]
Therefore we have
\[\|u(t_n)\|_{L^\infty}\geq \|v(t_n)\|_{L^\infty}-\|U(t_n)\|_{L^\infty}-\|w(t_n)\|_{L^\infty}\gtrsim \frac{1}{\sqrt{t_n}}.\]
The last conclusion of the theorem follows from Proposition \ref{prop:borderline}.

\textbf{Proof of Theorem \ref{thm-non-unique}:}

Having constructed in Section~\ref{sec:corrector} the solution $u$ with a free parameter $N_0>0$, we define the one-parameter family of solutions $(u^{(\sigma)})_{\sigma\in[0,1]}$ as
\begin{align*}
    u^{(\sigma)}=\begin{cases}
        u|_{N_0=C/\sigma},&\sigma\in(0,1]\\
        U,&\sigma=0,
    \end{cases}
\end{align*}
where $C>0$ is chosen sufficiently large so that Proposition~\ref{w-exists-fixed-point-proposition} applies with $N_0=C$. It follows that, for $\sigma\neq0$, the solution $u^{(\sigma)}$ takes the form
\begin{align*}
    u^{(\sigma)}=U+v^{C/\sigma}+w^{C/\sigma}.
\end{align*}
When $\sigma\in C/\mathbb Z$, clearly the rescaling preserves the $2\pi$-periodicity, so $u^{(\sigma)}$ is well-defined on $\mathbb T^d$. Let us denote by $(\sigma_n)_{n\in\mathbb N}$ a decreasing sequence contained in $C/\mathbb Z\cap[0,1]$.

We claim that the family of solutions $u^{(\sigma_n)}$ converges to $u^{(0)}$ in $L_t^\infty([0,T]; BMO^{-1}_x(\mathbb T^d))$ with the weak-* $BMO^{-1}$ topology. Indeed, by the rescaling, $\hat v^{C/\sigma_n}+\hat w^{C/\sigma_n}$ is supported only in $(C/\sigma_n)\mathbb Z^d$. Since they are zero average, the support must be contained in $\{\xi\in\mathbb Z^d:|\xi|>C/\sigma_n\}$. Further, by criticality, we have
\begin{align*}
    \|u^{(\sigma_n)}-U\|_{L_t^\infty BMO^{-1}}&=\|v^{C/\sigma_n}+w^{C/\sigma_n}\|_{L_t^\infty BMO^{-1}}=\|v+w\|_{L_t^\infty BMO^{-1}}.
\end{align*}
Having shown inclusion in the Koch--Tataru space on the lifetime of the solution, the uniform bound in $BMO^{-1}$ follows from the theorem of \cite{hou2024regularity}. Thus we conclude that for any test function $\varphi$ in the predual, we have
\begin{align*}
    |\langle\varphi,u^{(\sigma_n)}-U\rangle|&=|\langle P_{>C/(2\sigma_n)}\varphi,P_{>C/(2\sigma_n)}(u^{(\sigma_n)}-U)\rangle|\\
    &\lesssim \|P_{>C/(2\sigma_n)}\varphi\|_{L_t^1\dot F^1_{1,2}}\|u^{(\sigma_n)}-U\|_{L_t^\infty BMO^{-1}}.
\end{align*}
The second factor is bounded as shown above, while the first is clearly $o(1)$ as $n\to\infty$ and $\sigma_n\to0$.
\qed

\appendix

\section{}

\begin{lemma}[Commutator estimate for $e^{t\Delta}$]\label{l:commutator}
    Let $a\in C^\infty(\mathbb T^d)$ and $\xi\in\mathbb Z^d$, and define
    \begin{align*}
        A_i\coloneqq |\xi|^{-i}\|\nabla^ia\|_{L^\infty}.
    \end{align*}
    Then
    \begin{equation}\begin{aligned}\label{commutator_inequality}
        &\|\nabla^n[e^{t\Delta},a(x)]\sin(\xi\cdot x)\|_{L^\infty}\\
        &\qquad\lesssim_{m,n} |\xi|^{n}\sum_{i=0}^n \left((A_i^{1-1/m}A_{m+i}^{1/m}+A_i^{1-2/m}A_{m+i}^{2/m})e^{-|\xi|^2t/4}+A_{m+i}\right)
        \end{aligned}\end{equation}
        for $m\geq3+n$. Furthermore,
        \begin{align}\label{heat-decay-estimate}
             \|\nabla^n e^{t\Delta}(a(x)\sin(\xi\cdot x))\|_\infty&\lesssim_{m,n} |\xi|^{n}\left(\sum_{i=0}^{n}A_i e^{-|\xi|^2t/4}+\sum_{i=m}^{n+m}A_{i}\right).
        \end{align}
\end{lemma}

\begin{proof}
    First we prove the special case of \eqref{commutator_inequality} for $n=0$; the general case then follows easily from the Leibniz rule. Once we establish \eqref{commutator_inequality}, it is straightforward to conclude \eqref{heat-decay-estimate}. Indeed,
    \begin{align*}
        \nabla^n e^{t\Delta}(a(x)\sin(\xi\cdot x))&=\nabla^n(a\sin(\xi\cdot x))e^{-|\xi|^2t}+\nabla^n[e^{t\Delta},a]\sin(\xi\cdot x).
    \end{align*}
    The first term can be bounded trivially by the Leibniz rule, For the second, we apply \eqref{commutator_inequality} and Young's inequality to the products of the form $A^{1-1/m}B^{1/m}$, etc.

    Toward \eqref{commutator_inequality}, we have the exact identity
    \begin{align*}
        [e^{t\Delta},a(x)]e^{ix\cdot n}&=\mathcal F^{-1}\left(e^{-t|\xi|^2}\mathcal F(a(x)e^{in\cdot x})(\xi)\right)-a(x)e^{t\Delta}(e^{in\cdot x})\\
        &=\mathcal F^{-1}\left(e^{-t|\xi|^2}\hat a(\xi-n)\right)-a(x)e^{-|n|^2t+in\cdot x}\\
        &=\sum_{\xi\in\mathbb Z^d}e^{ix\cdot\xi}(\hat a(\xi-n)e^{-t|\xi|^2}-\hat a(\xi)e^{-t|n|^2+in\cdot x})\\
        &=e^{ix\cdot n}\sum_{\xi\in\mathbb Z^d}e^{ix\cdot\xi}(e^{-t|\xi+n|^2}-e^{-t|n|^2})\hat a(\xi).
    \end{align*}
    Taking the imaginary part and defining the multiplier
    \begin{align*}
        m(\xi)=e^{-t|\xi+n|^2}-e^{-t|n|^2},
    \end{align*}
    we have
    \begin{align}
        [e^{t\Delta},a(x)]\sin(x\cdot n)&=\sin(x\cdot n)m(\nabla/i)a(x).\label{commutator-identity}
    \end{align}

    First we use Taylor's theorem to estimate
    \begin{align}
        |m(\xi)|&\lesssim t|\xi|(|\xi|+|n|)e^{-t\min(|\xi+n|,|n|)^2}.\label{m-bound-1}
    \end{align}
    When $t\leq 10|n|^{-2}$, we estimate \eqref{commutator-identity} using \eqref{m-bound-1}, bounding the exponential factor by $1$. This yields
    \begin{align*}
        \|[e^{t\Delta},a(x)]\sin(n\cdot x)\|_\infty&\lesssim\sum_{N\leq N_0}\|P_Nm(\nabla/i)a\|_\infty+\sum_{N> N_0}N^{-m}\|P_Nm(\nabla/i)\nabla^m a\|_\infty\\
        &\lesssim\sum_{N\leq N_0}tN(|n|+N)\|a\|_\infty+\sum_{N>N_0}tN^{-m+1}(|n|+N)\|\nabla^ma\|_\infty\\
        &\lesssim \Big(\frac{N_0}{|n|}+\left(\frac{N_0}{|n|}\right)^2\Big)\|a\|_\infty+\Big(\frac{1}{|n|N_0^{m-1}}+\frac1{|n|^2N_0^{m-2}}\Big)\|\nabla^ma\|_\infty
    \end{align*}
    by standard multiplier estimates. Setting $N_0=(\|\nabla^{m}a\|_\infty/\|a\|_\infty)^{1/m}$, we reach the claimed upper bound.

    If instead $t>10|n|^{-2}$, we decompose
    \begin{align*}
        \|[e^{t\Delta},a(x)]\sin(n\cdot x)\|_\infty&\leq I_1+I_2+I_3
    \end{align*}
    where $I_1,I_2,I_3$ are the projections to $N\leq (t|n|)^{-1}$, $(t|n|)^{-1}<N\leq|n|/10$, and $N>|n|/10$.
    
    In the low-frequency case, we have $|\xi|\leq 2N\leq2/(t|n|)\leq |n|/5$, so the exponential in \eqref{m-bound-1} is controlled by $e^{-t|n|^2/2}$. Combining with \eqref{commutator-identity}, we have
    \begin{align*}
        I_1&=\sum_{N\leq\min(N_0,(t|n|)^{-1})}\|P_Nm(\nabla/i)a\|_\infty\\
        &\qquad+\sum_{N_0<N\leq\max(N_0,(t|n|)^{-1})}N^{-m}\|P_Nm(\nabla/i)\nabla^ma\|_\infty\\
        &\lesssim \sum_{N\leq N_0}tN(|n|+N)\|a\|_\infty e^{-t|n|^2/2}+\sum_{N>N_0}tN^{-m+1}(|n|+N)\|\nabla^ma\|_\infty e^{-t|n|^2/2}\\
        &\lesssim t|n|^2\left(\frac{N_0}{|n|}(1+\frac{N_0}{|n|})\|a\|_\infty +|n|^{-1}N_0^{-m+1}(1+\frac{N_0}{|n|})\|\nabla^ma\|_\infty\right) e^{-t|n|^2/2}\\
        &\lesssim \left(\frac{\|a\|_\infty^{1-1/m}\|\nabla^ma\|_\infty^{1/m}}{|n|} + \frac{\|a\|_\infty^{1-2/m}\|\nabla^ma\|_\infty^{2/m}}{|n|^2}\right) e^{-t|n|^2/4}.
    \end{align*}
    
Next, we estimate
\begin{align}
    |m(\xi)|\lesssim e^{-t|n|^2/2}\label{m-bound-2}
\end{align}
for all $|\xi|\leq|n|/4$. In particular, when $N\leq|n|/10$, \eqref{m-bound-2} holds for all $\xi$ in the support of the symbol of $P_N$. It follows that the intermediate frequencies are bounded by 
    \begin{align*}
        I_2&\lesssim\sum_{(t|n|)^{-1}<N\leq|n|/10}N^{-1}\|P_Nm(\nabla/i)\nabla a\|_\infty\\
        &\lesssim \sum_{N>(t|n|)^{-1}}N^{-1}\|\nabla a\|_\infty e^{-t|n|^2}\\
        &\lesssim t|n|\|\nabla a\|_\infty e^{-t|n|^2}\\
        &\lesssim |n|^{-1}\|a\|_\infty^{1-1/m}\|\nabla^ma\|_\infty^{1/m} e^{-t|n|^2/4}.
    \end{align*}
Finally, for the high frequencies, we use the trivial bound $|m(\xi)|\leq1$ to find
    \begin{align*}
        I_3&=\sum_{N>|n|/10}N^{-m}\|P_Nm(\nabla/i)\nabla^{m} a\|_\infty\lesssim |n|^{-m}\|\nabla^{m} a\|_\infty.
    \end{align*}
    
\end{proof}

\begin{lemma}\label{l:heat_stationary_phase}
Let $\xi\in\mathbb Z^d$, $a\in C^\infty(\mathbb T^d; \mathbb R^d)$, and $A_i$ be as in Lemma~\ref{l:commutator}. Additionally define $A_{i,\kappa}\coloneqq|\xi|^{-i-\kappa}\|\nabla^ia\|_{C^\kappa}$ for $\kappa\in(0,1)$. Then
\begin{align*}
\|\nabla^ne^{t\Delta}\mathcal R\mathbb P(ae^{i\xi\cdot x})\|_{C^\kappa}&\lesssim |\xi|^{n+\kappa}\left(|\xi|^{-1}\sum_{i=0}^{m-1}A_{i,\kappa}e^{-|\xi|^2t/4}+\sum_{i=m}^{n+2m}
A_{i,\kappa}\right).
\end{align*}
\end{lemma}

\begin{proof}

The Leray projection is well-known to be bounded on $C^\kappa$ so it can be removed immediately. For the rest, we recall the following identity from \cite{DeLellisSzekelyhidi2013}*{Prop.~5.2}:
\begin{align*}
    e^{t\Delta}(a(x)e^{i\xi\cdot x})&=-i\sum_{j=0}^{m-1}\div\Big(\frac\xi{|\xi|^2}F_j(x)\Big)+F_m(x)
\end{align*}
where
\begin{align*}
F_j(x)&=e^{t\Delta}(\left(\frac{i\xi}{|\xi|^2}\cdot\nabla\right)^ja(x)e^{i\xi\cdot x}).
\end{align*}
Since we can write $\mathcal Rf=\Delta^{-1}(2\nabla\odot f-(\div f)\Id)$, it suffices to show the estimate with $\mathcal R$ replaced by $\Delta^{-1}\nabla$. By Schauder estimates,
\begin{align*}
    \|e^{t\Delta}\Delta^{-1}\nabla(ae^{i\xi\cdot x})\|_{C^{n,\kappa}}&\lesssim \sum_{j=0}^{m-1}|\xi|^{-1}\|F_j\|_{C^{n,\kappa}}+\|F_m\|_{C^{n,\kappa}}+\left|\fint ae^{i \xi\cdot x}\right|.
\end{align*}
Clearly the last term is $O(|\xi|^{-m}\|\nabla^ma\|_{L^\infty})$. For the others, by \eqref{heat-decay-estimate} and interpolation, we have
\begin{align*}
\|F_j\|_{C^{n,\kappa}}&\lesssim |\xi|^{n}\Big(\sum_{i=0}^nA_{i+j,\kappa}e^{-|\xi|^2t/4}+\sum_{i=m}^{n+m}
A_{i+j,\kappa}\Big)
\end{align*}
Combining these estimates and removing the redundant contributions yields the claim.
\end{proof}

\begin{lemma}[Oscillation estimate]\label{l:oscillation_estimate}
Let $a\in C^\infty(\mathbb T^d;\mathbb R^d)$ and $b\in C^\infty(\mathbb T^d;\mathbb R)$, where $\mathbb T^d$ is the $2\pi$-periodic torus. Let $\lambda\in\mathbb N$ be large and $\Xi\in \mathbb Z^d$ with $|\Xi|\in\lambda\mathbb N$. Then, for any $n\geq0$, $m\geq1$, and $\kappa\in(0,1)$, we obtain 
\begin{align*}
&\Big\|\nabla^ne^{t\Delta}\mathcal R\mathbb P\Big(a(x)\Big(b(\lambda x)\sin^2(\Xi\cdot x)-\fint_{\mathbb T^d}b(\lambda y)\sin^2(\Xi\cdot y)dy\Big)\Big)\Big\|_{C^\kappa}\\
&\qquad\lesssim_{m,n,\kappa} \Bigg(\sum_{i=0}^{m-1}(\lambda^{n-i-1}+|\Xi|^{n-i-1})\|a\|_{C^{i,\kappa}}e^{-\lambda^2t/4}\\
&\qquad\qquad\qquad\qquad+\sum_{i=m}^{n+2m}
(\lambda^{n-i}+|\Xi|^{n-i})\|a\|_{C^{i,\kappa}}\Bigg)\|\nabla^{n+2d}b\|_{L^\infty}
\end{align*}
\end{lemma}

\begin{proof}
    Define $f(x)=b(x)\sin^2(\Xi\cdot x/\lambda)$. Clearly $f$ is periodic and has a Fourier series given by
    \begin{align*}
        \hat f(\xi)&=\mathcal F\Big(b(x)(\frac12-\frac14e^{2i\Xi\cdot x/\lambda}-\frac14e^{-2i\Xi\cdot x/\lambda})\Big)(\xi)\\
        &=\frac12\hat b(\xi)-\frac14\hat b(\xi-2\Xi/\lambda)-\frac14\hat b(\xi+2\Xi/\lambda)
    \end{align*}
    for $\xi\in\mathbb Z^d$.

From here we estimate using Lemma~\ref{l:heat_stationary_phase},
\begin{align*}
&\Big\|\hat f(\xi)\nabla^ne^{t\Delta}\mathcal R\mathbb P\Big(a(x)e^{i\lambda x\cdot \xi}\Big)\|_{C^\kappa}\\
&\quad\lesssim\Big(|\hat b(\xi)|+|\hat b(\xi-2\Xi/\lambda)|+|\hat b(\xi+2\Xi/\lambda)|\Big)\\
&\quad\qquad\times|\lambda\xi|^{n+\kappa}\Big(|\lambda\xi|^{-1}\sum_{i=0}^{m-1}A_{i,\kappa}e^{-|\lambda\xi|^2t/4}+\sum_{i=m}^{n+2m}
A_{i,\kappa}\Big)\\
&\quad\lesssim\Big(\langle\xi\rangle^{-n-2d}+\langle\xi-2\Xi/\lambda\rangle^{-n-2d}+\langle\xi+2\Xi/\lambda\rangle^{-n-2d}\Big)\|\nabla^{n+2d}b\|_{L^\infty}\\
&\quad\qquad\times\Big(\sum_{i=0}^{m-1}|\lambda\xi|^{n-i-1}\|a\|_{C^{i,\kappa}}e^{-|\lambda\xi|^2t/4}+\sum_{i=m}^{n+2m}
|\lambda\xi|^{n-i}\|a\|_{C^{i,\kappa}}\Big).
\end{align*}
After bounding $e^{-|\lambda\xi|^2t/4}$ by $e^{-\lambda^2t/4}$, we sum over $\xi\in\mathbb Z^d\setminus0$, which yields the claim.
\end{proof}



\bibliographystyle{amsrefs}
\bibliography{NSE}

\end{document}